\documentclass[a4paper]{amsart}
\usepackage[english]{babel} % language
\usepackage[utf8]{inputenc}
\usepackage[T1]{fontenc}
\usepackage{lmodern}
\usepackage{amsmath,amsfonts,amssymb,amsthm}
\usepackage{mathrsfs,MnSymbol}
\usepackage{xfrac}
\usepackage[longnamesfirst,numbers]{natbib}
\usepackage[margin]{fixme}
\usepackage{setspace}
\usepackage{enumitem}         % better lists
\usepackage{stmaryrd} % stochastic intervals
\usepackage{lscape}
\usepackage{color}
\definecolor{darkgreen}{rgb}{0,0.55,0}
\usepackage{hyperref}
\hypersetup{pdfstartview={Fit},    % fits the width of the page to the window
  pdftitle={Fwd Invariance and WZa for a Stochastic Moving Boundary Problem},    % title
  pdfauthor={Martin Keller-Ressel, Marvin S. Mueller},     % author
  pdfsubject={},   % subject of the document            
  pagebackref=true,
  colorlinks=true,
  linkcolor=blue,
  citecolor=darkgreen}

\usepackage{aliascnt}

\newcommand{\newautotheorem}[3]{
  \newaliascnt{#1}{#2}  
  \newtheorem{#1}[#1]{#3}
  \aliascntresetthe{#1}
}

\numberwithin{equation}{section}

\renewcommand\*{\cdot}
\newcommand\1{\mathbf{1}}

\newcommand\R{\mathbb{R}}
\newcommand\C{\mathbb{C}}
\newcommand\N{\mathbb{N}}

\newcommand\Q{\mathbb{Q}}

\newcommand{\supp}{\operatorname{supp}}

\providecommand{\setc}[2]{\left\{ #1 \cond #2\right\}}
\providecommand{\abs}[1]{\left\lvert#1\right\rvert}
\providecommand{\norm}[2]{\left\lVert#1\right\rVert_{#2}}
\newcommand{\scal}[3]{\left\langle #1, #2\right\rangle_{#3}}
\newcommand{\innerprod}[3]{\left\langle #1, #2\right\rangle_{#3}}
\newcommand{\dist}{\operatorname{dist}}

\newcommand{\PP}{\mathbb{P}}

\newcommand{\MM}{\mathbb{M}}
\newcommand{\EE}{\mathbb{E}}

\newcommand{\cA}{\mathcal{A}}
\newcommand{\cB}{\mathcal{B}}

\newcommand{\cO}{\mathcal{O}}

\newcommand{\cC}{\mathcal{C}}

\newcommand\ddt{\tfrac{\partial}{\partial t}}

\newcommand\ddx{\tfrac{\partial}{\partial x}}
\newcommand\ddxx{\tfrac{\partial^2}{\partial x^2}}
\newcommand\ddy{\tfrac{\partial}{\partial y}}

\newcommand\ddz{\tfrac{\partial}{\partial z}}

\newcommand\linspan[1]{\operatorname{span}\left\{#1\right\}}

\newcommand\mHC{\cdot}
\newcommand\F{\mathcal{F}}

\renewcommand\d{\,\operatorname{d}\hspace{-0.05cm}}

\providecommand\cond{\,\middle\vert\,}
\renewcommand\cond{\,\middle\vert\,}

\renewcommand{\P}[1]{\mathbb{P}\left[#1\right]}                     % probability
\newcommand{\E}[1]{\mathbb{E}\left[#1\right]}                     % expectation
\newcommand{\llbrak}{\llbracket}
\newcommand{\rrbrak}{\rrbracket}

\newcommand\Id{\operatorname{Id}}                               % Identity Operator 
\newcommand\dom{\mathcal{D}}                                      % Domain

\newcommand\DA{\dom(\cA)}

\renewcommand\H{\mathfrak H}
\renewcommand\L{\mathfrak L}

\newcommand\cI{\mathcal I}
\newcommand{\HS}{\mathscr{L}_2}
\newcommand{\Lbd}{\mathscr{L}}
\newcommand\tmm{[t]_m}
\newcommand\tmp{[t+\delta_m]_m}
\title[Approximation and Invariance for SMBPs]{Forward-Invariance and
  Wong-Zakai Approximation for Stochastic Moving Boundary Problems}

\author{Martin Keller-Ressel}
\author{Marvin S. M\"uller}
\address[M. ~Keller-Ressel]{Department of Mathematical Stochastic, TU Dresden, Germany}
\email[M. ~Keller-Ressel]{martin.keller-ressel@tu-dresden.de}
\address[M. S. ~M\"uller]{Department of Mathematics, ETH Z\"urich, Switzerland}
\email[M.\,S. ~M\"uller]{marvin.mueller@math.ethz.ch}

\thanks{Both authors acknowledge support by the German Research Foundation (DFG) under the grant ZUK 64. MM also acknowledges support by the Swiss National Science Foundation through grant SNF $205121\_163425$. Most of the work of MM was carried out within the scope of his dissertation~\cite{diss}.}


\date{\today}
\begin{document}
\begin{abstract}
  We discuss a class of stochastic second-order PDEs in one space-dimension with an inner boundary moving according to a possibly non-linear, Stefan-type condition. We show that proper separation of phases is attained, i.e., the solution remains negative on one side and positive on the other side of the moving interface, when started with the appropriate initial conditions. To extend results from deterministic settings to the stochastic case, we establish a Wong-Zakai type approximation. After a coordinate transformation the problems are reformulated and analysed in terms of stochastic evolution equations on domains of fractional powers of linear operators.
\end{abstract}
%------------------------------------------------------------------------------
%theorems
\theoremstyle{plain}
\newtheorem{thm}{Theorem}[section]
\providecommand*{\thmautorefname}{Theorem} 
\newautotheorem{lem}{thm}{Lemma}
\providecommand*{\lemautorefname}{Lemma} 
\newautotheorem{prop}{thm}{Proposition}
\providecommand*{\propautorefname}{Proposition} 
\newautotheorem{cor}{thm}{Corollary}
\providecommand*{\corautorefname}{Corollary} 
\newautotheorem{hyp}{thm}{Assumption}
\providecommand*{\hypautorefname}{Assumption} 
\newtheorem*{prob}{Problem}
\theoremstyle{remark}
\newtheorem{step}{Schritt}
\newtheorem*{idea}{Idea}
\newautotheorem{rmk}{thm}{Remark}
\providecommand*{\rmkautorefname}{Remark}
\newtheorem*{notation}{Notation}
\newtheorem*{question}{Question}
\newtheorem{rmkc}{Remark}[section]
\providecommand*{\rmkcautorefname}{Remark}

\theoremstyle{definition}
\newautotheorem{defn}{thm}{Definition}
\providecommand*{\defnautorefname}{Definition} 
\newtheorem*{claim}{Claim}
\newautotheorem{ex}{thm}{Example}
\providecommand*{\exautorefname}{Example}
\newautotheorem{example}{thm}{Example}
\providecommand*{\exampleautorefname}{Example}
\newautotheorem{ass}{thm}{Assumption}
\providecommand\assautorefname{Assumption}
\newautotheorem{exc}{rmkc}{Example}
\providecommand*{\excautorefname}{Example}

\theoremstyle{examples}
\newtheorem*{exs}{Examples}

%----------------------------------------------------------------------------
%Section etc start with capital letter
\providecommand*{\sectionautorefname}{Section}
\renewcommand{\sectionautorefname}{Section}
\providecommand*{\subsectionautorefname}{Subsection}
\renewcommand{\subsectionautorefname}{Subsection}
\providecommand*{\appendixautorefname}{Appendix}
\renewcommand{\appendixautorefname}{Appendix}

%------------------------------------------------------------------------------

\setcounter{tocdepth}{2}
\maketitle
\tableofcontents

\subsection*{Introduction}
Moving boundary problems allow for modeling of multi-phase systems with separating boundaries evolving in time. The classical model for temperature evolution in a system of water and ice is the so called Stefan problem~\cite{stefanEis},
\begin{align*}
  \d v(t,x) &= \eta_+ \Delta v(t,x)\d t,\qquad x>x_*(t),\\
  \d v(t,x) &= \eta_- \Delta v(t,x)\d t,\qquad x<x_*(t),\\
  \d x_*(t) &= \rho\*\left(\nabla v(t,x_*(t)-) - \nabla v(t,x_*(t)+)\right)\d t,\\
  v(t,x_*(t)) &= 0.
\end{align*}
In this work, we will study stochastic and semilinear extension of the
Stefan problem in two directions. On one hand, we show a Wong-Zakai
type approximation result for stochastic moving boundary
problems. This gives an understanding of the stochastic problems with deterministic
extensions of the Stefan problem, which have been widely studied in
the second half of the 20th century. On the other hand, note that
proper separation of the two phases is attained only if the solution
remains negative on one side, and positive on the other side of the
moving interface, i.\,e.
\[ v(t,x) \geq 0,\quad \text{ if } x>x_*(t),\quad\text{and}\quad
  v(t,x)\leq 0,\quad \text{if }\quad x< x_*(t).\]
Using the Wong-Zakai-type approximation we show that under reasonable
``inward pointing drift'' and ``parallel to the boundary diffusion''
conditions on the coefficients, separation of phases is indeed
maintained for the solutions. 

The analysis builds on a framework of stochastic evolution equations
on domains of fractional power, as set up in~\cite{SFBPDir,
  SFBP1stOrder} to prove existence and uniqueness for semilinear stochastic
moving boundary problems.

More detailed, by a change of coordinates, these problems are linked with the so called forward invariance of closed sets for stochastic evolution equations,
\begin{equation}
  \d X_t = \left[A X_t + B(X_t)\right] \d t + C(X_t) \d W_t,\quad t\geq 0, \label{eq:SEEa}
\end{equation}
and their mild formulation
\begin{equation}
  X_t = X_0 + \int_0^t e^{(t-s)A} B(X_s) \d s + \int_0^t e^{(t-s)A}C(X_s) \d W_s,\quad t\geq 0, \label{eq:SEEa}
\end{equation}
on a separable Hilbert space $E$. Here, $W$ is a cylindrical Wiener
process on another separable Hilbert space $U$, with coefficients
$A:\dom(A)\to E$, a generator of a $C_0$ semigroup $(e^{tA})_t$ on
$E$, $B: \dom(B)\to E$ and $C: E\to \HS(U;E)$ . A subset $\MM\subset
E$ is called \emph{forward invariant} for~\eqref{eq:SEEa}, if for
every local solution $(X,\tau)$ of~\eqref{eq:SEEa} with initial data
$X_0\in \MM$ it holds that $X_t \in \MM$ on the stochastic interaval 
$\llbrak 0,\tau\llbrak$. A priori, we should assume that $\MM$ is invariant under $(S_t):=
(e^{tA})$, that is, $S_t\MM\subset \MM$ for all $t\geq 0$. 

A typical application is when~\eqref{eq:SEEa} describes an SPDE on $E= L^2$ and $\MM := L^2_+$ is the closed convex cone of non-negative functions in $L^2$. The main motivation in the literature for forward invariance in the framework of mild solutions seems to come from the question of positivity of solutions for HJM interest rate models; see for example~\cite{filipovicHJM, nakayamaViab, zabczykFwdInvariance}. 

Milian~\cite{milianComparison} used Yosida approximations to extend inward pointing and parallel to the boundary conditions from finite dimensional equations to prove a comparison result for stochastic evolution equations, under Lipschitz conditions of $B$ and $C$ on $E$. These results have been extended by Filipovic et al.~\cite{filipovicHJM}, to show positivity for HJM equations provided that point-wise versions of inward pointing and parallel to the boundary conditions are satisfied. 

Forward-invariance for deterministic evolution equations $(C = 0)$ was extensively studied in the 70s and 80s.
For mild and strong solutions of deterministic evolution equations, Pavel~\cite{pavelInvariant} and Jachimiak~\cite{jachimiakInvariance} have shown that under Lipschitz assumptions on $B:E\to E$, forward-invariance is equivalent to the Nagumo condition in the form
\begin{equation}
  \label{eq:nagumomild}
  \dist_E(S_{\epsilon}u_0 + \epsilon B(u_0); \MM) = o(\epsilon),\qquad \text{as } \epsilon \searrow 0.  
\end{equation}

Zabczyk~\cite{zabczykFwdInvariance} extended this result to stochastic
evolution equations additive noise. For multiplicative noise,
Nakayama~\cite{nakayamaViab} used a support theorem to extend the
Nagumo condition~\eqref{eq:nagumomild} to stochastic evolution
equations. Following this approach, we extend the Wong-Zakai
approximation theorem in~\cite{nakayamasupport} in two directions. On
one hand, to the situation when $-A$ is generator of an analytic
semigroup of negative type, but $B:\dom((-A)^\alpha) \to E$ for some
$\alpha \in [0,1)$ and $C: \dom((-A)^\alpha) \to
\HS(U;\dom((-A)^\alpha))$ and, on the other hand, that the
coefficients need to be Lipschitz continuous only on bounded sets and
the solution might explode in
finite time. For an overview on Wong-Zakai approximations in infinite dimensions, see also~\cite{twardowskaSurvey}, \cite{zabzcykWZA} and references therein. 

For the corresponding deterministic equations, general existence and
invariance results are given in~\cite{amannInvariant} for compact
semigroups. Since we are interested in SPDEs on  unbounded domains, we
do not have compact semigroups and will make use of the very general
result in~\cite{pruessInvariant} to show that the Nagumo condition is
a sufficient criterion for the forward invariance of this class of
stochastic evolution equation. Applying the results to a class of
stochastic moving boundary problems, we derive sufficient point-wise
criteria on the coefficients.

\subsection*{Notation}
For a stopping time $\tau$ we denote the closed stochastic interval by
$\llbrak 0,\tau \rrbrak := \{(t,\omega)\in [0,\infty)\times
\Omega\,\vert\, t\leq \tau(\omega)\}$. Respectively, we define
$\llbrak 0,\tau\llbrak$, $\rrbrak 0,\tau\llbrak$ and $\rrbrak
0,\tau\rrbrak$. For stochastic processes $X$ and $Y$ we say $X(t) =
Y(t)$ on $\llbrak 0,\tau\llbrak$, if equality holds for almost all
$\omega \in \Omega$ and all $t\geq 0$ such that $(t,\omega)\in \llbrak
0,\tau\llbrak$. Given Hilbert spaces $E$ and $H$, we write
$E\hookrightarrow H$ when $E$ is continuously and densely embedded
into $H$. As usual, we denote by $L^q$ the Lebesgue space, $q\geq 1$,
and with $H^s$, $s>0$, the Sobolev spaces of order $s>0$, for $k\in
\N$, $C^k$ will be the space of $k$-times continuously differentiable
functions, $C^k_b$ the subspace of $C^k$ where the elements and all derivatives up to
order $k$ are bounded and $BUC^k$ will be the subspace of all elements
which together with their derivatives up to order $k$ are bounded and
uniformly continuous. Moreover, for separable Hilbert spaces $U$ and $E$, $\Lbd(U,E)$ is the space of
linear continuous operators from $U$ to $E$ and $\HS(U;E)$ is the space of Hilbert-Schmidt operators from $U$ into $E$. The scalar product on $E$ will be denoted by $\scal{.}{.}{E}$. We will work only with real separable Hilbert spaces and implicitely use their complexification when necessary to apply results from the literature.

%%% Local Variables: 
%%% mode: latex
%%% TeX-master: "0paper"
%%% End: 

\section[Phase Separation and Approximation]{Phase Separation and Approximation for SMBPs}
\label{sec:positivityresults}

We work on a filtered probability space $(\Omega, \F, (\F_t), \PP)$
with the usual conditions on which a $\Id_U$-cylindrical Wiener
process $W$ taking values in the Hilbert space $U=L^2(\R)$ lives. For a CONS $(e_k)$ of $U$ and a series $(\beta_k)$ of independent real Brownian motions can represent $W$ by
\[ W_t = \sum_{k=1}^\infty e_k \beta_k(t),\quad t\in [0,T],\]
for a finite time horizon $T>0$. We denote by $\xi$ the spatially colored noise
\begin{equation}
  \label{eq:colorednoise}
  \xi_t(x) := T_\zeta W_t (x),\qquad T_\zeta w(x) := \int_\R \zeta(x,y) w(y) \d y,\quad x\in \R,
\end{equation}
for some integral kernel $\zeta: \R^2\rightarrow \R$. Recall that the
mapping $w\mapsto T_\zeta w(x)$ is Hilbert-Schmidt from $L^2(\R)$ into
$\R$ and $(\xi_t(x))_{t\geq 0}$ is a real Brownian motion for each
$x\in \R$. 

In that setting, consider the following class of stochastic 2-phase systems in one space dimension, 
\begin{equation}
  \left\{\begin{aligned}
      \d v(t,x) &= \left[\eta_+ \ddxx v +  \mu_+\left(x-x_*(t), v, \ddx v \right) \right] \d t\\
      & \qquad \qquad\qquad\qquad\qquad+ \sigma_+\left(x-x_*(t), v\right)\d \xi_t (x), \quad  x > x_*(t),\\
      \d v(t,x) &= \left[\eta_- \ddxx v -\mu_-\left(x-x_*(t), v, \ddx v \right) \right] \d t\\
      &\qquad\qquad\qquad\qquad\qquad  -\sigma_-\left(x-x_*(t), v\right) \d \xi_t(x), \quad  x < x_*(t),\\
      v(0,x) &= v_0(x),
    \end{aligned}\right.\label{eq:mbp}
\end{equation}
with inner boundary conditions
\begin{gather}
  \left\{\begin{aligned}
    \ddx v(t,x_*(t)+) &= \kappa_+  v(t,x_*(t)+), \\
    \ddx v(t,x_*(t)-) &= - \kappa_- v(t,x_*(t)-), 
  \end{aligned}\right.\label{eq:bcx}
\end{gather}
for $\kappa_+$, $\kappa_-\in (0,\infty)$ or $\kappa_+ = \kappa_-
=\infty$ and dynamics of the interface $x_*$ governed by
\begin{equation}
  \label{eq:interface_dirichlet}
  \left\{\begin{aligned}
  \ddt x_*(t) &= \varrho\left(\ddx v(t, x_*(t)+),\ddx v(t,x_*(t)-)\right),\\
  x_*(0) &= x_0,    
  \end{aligned}\right.
\end{equation}
when $\kappa_+=\kappa_-=\infty$ or, else,
\begin{equation}
  \label{eq:interface_1st}
  \left\{\begin{aligned}
  \ddt x_*(t) &= \varrho\left(v(t, x_*(t)+), v(t,x_*(t)-)\right),\\
  x_*(0) &= x_0,    
  \end{aligned}\right.
\end{equation}
where $t\in [0,T]$, $T\in (0,\infty)$, and $\mu_+$, $\mu_-: \R^3\rightarrow \R$, $\sigma_+$, $\sigma_-: \R^2\rightarrow \R$, $\varrho:\R^2\to \R$, and $\eta_+$, $\eta_->0$.

Here, the case where $\kappa_+ = \kappa_- = \infty$ is
interpreted as imposing Dirichlet boundary conditions on $v$. For this
case, existence and uniqueness of solutions in an analytically strong
framework have been shown in \cite{SFBPDir}. Neumann or Robin boundary
conditions, corresponding to $\kappa_+$, $\kappa_- <\infty$, were
investigated in \cite{SFBP1stOrder}. We will refer to these as "first
order" boundary conditions. In both cases, under sufficient
assumptions on the coefficients and initial data, there exists a
solution, in the sense that there exists a maximal predictable strictly positive
stopping time $\tau\leq T$, an $L^2(\R)\otimes \R$ predictable stochastic
process $(v,x_*)$ taking values in 
\[\bigcup_{x\in \R} \left( \Gamma(x)  \times \{x\}\right),\]
where
\[\Gamma(x) := \{ v\in H^2(\R\setminus\{x\}) \,\vert\, \ddx v(x+) =
  \kappa_+ v(x),\quad \ddx v(x-) = - \kappa_- v(x-)\}.\]
and a such that for all $\phi \in
C_0^\infty(\R)$, on $\llbrak 0,\tau\llbrak$, 
\begin{multline}
  \innerprod{v(t,.) - v_0}{\phi}{} = \int_0^t \innerprod{\bar\mu(., v(s,.), \nabla
    v(s,.), \Delta v(s,.), x_*(s))}{\phi}{} \d s\\
  + \int_0^t \innerprod{\bar \sigma(.,v(s,.))}{\phi}{} \d \xi_s \\
  + \int_0^t (v(s,x_*(s)-) - v(s,x_*(s)+)) \phi(x_*(s)) \d x_*(s),
\end{multline}
and either~\eqref{eq:interface_dirichlet} or~\eqref{eq:interface_1st}
being satisfied. Moreover, uniqueness holds true under sufficient regularity
constraints on the solution, cf. \cite[Theorem 2.11]{SFBPDir}, \cite[Theorem 1.15]{SFBP1stOrder}.

This describes a two-phase system with phase
change at the moving interface $x_*$, which itself is driven by local
imbalances of the system between both phases. In applications as
the two-phase Stefan problem or modeling of limit order books or
demand and supply in financial mathematics, cf. \cite{SFBP1stOrder}, one often expects that a proper separation of both phases is preserved, i.\,e. that $\d x$-a.\,e.
\begin{equation}
  \label{eq:vgleq0}
  v(t,x)\leq 0, \quad x<x_*(t)\quad \text{and}\quad v(t,x)\geq 0, \quad x>x_*(t),
\end{equation}
holds on $\llbrak 0,\tau\llbrak$, provided that it holds true for $t= 0$. 

In relative coordinates, namely
\begin{equation}
  \label{eq:trafo}
  u_1(t,x) := v(t,x_*(t)+x),\quad \text{and}\quad u_2(t,x) := v(t,x_*(t)-x),\quad x>0,
\end{equation}
the moving boundary problem becomes the coupled system of stochastic equations on $\R_+$,
\begin{gather}
  \begin{split}
    \d u_1(t,x) &= \left[\eta_+ \ddxx u_1 +  \mu_+\left(x, u_1, \ddx u_1 \right)+\ddt x_*(t)\ddx u_1(t,x)  \right] \d t,\\
    & \qquad \qquad\qquad\qquad\qquad+ \sigma_+\left(x, u_1\right)\d \xi_t (x_*(t)+x), \\
    \d u_2(t,x) &= \left[\eta_- \ddxx u  + \mu_-\left(-x, u_2, -\ddx u_2 \right) - \ddt x_*(t)\ddx u_2(t,x) \right] \d t,\\
    &\qquad\qquad\qquad\qquad\qquad + \sigma_-\left(-x, u_2\right) \d \xi_t(x_*(t)-x), \\
    u_1(0,x)&= u_{1,0}(x),\qquad u_2(0,x) = u_{2,0}(x),\qquad x_*(0) = x_0,
  \end{split}\label{eq:SPDE}
\end{gather}
with boundary conditions at $0$, for $t\in (0,T]$,
\begin{gather}
    u_1(t,0) = \kappa_+ \ddx u_1(t,0), \qquad
    u_2(t,0) = \kappa_- \ddx u_2(t,0).
    \label{eq:bc}
\end{gather}
Here,
\begin{equation}
  \label{eq:shiftedinit}
  u_{1,0}(x) = v_0(x_0+x),\qquad u_{2,0}(x_0-x).
\end{equation}
The interface conditions becomes 
\begin{equation}
  \label{eq:dxDirichlet}
  \ddt x_*(t) = \varrho\left(\ddx u_1(t, 0+), \ddx u_2(t,0+)\right),
\end{equation}
and else, when $\kappa_+$, $\kappa_- <\infty$, 
\begin{equation}
  \label{eq:dxNeumann}
  \ddt x_*(t) = \varrho\left(u_1(t, 0+), u_2(t,0+)\right).
\end{equation}

Let us shortly summarize the existence results for the centered
equations which we derived in \cite{SFBPDir} and \cite{SFBP1stOrder},
respectively. In both cases, there exist a unique maximal strong
solution $(u_1, u_2, x_*)$, up to a predictable stopping time
$\tau_*$, such that~\eqref{eq:SPDE} is satisfied in the sense of
$L^2(\R_+)$ integral equations and the boundary
conditions~\eqref{eq:bc} and either~\eqref{eq:dxDirichlet}
or~\eqref{eq:dxNeumann} hold true $\d t\otimes \PP$ almost everywhere,
on $\llbrak 0,\tau \llbrak$. Dirichlet boundary conditions, $u_1$,
$u_2$ take values in $C([0,\tau); H^2\cap H_0^1(\R_+))$ provided that
$u_{1,0}$, $u_{2,0} \in H^2\cap H^1_0(\R_+)$. For first order boundary
conditions and provided that $u_{1,0}$, $u_{2,0}\in H^1(\R_+)$, we get
a unique strong solution with almost surely $u_1$, $u_2 \in
C([0,\tau);H^1(\R_+)\cap L^2([0,\tau); H^2(\R_+))$. and $u_1$, $u_2$
fulfill~\eqref{eq:bc}. Moreover, note that the moving boundary problem~\eqref{eq:mbp} can be characterized completely by the centered equations~\eqref{eq:SPDE}.

Translated into the notion of the centered equations~\eqref{eq:SPDE}, the condition for phase-separation \eqref{eq:vgleq0} becomes 
\begin{equation}
  \label{eq:ugleq0}
  u_1(t,x) \geq 0,\quad\text{and}\quad u_2(t,x)\leq 0,\quad \text{for almost all }x\in \R_+. 
\end{equation}

A well-known criterion also from theory of finite dimensional equations are the so called inward-pointing-drift and parallel-to-the-boundary-diffusion conditions. Formulated point-wise, they read as follows.
\begin{ass}
  \label{a:inpoint}
  For all $x\geq 0$ it holds that, 
  \[  \mu_+(x,0,0) \geq 0,\qquad \mu_-(x,0,0) \leq 0,\qquad\text{and}\qquad \sigma_+(x,0) = \sigma_-(x,0)=0.\]
\end{ass}

To formulate the result for~\eqref{eq:SPDE}, we introduce the following closed convex cone,
\[ \MM := \left\{ (u_1,u_2,x)\in L^2\times L^2\times \R \,\vert\, u_1 \geq 0,\, u_2\leq 0, \; \text{$\d x$-a.\,e.}\right\}.\]
Phase separation in the sense of~\eqref{eq:ugleq0} is now
equivalent to so called forward invariance of $\MM$, supposed that the
coefficients and initial data are sufficiently regular, see
assumptions below.

\begin{thm}[Forward Invariance]
  \label{thm:fwdinv}
  Assume that \autoref{a:inpoint} and one of the following hold true.
  \begin{enumerate}[label=(\alph*)]
  \item $\kappa_+,\kappa_- <\infty$ and Assumptions~\ref{a:rho},
    \ref{a:zeta}, \ref{a:mu1}, \ref{a:sigma1} and~\ref{a:init1}, or, 
  \item Dirichlet boundary conditions at $0$ and,
    Assumptions~\ref{a:rho}, \ref{a:zeta}, \ref{a:mu}, \ref{a:sigma} and~\ref{a:init}
  \end{enumerate}
  Then, the set $\MM$ is forward invariant for~\eqref{eq:SPDE} in the sense that $(u_{0,1},u_{0,2},x_0)\in \MM$ yields that $(u_1(t,.),u_2(t,.),x_*(t))\in \MM$ on $\llbrak 0,\tau\llbrak$.
\end{thm}

\begin{cor}[Phase separation]
  Let the assumptions of \autoref{thm:fwdinv} hold true and assume
  that $\d x$-almost everywhere,
  \begin{equation}
    \label{eq:2}
    v_0(x) \leq 0,\quad x<x_0,\qquad v_0(x) \geq 0,\quad x>x_0.
  \end{equation}
  Then, on $\llbrak 0,\tau \llbrak$,
  \begin{equation}
    \label{eq:2}
    v_t(x) \leq 0,\quad x<x_*(t),\qquad v_t(x) \geq 0,\quad x>x_*(t).
  \end{equation}  
\end{cor}

We now list the assumptions.

\begin{ass}\label{a:rho}
  $\varrho:\R^2\to\R$ is locally Lipschitz continuous. More precisely, for all $N\in \N$ there exists an $L_{\varrho,N}$ such that for all $y$, $\tilde y\in \R^2$ with $\abs{y}$, $\abs{\tilde y}\leq N$ holds
  \begin{equation*}
    \abs{\varrho(y) - \varrho(\tilde y)} \leq L_{\varrho,N} \abs{y -\tilde y}.
  \end{equation*}
\end{ass}

\begin{ass}
  \label{a:zeta}
  $\zeta(.,y) \in C^4(\R)$ for all $y\in \R$ and $\tfrac{\partial^{i}}{\partial x^i}\zeta(x,.)\in L^2(\R)$ for all $x\in \R$, $i\in\{0,1,...,4\}$. Moreover, 
  \begin{equation}
    \sup_{x\in \R} \norm{\tfrac{\partial^{i}}{\partial x^i} \zeta(x,.)}{L^2(\R)} <\infty,\quad i=0,1,...,4.\label{eq:Azeta}
  \end{equation}
  For the remainder of this paper, we use the notation $\zeta^{(i)}:=\tfrac{\partial^{i}}{\partial x^i} \zeta$.
\end{ass}
\begin{rmk}
  When $\kappa_+$, $\kappa_- <\infty$, it suffices to assume that~\ref{a:zeta} holds for $i\in \{0,1,2,3\}$.
\end{rmk}
\begin{example}[Convolution]\label{ref:conv}
  Let $\zeta$ be a convolution kernel, i.\,e. $\zeta(x,y) := \zeta(x-y)$, $x$, $y\in \R$. If $\zeta\in C^{\infty}(\R)\cap H^4(\R)$, 
  then Assumption~\ref{a:zeta} is satisfied. In this case, one can write $T_\zeta = \zeta * (.)$.
\end{example}

\subsection{Dirichlet Boundary Conditions}
In comparison to the assumptions one requires to obtain existence from~\cite{SFBPDir}, we need additional regularity for the noise coefficient.
\begin{ass}\label{a:mu} For $\mu := \mu_+$, resp. $\mu:= \mu_-$ it holds that $\mu \in C^{1}(\R \times \R^2; \R)$, and
  \begin{enumerate}[label=(\roman*)]
  \item\label{ai:mugrowths} there exist $a\in L^2(\R)$, $b$, $\tilde b \in L^\infty_{loc}(\R^2; \R)$ such that for all $x\in \R$, $y$, $z\in \R$
    \[ \abs{\mu(x,y,z)},\,\abs{\ddx \mu(x,y,z)} \leq  b(y,z) \left(a(x) + \abs{y} + \abs{z}\right),\]
    and
    \[ \abs{\ddy \mu(x,y,z)},\; \abs{\ddz\mu(x,y,z)} \leq \tilde b(y,z), \]
  \item\label{ai:mulip} $\mu$ and its partial derivatives are locally Lipschitz continuous and the local Lipschitz constants are uniformly bounded in $x\in \R$.
  \end{enumerate}  
\end{ass}
\begin{ass}\label{a:sigma}
  For $\sigma:= \sigma_+$, resp. $\sigma:= \sigma_-$ it holds that $\sigma \in C^4(\R\times \R;\R)$, and
  \begin{enumerate}[label=(\roman*)]
  \item\label{ai:sigmagrowths} there exist $a_{i,j}\in L^2(\R)$ and $b_{i,j} \in L^\infty_{loc}(\R, \R_+)$, $i$, $j\in \N_0$, $i+j\leq 4$ such that
    \[  \abs{\tfrac{\partial^{i+j}}{\partial x^{i}\partial y^{j}} \sigma(x,y)} \leq  \begin{cases}   b_{i,j}(y) (a_{i,j}(x) + \abs{y}), & j = 0, \\ b_{i,j} (y), & j\neq 0. \end{cases}\]
  \item\label{ai:sigmalip} $\sigma$ and its partial derivatives are locally Lipschitz with Lipschitz constants independent of $x\in \R$.  
  \item\label{ai:sigmabc} $\sigma$ fulfill the boundary conditions
    \begin{equation}
      \sigma(0,0) = 0.\label{eq:sigmaD}
    \end{equation}
  \end{enumerate}
\end{ass}

\begin{ass}\label{a:init}
  Assume that $x_0\in \R$ and $v_0 \in \Gamma(x_0)$,
  i.\,e. $u_{1,0}$, $u_{2,0} \in H^2(\R_+)\cap H^1_0(\R_+)$.
\end{ass}

\subsection{Neumann and Robin Boundary Conditions}
For first order boundary conditions, we omit the $\ddx$-terms in the
dynamics of $x_*$, see~\eqref{eq:dxNeumann}, and so it suffices to
work on $H^1$ instead of $H^2$. Consequently, we can relax the spatial
regularity assumptions on the coefficients compared with the situation
of Dirichlet boundary problems.

\begin{ass}\label{a:mu1} For $\mu \in \{\mu_+,\mu_-\}$ it holds that $\mu: \R^3 \to \R$ and
  \begin{enumerate}[label=(\roman*)]
  \item\label{ai:mu1growths} there exist $a\in L^2(\R)$, $b \in L^\infty_{loc}(\R; \R)$ such that for all $x\in \R$, $y$, $z\in \R$
    \[ \abs{\mu(x,y,z)}  \leq   b(y)\left(a(x) + \abs{y} + \abs{z}\right),\]
  \item\label{ai:mu1lip} For all $x\in \R$, $\mu(x,.,.)$ is locally Lipschitz continuous and the local Lipschitz constants are uniformly bounded in $x\in \R$.
  \end{enumerate}  
\end{ass}
\begin{ass}\label{a:sigma1}
  For $\sigma:= \sigma_+$, resp. $\sigma:= \sigma_-(-.,.)$ holds $\sigma \in C^3(\R_{\geq 0}\times \R;\R)$, and
  \begin{enumerate}[label=(\roman*)]
    \item\label{ai:sigma1growths} there exist $a_{i,j}\in L^2(\R)$ and $b_{i,j} \in L^\infty_{loc}(\R, \R_+)$, $i$, $j\in \N_0$, $i+j\leq 3$ such that
      \[  \abs{\tfrac{\partial^{i+j}}{\partial x^{i}\partial y^{j}} \sigma(x,y)} \leq  \begin{cases}   b_{i,j}(y) (a_{i,j}(x) + \abs{y}), & j = 0, \\ b_{i,j} (y), & j\neq 0. \end{cases}\]
  \item\label{ai:sigma1lip} $\sigma$ and its partial derivatives are locally Lipschitz with Lipschitz constants independent of $x\in \R$.  
  \end{enumerate}
\end{ass}

\begin{ass}\label{a:init1}
  Assume that $x_0\in \R$ and $v_0 \in H^1(\R\setminus\{x_0\})$,
  i.\,e. $u_{1,0}$, $u_{2,0} \in H^1(\R_+)$.
\end{ass}

\subsection{Wong-Zakai Approximations}
As a first step, we will study an approximation technique reducing the forward invariance question to deterministic equations. To this end, we fix a finite time horizon $T>0$ and denote by $P_m$ the partitions 
\begin{equation}
  \label{eq:28}
  P_m := \left\{ \left[\tfrac{k}{m}T, \tfrac{k+1}{m}T\right] \,\vert\, k=0,...,m-1\right\}.  
\end{equation}
To map onto the time grid, we use the notation $[t]_m := \frac{k}m T$, where $k\in \{0,...,m\}$ is chosen such that $t\in [\frac{k}{m}T, \frac{k+1}{m}T)$. In order to approximate $\xi_t(x) = \sum_{k=1}^\infty T_\zeta e_k(x) \beta_k(t)$ we interpolate the Brownian motions linearly,
\begin{equation}
  \label{eq:bminterp}
  \beta_k^m(t) := \beta_k([t]_m) + (t-[t]_m) (\beta_k([t+\tfrac1m T]_m) - \beta_k([t]_m)),\quad t<T,
\end{equation}
and $\beta_k^m(T) := \beta_k(T)$. Then, we consider $\omega$-wise the partial differential equations,
\begin{gather}
  \left\{\begin{split}
      \ddt w^{m,n}_1(t,x) &= \eta_+ \ddxx w^{m,n}_1 +  \mu_+\left(x, w^{m,n}_1, \ddx w^{m,n}_1 \right)+\ddt x^{m,n}_*(t)\ddx w^{m,n}_1(t,x)\\
      &\quad  -\tfrac12  \sigma_+(x,w^{m,n}_1) \ddy \sigma_+(x,w^{m,n}_1) \norm{\zeta(x_*^{m,n}(t)+x,.)}{L^2}^2\\
      &\quad +   \sigma_+(x, w^{m,n}_1(t,x)) \sum_{k=1}^n T_\zeta e_k(x_*^{m,n}(t) + x)\dot \beta_k^m([t]_m), \\
      \ddt w^{m,n}_2(t,x) &= \eta_- \ddxx w^{m,n}_2   + \mu_-\left(-x, w^{m,n}_2, -\ddx w^{m,n}_2 \right) - \ddt x^{m,n}_*(t)\ddx w^{m,n}_2(t,x) \\
      &\quad - \tfrac12  \sigma_-(-x,w^{m,n}_1) \ddy \sigma_-(-x,w^{m,n}_1) \norm{\zeta(x_*^{m,n}(t)-x,.)}{L^2}^2\\
      &\quad + \sigma_-^k(-x,w^{m,n}_2(t,x))   \sum_{k=1}^n T_\zeta e_k( x^{m,n}_*(t) - x) \dot \beta_k^m([t]_m), \\
      w^{m,n}_1(0,x) &= u_{1,0}(x),\qquad w^{m,n}_2(0,x) =  u_{2,0}(x),
    \end{split}\right.\label{eq:PDEwza}
\end{gather}
$t\geq 0$, $x\in \R_+$, with interface condition $x_*(0) = x_0$ and either
\begin{equation}
  \label{eq:dxDirichletwza}
  \ddt x_*^{m,n}(t) = \varrho\Big(w^{m,n}_1(t,0), \ddx w^{m,n}_1(t,0), w^{m,n}_2(t,0),\ddx w^{m,n}_2(t,0)\Big),
\end{equation}
for the case $\kappa_+ = \kappa_- = \infty$ or, else,
\begin{equation}
  \label{eq:dxNeumannwza}
  \ddt x_*^{m,n}(t) = \varrho\Big(w^{m,n}_1(t,0), \ddx w^{m,n}_1(t,0), w^{m,n}_2(t,0),\ddx w^{m,n}_2(t,0)\Big),
\end{equation}
%%% Begin DIRICHLET
and with boundary conditions, for $t\in (0,T]$,
\begin{equation}
  \label{eq:47}
  \ddx w^{m,n}_1(t,0) = \kappa_+ w^{m,n}_1(t,0),\qquad   \ddx w^{m,n}_2(t,0) = \kappa_- w^{m,n}_2(t,0).
\end{equation}
Since explosion of the solutions might happen in finite time, let us introduce the exit times,
\begin{equation}
  \label{eq:46}
  \tau_k^{(r)} := \inf\{t\geq 0\, \vert\, t<\tau, \norm{u_1(t)}{H^k(\R_+)} + \norm{u_2(t)}{H^k(\R_+)} + \abs{x_*(t)} >r \},
\end{equation}
for $r>0$, $k\in \N$. The appearance of $\ddy \sigma$ in the dynamics
of~\eqref{eq:PDEwza} indicates already why we need to assume existence
of higher order derivatives in \autoref{a:sigma} and \ref{a:sigma1},
compared to the assumption for the existence results in~\cite{SFBPDir,SFBP1stOrder}.

\begin{thm}[Approximation 1]
  \label{thm:wzadirichlet}
  Assume that $\kappa_+=\kappa_- =\infty$ and that
  Assumptions~\ref{a:rho}, \ref{a:zeta}, \ref{a:mu} and~\ref{a:sigma}
  are fulfilled and denote by $(u_1,u_2,x_*)$ the unique solution
  of~\eqref{eq:SPDE} with \eqref{eq:dxDirichlet} on the maximal
  interval $\llbrak 0,\tau\llbrak$. Respectively, denote by
  $(w^{m,n}_1, w^{m,n}_2, x^{m,n}_*)$ the unique solutions
  of~\eqref{eq:PDEwza} with \eqref{eq:dxDirichletwza} and Dirichlet conditions at $0\pm$. Then, it holds for $i\in \{1,2\}$, $r>0$, that
  \[ \lim_{n\to \infty} \lim_{m\to \infty} w^{m,n}_i = u_i, \]
  and
  \[\lim_{n\to\infty} \lim_{m\to \infty} x_*^{m,n} = x_*,\]
  with uniform convergence on $[0,\tau^{(r)}_2]$ in $L^{2p}(\Omega; H^2\oplus H^2\oplus \R)$.
\end{thm}

%%%% END Dirichlet
\begin{thm}[Approximation 2]
  \label{thm:wza1}
  Assume that $\kappa_+, \kappa_- < \infty$ and that
  Assumptions~\ref{a:rho}, \ref{a:zeta}, \ref{a:mu1}
  and~\ref{a:sigma1} are fulfilled and denote by $(u_1,u_2,x_*)$ the
  unique solution of~\eqref{eq:SPDE} with \eqref{eq:dxNeumann} on the
  maximal interval $\llbrak 0,\tau\llbrak$. Respectively, denote by
  $(w^{m,n}_1, w^{m,n}_2, x^{m,n}_*)$ the unique solutions
  of~\eqref{eq:PDEwza} with \eqref{eq:dxNeumannwza} and first order boundary conditions at $0\pm$. Then, it holds for $i\in \{1,2\}$, $r>0$, 
  \[ \lim_{n\to \infty} \lim_{m\to \infty} w^{m,n}_i = u_i, \]
  and
  \[\lim_{n\to\infty} \lim_{m\to \infty} x_*^{m,n} = x_*,\]
  with uniform convergence on $[0,\tau^{(r)}_1]$ in $L^{2p}(\Omega; H^1\oplus H^1\oplus \R)$.
\end{thm}

A more clean and precise formulation of the convergence in terms of the corresponding evolution equations is provided in~\eqref{eq:wzconv1storder} and~\eqref{eq:wzconvdirichlet}, respectively.

%%%% END first order
\subsection{Outline of the Proof}
\begin{itemize}
\item In \autoref{sec:Pre} we will recall concepts of a class of interpolation spaces from analysis, which will be used for understanding the stochastic evolution equations.
\item In \autoref{sec:wza}, we switch to the abstract framework of stochastic evolution equations and consider approximations of Wong-Zakai-type. This will provide the basis for extensions of properties of deterministic equations.
\item In \autoref{sec:inv}, we discuss forward invariance and viability results for deterministic evolution equations.
\item In \autoref{sec:positivityproofs}, we reformulate the centered
  equations as (stochastic) evolution equations. We show that the
  assumptions stated above are sufficient to apply the results from
  the abstract setting and finish the proofs of
  Theorem~\ref{thm:fwdinv}, \ref{thm:wzadirichlet} and
  \ref{thm:wza1}. The convergence statements from the latter two
  statements are stated explicitly in~\eqref{eq:wzconv1storder}
  and~\eqref{eq:wzconvdirichlet}.  
\item Some results on Fr\'echet differentiability of Nemytskii
  operators and of the noise coefficients appearing here are delayed
  to \autoref{A:nem} and~\autoref{A:noise}, respectively.
\end{itemize}

%%% Local Variables: 
%%% mode: latex
%%% TeX-master: "0paper"
%%% End: 

\section{Preliminaries}
\label{sec:Pre}
\begin{ass}
  \label{a:A}
    $A$ is a densely defined and sectorial operator with domain $\dom(A)\subset E$. Moreover, the resolvent set of $A$ contains $[0,\infty)$ and there exists a $M > 0$ such that the resolvent $R(\lambda,A)$ satisfies 
  \begin{equation}\label{eq:resolvent_eq}
    \norm{R(\lambda,A)}{} \le \frac{M}{1 + \lambda}, \qquad \text{for all $\lambda > 0$.}
  \end{equation}
\end{ass}
\begin{rmk}
          This assumption is equivalent to each of the following statements
  \begin{itemize}
  \item  Equation \eqref{eq:resolvent_eq} holds and the resolvent set of $A$ contains $0$ and a sector
    \[\{\lambda \in \C: |\arg \lambda | < \theta\}\]
    for some $\theta \in (\pi/2,\pi)$. 
  \item The operator $A$ is sectorial and $-A$ is positive in the sense of \cite{lunardiInterpol}.
  \item $A$ is the generator of an analytic $C_0$-semigroup $(S_t)_{t \geq 0}$ of negative type, 
  \end{itemize}
  In particular, there exist $\delta$, $M > 0$ such that $\norm{S_t}{} \le M e^{-\delta t}$. Assumption \eqref{a:A} also ensures that fractional powers of $-A$ are well defined. 
\end{rmk}

\begin{notation}
  For $\alpha \geq 0$ we write
  \begin{equation}
    \label{eq:Ealpha}
    E_\alpha :=  \dom((-A)^\alpha),\qquad \norm{h}{\alpha} := \norm{(- A)^\alpha h}{E}, \; h\in E_\alpha.
  \end{equation}
  It is known that also $E_\alpha$ with the induced scalar product is a separable Hilbert space. In particular, $\norm{.}{1}$ is equivalent to the graph norm of $A$ and the following continuous embedding relations hold for $\alpha \in [0,1]$:
  \begin{equation}
    \label{eq:15}
    \dom(A) = E_1 \hookrightarrow E_\alpha \hookrightarrow E_0 = E, 
  \end{equation}
\end{notation}

  Note that the restriction of $A$ to any $E_\alpha, \alpha \in [0,1]$ is again a densely defined and closed operator on $E_\alpha$. Moreover, it is the infinitesimal generator of the restriction of $S_t$ to $E_\alpha$, which is again an analytic (contraction) semigroup; see e.g. \cite[Ch.~II.5]{engelnagel}. We in particular have the following property.
\begin{prop}\label{prop:reiteration}
  For $\alpha$, $\beta\in \R$, and $u\in \dom((-A)^{\alpha + \beta})$ it holds that 
  \[ (-A)^{\alpha}((-A)^{\beta} u) = (-A)^{\beta}((-A)^{\alpha} u) = (-A)^{\alpha+\beta} u.\]
\end{prop}
\begin{rmk}\label{rmk:interpol:reiteration}
  The part of $A$ in $E_\alpha, \alpha >0$, is again a densely defined and closed operator on $E_\alpha$. Moreover, it is the infinitesimal generator of the restriction of $S_t$ to $E_\alpha$, which is again an analytic and strongly continuous semigroup. The same holds true on the spaces $D(\alpha,p)$ and $D(\alpha)$, with $\alpha \in (0,1)$, $p<\infty$ and for the extension of $S_t$ to $E_\alpha$, when $\alpha <0$; see e.g. \cite[Ch.~II.5]{engelnagel}.
\end{rmk}
The following regularity property of $S_t$ between different interpolation spaces $E_\alpha$, $\alpha \in [0,1]$ will be crucial in the proofs that follow. We derive it from results in~\cite{lunardiInterpol} on interpolation spaces.
\begin{lem}\label{lem:StHa}
  Let $\beta \geq 0$ and $\alpha > \beta$. Then, for all $t>0$ and $h\in E_{\beta}$,
  \[ \norm{S_t h}{\alpha} \leq K_{\alpha,\beta} t^{\beta - \alpha}  e^{-\delta t}\norm{h}{\beta}. \]
\end{lem}
Note that the factor in front of $\norm{h}{\beta}$ is integrable at time $t = 0$, which is the key property used in the following sections. On the other hand, to deal with the singularity in $0$, we will use an extended version of Gronwall's lemma, see~\cite[Lem 7.0.3]{lunardiAnalytic} or, for a proof,~\cite[p. 188]{henryGeo}.
\begin{lem}[Extended Gronwall's lemma]\label{lem:gronwall}
  Let $\alpha>0$, $a$, $b\geq 0$, $T \geq 0$, and $u:[0,T]\to\R$ be non-negative and integrable. If, for all $t\in [0,T]$,
  \begin{equation}
    \label{eq:6}
    u(t) \leq a + b\int_0^t u(s) (t-s)^{\alpha-1} \d s,
  \end{equation}
  then exists a constant $K_{\alpha, b,T}$, depending only on $\alpha$, $b$ and $T$, such that,
  \begin{equation}
    \label{eq:18}
    u(t) \leq a K_{\alpha,b,T},\qquad t\in [0,T].
  \end{equation}
\end{lem}

%%% Local Variables: 
%%% mode: latex
%%% TeX-master: "0paper"
%%% End: 

\section[Wong-Zakai Approximation]{Wong-Zakai Approximation for Stochastic Evolution Equations}
\label{sec:wza}

In this section, we discuss an approximation method of Wong-Zakai-type for a class of semilinear stochastic evolution equations in the mild framework. We extend the proof of Nakayama~\cite{nakayamasupport} to the situation where the linearity generates an analytic semigroup but the drift can be controlled only on a smaller subspace, and, in addition, consider the case when the coefficients are Lipschitz continuous only on bounded sets. The latter seems to be new even for the classical situation of Heath-Jarrow-Morton-type equations.

Let $U$ and $E$ be separable Hilbert spaces and $(\Omega, \F, (\F_t),
\PP)$ be a filtered probability space on which an $U$-valued
cylindrical Wiener process with covariance operator $\Id_U$
lives. Recall that there exist independent real $(\F_t)$-Brownian
motions $\beta^k$, $k\in \N$, such that for a CONS $(e_k)_{k\in \N}$ of $U$, $t\in [0,T]$,
\[ W_t = \sum_{k=1}^\infty e_k \beta^k(t).\]

We keep $(e_k)$ fixed for the remainder of this section and consider the stochastic evolution equation
\begin{equation} \label{eq:eeq}
  \left\{ \begin{aligned}
      \d X(t) &= \left[AX(t) + B(X(t))\right] \d t + C(X(t))\d W_t,\;t\geq 0,\\
      X(0) &= X_0,
    \end{aligned}\right.
\end{equation}
where $A$ is a linear operator on $E$ with domain $\dom(A)$, and $C: E\rightarrow \HS(U;E)$ the noise coefficient and $B: E\rightarrow E$ the non-linear part of the drift term, are assumed to be Borel measurable functions. Recall that a \emph{mild solution} of~\eqref{eq:eeq} on a Hilbert space $H\hookrightarrow E$ is a predictable $H$-valued process which satisfies the $H$-integral equation
\begin{equation}
  \label{eq:1}
  X(t) = S_t X_0 +   \int_0^t S_{t-s}B(X(s)) \d s + \int_0^t S_{t-s} C(X(s))\d W_s,
\end{equation}
on $\llbrak 0,\tau\llbrak$ for a strictly positive predictable stopping time $\tau>0$. The stopping time $\tau$ is called maximal if there does not exist a solution on a striclty larger stochastic interval and the solution is global, if $\tau = T$ almost surely. In the following, we will be in the situation where $H$ will be the domain of a fractional power of $-A$. In order for them to be well-defined, we need the following.

%%%%%%%%%%%%%%%%%%%%%%%%%%%%%%%%%%%%%%%%%%%%%%%%%%%%%%%%%%%
\begin{ass}
  \label{a:A}
  Let $A$ be the generator of an analytic $C_0$-semigroup of negative type on E.
\end{ass}
As in~\eqref{eq:Ealpha}, we set $E_\alpha := \dom((-A)^\alpha)$ for $\alpha\in\R$, and write $\norm{u}{\alpha} := \norm{(-A)^\alpha u}{E}$, $\alpha \in \R$. It is worth to recall that $A$ fulfills \autoref{a:A} on $E_\alpha$ for any $\alpha\in\R$, once it is fulfilled for $\alpha=0$. From now on, we keep $\alpha \in [0,1)$ fixed. Throughout this section, we will assume $X_0\in E_\alpha$ is deterministic. The following assumptions ensure that there exists a unique global mild solution of~\eqref{eq:eeq} on $E_\alpha$, cf.~\cite[Theorem~3.9]{SFBPDir}. 

\begin{ass}%
  $\;$\label{a:Ball}
  \begin{enumerate}[label=($B$.\roman*)]
  \item \label{a:driftlip} $B:E_\alpha\to E_0$ is Lipschitz continuous with Lipschitz constant $L_B$.
  \item \label{a:driftbd} For some constant $M_B$ and all $u\in E_\alpha$ it holds that $\norm{B(u)}{0}\leq M_B$. 
  \end{enumerate}  
\end{ass}

Introduce the shorthand,
\[ \sigma_k (u):= C(u)e_k,\; u\in E_\alpha,\;k\in\N.\]
Because $(e_k)$ is an CONS of $U$, we can decompose $C$ as
\begin{equation}
  \label{eq:Cek}
  C(u)w = \sum_{k=1}^\infty \scal{w}{e_k}{U}\sigma_k(u),
\end{equation}
for all $w\in U$, $u\in E_\alpha$. Keeping this mind, we define the Stratonovich or Wong-Zakai correction term for the projection of $C$ on the linear span $\linspan{e_1,...,e_n}$,
\[ \Sigma_n(u) := \frac12 \sum_{k=1}^n D\sigma_k(u)\sigma_k(u),\;u\in E_\alpha.\]

\begin{ass}%
  $\;$ \label{a:Call}
  \begin{enumerate}[label=($C$.\roman*)]
  \item $C: E_\alpha\to \HS(U;E_\alpha)$ is Lipschitz continuous with Lipschitz constant $L_C$
  \item There exists a constant $M_C$ such that $\norm{C(u)}{\HS(U;E_\alpha)} \leq M_C$, $\forall u\in E_\alpha$
  \item $\sigma_k:E_\alpha\to E_\alpha$, $k\in \N$ are twice Fr\'echet differentiable and the derivatives are bounded
  \item\label{ai:SigmaLip} There exists an $L_\Sigma >0$ such that \( \norm{\Sigma_n(u) - \Sigma_n(v)}{E_\alpha} \leq L_\Sigma\norm{u-v}{E_\alpha} ,\;\forall\, u,v\in E_\alpha,\,\forall n\in\N\).
  \item\label{ai:SigmaConv} There exists $\Sigma_\infty:E_\alpha\rightarrow E_\alpha$ such that
    \[\lim_{n\to \infty} \Sigma_n(u) = \Sigma_\infty(u),\qquad \forall\, u\in E_\alpha.\]
  \end{enumerate}
\end{ass}

\begin{rmk}
  The conditions on $\Sigma_n$ imply that $\Sigma$ is Lipschitz continuous on $E_\alpha$ with Lipschitz constant $L_\Sigma$. Moreover, $\Sigma_n\to \Sigma$ as $n\to \infty$ uniformly on compact sets in $E_\alpha$. Indeed, let $K\subset E_\alpha$ be compact. For $\epsilon >0$ let $\bigcup_{k=1}^N K(u_k,\delta)$ be the covering by open balls of radius $\delta:= \sfrac{\epsilon}{(4L_\Sigma)}$. Moreover, let $n_0\in \N$ such that for all $n\geq n_0$ it holds that
  \[\sup_{k=1,..,N} \norm{\Sigma_n(u_k) - \Sigma_\infty(u_k)}{\alpha}<\sfrac{\epsilon}{2}.\]
  Hence, for all $n\geq n_0$,
  \begin{gather*}
    \begin{split}
      \sup_{u\in K} \norm{\Sigma_n(u) - \Sigma_\infty(u)}{\alpha} &\leq \sup_{k=1,...,N}\sup_{u\in K(u_k,\delta)} \left[\norm{\Sigma_n(u) - \Sigma_n(u_k)}{\alpha} \right.\\
      &\qquad + \left.\norm{\Sigma_\infty(u) - \Sigma_\infty(u_k)}{\alpha} + \norm{\Sigma_n(u_k) - \Sigma_\infty(u_k)}{\alpha}\right]\\
      &< 2 L_\Sigma \delta + \frac{\epsilon}{2} = \epsilon.      
    \end{split}
  \end{gather*}
  
\end{rmk}
\begin{rmk}
  Even in the finite dimensional case with $A=0$ it is often assumed that $B$ is globally bounded and $C$ of class $C^2_b$, see \cite{wz} for the scalar case $E= \R$ and \cite{stroockvaradhanSupport}, \cite[Thm 7.2]{ikedaWatanabe} for Wong-Zakai approximations for stochastic differential equations on $\R^d$.  We will pass over to the more general case by truncation in Subsection~\ref{ssec:wzalocal} below.
\end{rmk}
Note that $B$ and $C$ can be trivially extended to Borel functions on
$E$, since $E_\alpha$ is an $E$-Borel set by continuity of the
imbedding $E_\alpha \hookrightarrow E$ and Kuratowski's
theorem. Moreover, the Lipschitz conditions on $B$ and $C$ yield
existance of a unique mild solution $X$ of~\eqref{eq:eeq} on
$E_{\alpha}$, cf. \cite[Theorem 3.9]{SFBPDir}. Moreover, $X \in C([0,T]; E_\alpha)$ a.\,s., and for all $p>1$
\begin{equation}
  \E{\sup_{0\leq t\leq T} \norm{X(t)}{\alpha}^{2p}} \leq K_{p,T} \left(1 + \E{\norm{X_0}{\alpha}^{2p}}\right). \label{eq:EsupXt}
\end{equation}

For $m\in \N$ we introduce the time grid $k\delta_m$, $k=0,...,m$ on $[0,T]$,
\[\delta_m:=\sfrac{T}{m},\quad \tmm := k\delta_m,\] 
where $k\in\{0,...,m-1\}$ is chosen such that $t\in [k\delta_m, (k+1)\delta_m)$, $t\in [0,T]$. On this grid, we have already defined the linearly interpolated Brownian motions $\beta_m^k$, see~\eqref{eq:bminterp}. 

By $\dot\beta_m^k$ we denote the time-derivatives which exist piece-wise on $[0,T]$. The approximating Wong-Zakai equations are then the random evolution equations, 
\begin{equation} \label{eq:eeqWZAbd}
  \left\{ \begin{aligned}
      \ddt {Z}_{m,n}(t) &= A Z_{m,n}(t) + B(Z_{m,n}(t)) - \Sigma_\infty(Z_{m,n}(t)) \\
      &\qquad\qquad\qquad\qquad\qquad\qquad+ \sum_{k=1}^n \sigma_k(Z_{m,n}(t)) \dot{\beta}^k_m(t),\; t>0,\\
      Z_{m,n}(0) &= X_0,
    \end{aligned}\right.
\end{equation}

\begin{thm}\label{thm:WZA}
  Let Assumption~\ref{a:A},~\ref{a:Ball} and~\ref{a:Call} be satisfied and denote respectively by $X$ and $Z_{m,n}$ the unique mild solutions of~\eqref{eq:eeq} and~\eqref{eq:eeqWZAbd} on $E_\alpha$ for initial data $X_0 \in E_\alpha$. Then, for all $p>1$ it holds that 
  \[ \lim_{n\to \infty} \lim_{m\to\infty} \E{\sup_{0\leq t\leq T} \norm{X(t) - Z_{m,n} (t) }{\alpha}^{2p}} = 0.\]
\end{thm}
The proof is split into two main steps. First, we project the noise coefficient onto $\linspan{e_1,...,e_n}$, $n\in\N$ and obtain convergence. 

To this end, set $C_n :=   \sum_{k=1}^n\scal{.}{e_k}{U} \sigma_k$ and $B_n:= B +\Sigma_n - \Sigma_\infty$ for $u\in E_\alpha$. Since $C_n$ is the projection of $C$ to $\linspan{e_1,...,e_n}$ we get that $C_n: E_\alpha \to \HS(U;E_\alpha)$ is Lipschitz continuous, uniformly in $n\in\N$. By assumption, also $B_n$ is Lipschitz continuous with uniform Lipschitz constant. We denote by $X_n$ the unique global mild solutions of
\begin{equation}
  \label{eq:SEEnoiseapprox}
  \d X_n (t) = \left[A X_n(t) + B_n(X_n(t)) \right]\d t + C_n(X_t) \d W_t   
\end{equation}
The following result is now a direct application of the continuity of the solution map in the coefficients.
\begin{prop} \label{prop:continuityCoeff}
  Assume that Assumptions~\ref{a:A},~\ref{a:Ball} and~\ref{a:Call} hold true. Then, for the unique global mild solutions $X_n$ of~\eqref{eq:SEEnoiseapprox} and the unique mild solution $X$ of~\eqref{eq:eeq} on $E_\alpha$, for all $p>1$, it holds that
  \[\lim_{n\to \infty} \E{\sup_{0\leq t\leq T}\norm{X_n (t) - X(t)}{\alpha}^{2p}} = 0.\]
\end{prop}
\begin{proof}
  The convergence is a special situation of \cite[Prop. 3.2]{kunzeApprox}.
\end{proof}
This reduces the problem to the situation where $\sigma_k = 0$ for all but finitely many $k\in \N$. To keep a level of generality, we treat this case in a separate framework. \autoref{prop:continuityCoeff} together with \autoref{thm:WZAfinite} then yields \autoref{thm:WZA}.

\subsection{Finite Dimensional Noise}
To prove the convergence in the case of finite dimensional noise we
extend several estimates in~\cite{nakayamasupport} using the tools
from interpolation theory, see~\autoref{sec:Pre} .
We now consider the stochastic evolution equation, with finite dimensional noise,
\begin{gather}
  \label{eq:eeqn}
  \left\{\begin{split}
    \d X(t) &= \left[ A X(t) + B(X(t)) + \Sigma_n(X(t))\right] \d t+ \sum_{k=1}^n \sigma_k(X(t)) \d \beta^k_t,  \\
    X(0) &= X_0.    
  \end{split}\right.
\end{gather}
Here, $X_0\in E_\alpha$ is deterministic, as above. For this subsection, $\sigma_k$ are functions on $E_\alpha$ satisfying the following conditions.
\begin{ass}
  \label{a:noisefinite}
  For $k=1,...,n$ assume that $\sigma_k:E_\alpha\to E_\alpha$ is twice Fr\'echet differentiable, and that $\sigma_k$, $D\sigma_k$ and $D^2\sigma_k$ are bounded in $E_\alpha$.
\end{ass}
We define, similar to the situation above,
\[ \Sigma_n (u) := \frac12 \sum_{k=1}^n D\sigma_k(u) \sigma(u),\quad u\in E_\alpha.\]
\begin{lem}
  \label{lem:SigmanLipBd}
  $\Sigma_n: E_\alpha \to E_\alpha$ is bounded and Lipschitz continuous.
\end{lem}
\begin{proof}
  By assumption, for all $k=1,...,n$, $\sigma_k$ and $D\sigma_k$ are differentiable with bounded derivative. For $u$, $v\in E_\alpha$, it holds that 
  \begin{multline*}
    \norm{D\sigma_k(u)\sigma_k(u) - D\sigma_k(v)\sigma_k(v)}{\alpha}\leq \\
    \leq \norm{D\sigma_k(u)}{\Lbd(E_\alpha)} \norm{u-v}{\alpha} + \norm{D\sigma_k(u) - D\sigma_k(v)}{\Lbd(E_\alpha)} \norm{v}{E_\alpha}.
  \end{multline*}
  Hence, application of the mean value theorem to $D\sigma_k$ yields
  Lipschitz continuity of the maps $u\mapsto D\sigma_k(u)\sigma_k(u)$,
  $k\in \N$. Since the finite sum of Lipschitz functions is Lipschitz again the lemma is proven.
\end{proof}
We define $\beta^k_m$, $k=1,...,n$, in the same way as in~\eqref{eq:bminterp} so that the Wong-Zakai approximation of~\eqref{eq:eeqn} will be the solution of the random evolution equation
\begin{gather}
  \label{eq:eeqnWZ}
\left\{  \begin{split}
    \ddt Z_m (t) &=  A Z_m(t) + B(Z_m(t))+ \sum_{k=1}^n \sigma_k(Z_m(t))\dot\beta^k_m(t),\quad t\in [0,T]\\
  Z_m(0) &=X_0.    
\end{split}\right.
\end{gather}
On $[i\delta_m, (i+1)\delta_m)$, $i=0,...,m-1$, the coefficients of~\eqref{eq:eeqnWZ} are time-homogeneous and satisfy global Lipschitz assumptions. A standard existence results for deterministic equations, e.\,g.~\cite[Theorem 6.3.3]{pazySemigroups} yields existence and uniqueness of a global mild solution of the random equation, $\omega$-wise. By concatenation, we observe existence and uniqueness of a global continuous mild solution on $E_\alpha$.

\begin{rmk}
  The correction term $\Sigma_n$ can be removed from~\eqref{eq:eeqn} but then has to appear in~\eqref{eq:eeqnWZ} with a negative sign. Recall that the occurrence of the Stratonowitch correction term was quite surprising and an important step in the understanding of stochastic differential equations in terms of physical systems~\cite{wz}.
\end{rmk}
\begin{thm}
  \label{thm:WZAfinite}
  Let \autoref{a:A}, \autoref{a:Ball} and \autoref{a:noisefinite} hold true and denote by $X$ and $Z_m$ respectively the unique mild solution of~\eqref{eq:eeqn} and~\eqref{eq:eeqnWZ}, for $m\in\N$. Then for any $p>1$,
  \begin{equation}
    \label{eq:3}
    \lim_{m\to \infty}\E{\sup_{0\leq t\leq T}\norm{X(t) - Z_m(t)}{\alpha}^{2p}} = 0.
  \end{equation}
\end{thm}
This theorem is an extension of \cite[Prop 2.1]{nakayamasupport}. We roughly follow its proof from \cite[Section 2]{nakayamasupport} but perform the necessary changes. Since the noise terms $\sigma_k$ and the linear operator $A$ fulfill the assumptions in the reference, on $E_\alpha$, Lemmas 2.9, 2.11 and 2.12 in \cite{nakayamasupport} also apply in our setting. These parts are the basis for \cite[Lemma 2.13]{nakayamasupport}, which reads in our framework as follows.
\begin{lem}\label{lem:IntdiffNoise}
  For $p>1$ exists a constant $K_{T,p}$ and a sequence $(\epsilon_m)_{m\in\N}$ with
  \[  \lim_{m\to \infty}  \epsilon_m = 0,\]
  such that
  \begin{multline*}
    \EE\left[\sup_{0\leq t\leq T} \left\lVert \int_0^{\tmm} S_{t-s} \sigma_k(Z_m(s)) \dot \beta^k_m(s)\d s - \int_0^{\tmm} S_{t-s} \sigma_k(X(s)) \d\beta^k(s)\right.\right.\\ -\left.\left. \frac12 \int_0^{\tmm} S_{t-s} D\sigma_k(X(s))\sigma_k(s)\d s\right\rVert_{\alpha}^{2p}\right] \\
    \leq K_{T,p} \int_0^T\E{\sup_{0\leq s\leq t} \norm{Z_m(s) - X(s)}{\alpha}^{2p}} \d s + \epsilon_m.
  \end{multline*}
\end{lem}
 The proof of this lemma is based on the da Prato-Kwapien-Zabczyk factorization method which was introduced in~\cite{dPKZ} and has also been used, for instance, in~\cite{dPZstochConv} and~\cite{dPZinf} to show continuity of the stochastic convolution. Since the noise operator $C$ satisfies the ``standard'' assumptions, there are no modifications necessary in the proof. However, we have to adapt the parts involving also the drift term $B$, which does not fulfill the assumptions in the reference~\cite{nakayamasupport}. We now restrict the solutions onto the time grid $\{ i\delta_m\,\vert\, i=0,...,m\}$. For $m\in \N$ we define
\begin{equation*}
  \bar Z_m(t) := S_{t-\tmm} Z_m(\tmm),\quad\text{and}\quad\bar X_m(t) := S_{t-\tmm} X(\tmm).
\end{equation*}
At this point, recall that we have set $\tmm := i \delta_m$ for $i <m$ so that $t\in [i\delta_m, (i+1)\delta_m)$.
\begin{lem}
  \label{lem:Eubu}
  There exists a constant $K_{n,T,p,\alpha}>0$ such that for all $m\in \N$
  \begin{equation*}
    \E{\sup_{0\leq t\leq T}\norm{Z_m(t) - \bar Z_m(t)}{\alpha}^{2p}} \leq K_{n,T,p,\alpha} (m^{-2p(1-\alpha)} + m^{1-p}).
  \end{equation*}
\end{lem}

\begin{proof}
  We first decompose, for $t\in [0,T]$,
  \begin{equation}
 \label{eq:ubudiff}
    Z_m(t) - \bar Z_m(t) 
    = \int_{\tmm}^t S_{t-s}B(Z_m(s))\d s
    \quad + \sum_{k=1}^n \int_{\tmm}^t S_{t-s} \sigma_k(Z_m(s)) \dot\beta^k_m(s) \d s.
  \end{equation}
  From boundedness of $B$ and \autoref{lem:StHa} the first integral can be controlled by
  \begin{equation}
    \label{eq:6}
    \norm{ \int_{\tmm}^t S_{t-s}B(Z_m(s))\d s}{\alpha} \leq \int_{\tmm}^t K_\alpha M_B (t-s)^{-\alpha} \d s 
    \leq \frac{K_\alpha M_B}{1-\alpha} \delta_m^{1-\alpha} 
  \end{equation}
  To bound the second term first note that 
  \begin{equation*}
    \norm{\int_{\tmm}^t S_{t-s}\sigma_k(\xi_m(s))\dot{\beta}^j_m(s)\d s}{\alpha} \leq M_{k} \abs{\beta^k(\tmp) - \beta^k (\tmm)}
  \end{equation*}
  and then, following a standard procedure,
  \begin{equation*}
    \EE\left[ \sup_{0\leq t\leq T}\abs{\beta^k(\tmp) - \beta^k(\tmm)}^{2p}\right] \leq  \delta_m^p m\E{\abs{\beta_k(1)}^{2p}} =: K_{T,p} m^{1-p}.
  \end{equation*}
  We put everything together, 
  \begin{multline*}
    \E{\sup_{0\leq t\leq T}\norm{ Z_m(t) - \bar Z_m(t)}{\alpha}^{2p}} \leq  K_{2p} M_B^{2p} K_\alpha^{2p} \delta_m^{2p(1-\alpha)} + n M_k^{2p} \tilde K_{T,p}m^{1-p}\\
    \leq K_{p, T, \alpha} (m^{-2p(1-\alpha)} + n m^{1-p}).\qquad\qquad\qquad \qedhere
  \end{multline*}
\end{proof}

\begin{lem}
  \label{lem:EXbX}
  There exists a constant $K_{n,T,p}>0$ such that for all $m\in \N$
  \[ \E{\sup_{0\leq t\leq T} \norm{X(t) - \bar X_m(t)}{\alpha}^{2p}} \leq K_{n,T,p,\alpha}\left(m^{-2p\alpha} + m^{1-p}\right).\]
\end{lem}
\begin{proof}
  First, we write
  \begin{gather}
    \begin{split}
      X(t) - \bar X_m(t) &= \int_{[t]_m^-}^t S_{t-s} B(X(s)) \d s + \frac12 \sum_{k=1}^n \int_{\tmm}^t S_{t-s} D\sigma_k(X(s)) \sigma_k(X(s)) \d s\\
      &\qquad\qquad +\sum_{k=1}^n \int_{\tmm}^t S_{t-s} \sigma_k(X(s))\d \beta^k(s),\;t\in [0,T].
    \end{split}
  \end{gather}
  The first term can be controlled by $m^{-2p(1-\alpha)}$ in the same way as in the proof of \autoref{lem:Eubu}. For the Stratonovitch correction we get
  \[\norm{\int_{\tmm}^t S_{t-s} D\sigma_k(X(s)) \sigma_k(X(s)) \d s}{\alpha} \leq K_1M_k \delta_m,\]
  using that $(S_t)$ is strongly continuous and that \autoref{a:noisefinite} holds true. The stochastic integrals can be bounded by~\cite[Lemma 2.4]{nakayamasupport} which says, applied to $t\mapsto \sigma_k(X(t))\in E_\alpha$,
  \begin{equation*}
    \E{\sup_{0\leq t\leq T}\norm{\int_{\tmm}^t S_{t-s} \sigma_k(X(s))\d \beta^k(s)}{\alpha}^{2p}} \leq K_{T, p,k} \delta_m^{p-1}.
  \end{equation*}
  Finally,
  \begin{gather}
    \begin{split}
      \EE\left[\sup_{0\leq t\leq T}\norm{X(t) - \bar X_m(t)}{}^{2p}\right] \leq K_{T,\alpha,p}\left( m^{-2p(1-\alpha)} +n m^{-2p}+ n m^{1-p}\right).\qedhere
    \end{split}
  \end{gather}  
\end{proof}
\begin{rmk}
  Here, however, we directly see how the modifications of the assumptions can be covered by the tools provided by interpolation theory in \autoref{sec:Pre}. Likewise, one could directly apply results on space-time regularity of $X(t) - S_tX_0$. For instance, by~\cite[Prop. 4.2]{weisEvEqBS} the stochastic convolution is H\"older continuous with exponent $< \frac12 - \frac{1}{2p}$ which yields an estimate of type $\delta_m^{\lambda}$ with $0 <\lambda < p-1$. 
\end{rmk}
\begin{rmk}
  The last two proofs particularly show why we have to consider first the limits for $m\to\infty$, and then $n\to \infty$.
\end{rmk}

We now collect the arguments to finish the proof of \autoref{thm:WZAfinite}.
\begin{proof}[Proof of \autoref{thm:WZAfinite}]
  Applying \autoref{lem:Eubu} and~\ref{lem:EXbX}, we get a constant $K$, depending on $n$, $p$, $T$, and $\alpha$, such that
  \begin{equation*}
    \EE{\sup_{0\leq t\leq T} \norm{ X(t) -  Z_m(t)}{\alpha}^{2p}}\leq K  \left( \EE{\sup_{0\leq t\leq T} \norm{\bar X_m(t) - \bar Z_m(t)}{\alpha}^{2p}} + m^{-2p\alpha } + m^{1-p}\right).
  \end{equation*}
  Inserting the mild integral formulae respectively for  $Z_m$ and $X$ yields the decomposition
  \begin{gather*}
    \begin{split}
    \bar X_m(t) - \bar Z_m(t) 
    &= \int_0^{\tmm} S_{t-s} \left(B(X(s)) - B(Z_m(s)\right) \d s \\
    &\quad + \sum_{k=1}^n\int_0^{\tmm} S_{t-s}  \sigma_k(X(s)) \d \beta^k(s)  - \int_0^{\tmm} S_{t-s} \sigma_k(Z_m(s)) \dot\beta^k_m(s) \d s\\
    &\quad \qquad\qquad+ \frac12 \int_0^{\tmm} S_{t-s} D\sigma_k(X(s)) \sigma_k(X(s)) \d s .    
    \end{split}
  \end{gather*}
  For the first term on the right hand side, note that Assumption~\ref{a:driftlip} implies
  \begin{multline}
    \E{\sup_{0\leq t\leq T} \norm{\int_0^{\tmm} S_{t-s}\left[B(Z_m(s)) - B(X(s))\right]\d s}{\alpha}^{2p}} \\
    \leq \tilde K_{T,\alpha,p} \int_0^T \E{\sup_{0\leq s\leq t}\norm{Z_m(s) - X(s)}{\alpha}^{2p}} \frac{\d t}{(T-t)^{\alpha}}. \label{eq:IntdiffDrift}
  \end{multline}
  The remaining summands are covered by \autoref{lem:IntdiffNoise}, so that 
  \begin{multline*}
    \E{\sup_{0\leq t\leq T} \norm{\bar X_m(t) - \bar Z_m(t)}{\alpha}^{2p}}\\
    \leq K_{T,p,\alpha} \int_0^T \E{\sup_{0\leq s\leq t}\norm{Z_m(s) - X(s)}{\alpha}^{2p}} \frac{\d t}{(T-t)^{\alpha}}\\
    + K_{T,p,\alpha} \int_0^T \E{\sup_{0\leq s\leq t}\norm{Z_m(s) - X(s)}{\alpha}^{2p}}\d t 
    + \epsilon_m,
  \end{multline*}
  where $(\epsilon_m)_{m\in\N}$ is a sequence with $\lim_{m\to\infty} \epsilon_m= 0$. Note that on any time interval $(0,t)$ it holds that
  \[ 1 = s^{\alpha-1}s^{1-\alpha} \leq t^{1-\alpha} s^{\alpha-1}, \;s\in (0,t)\]
  and thus
  \begin{multline*}
  \int_0^T \E{\sup_{0\leq s\leq t}\norm{Z_m(s) - X(s)}{\alpha}^{2p}}\d t \\ \leq K_{T,\alpha} \int_0^T \E{\sup_{0\leq s\leq t}\norm{Z_m(s) - X(s)}{\alpha}^{2p}}\frac{\d t}{(T-t)^{\alpha}} + K_{T,\alpha} \epsilon_m.    
  \end{multline*}
  To obtain the convergence, we use Gronwall's lemma in its extended version~\autoref{lem:gronwall}. 
\end{proof}

\subsection{Convergence of Local Solutions}
\label{ssec:localapprox}

%%%%%%%%%%%%%%%%%%% 
Before we go into more detail for the approximation of local solutions of stochastic evolution equations, we discuss preliminary results on explosion times for stochastic processes. We prepare this discussion by some natural thoughts about the deterministic situation.
\begin{lem}\label{lem:exitdet}
  Let $V$ be a real Banach space, $T>0$ and $f_n$, $f_\infty \in C([0,T];V)$, such that $f_n \to f_\infty$ uniformly, as $n\to \infty$. Fix $r>0$ and define
  \[ t_n := \inf\{t\geq 0\,\vert\, \norm{f_n(t)}{V} > r\},\quad s_n := \inf\{t\geq 0\,\vert\, \norm{f_n(t)}{V}\geq r\},\quad n\in \N\cup \{\infty\}.\]
  Here, we set $\inf \emptyset := T$. Then,
  \[\liminf_{n\to\infty} s_n\geq s_\infty,\quad  t_\infty \geq \limsup_{n\to \infty} t_n.\]
  Moreover, if $s_\infty = t_\infty$, then
  \[ \lim_{n\to\infty}s_n = \lim_{n\to \infty} t_n = t_\infty.\]
\end{lem}
\begin{rmk}
  Denote by $t_{n}^{(r)}$, $s_n^{(r)}$ the exit times of the balls of radius $r>0$, for $f_n$ as in the lemma. Then, \autoref{lem:exitdet} yields for all $\epsilon >0$
  \begin{equation}\label{eq:exitestdet}
    \limsup_{n\to\infty} s_n^{(r)}\leq  \limsup_{n\to \infty} t_n^{(r)} \leq t_\infty^{(r)} \leq s_\infty^{(r+\epsilon)} \leq \liminf_{n\to\infty} s_n^{(r+\epsilon)} \leq \liminf_{n\to\infty} t_n^{(r+\epsilon)}.    
  \end{equation}
\end{rmk}
\begin{proof}
  Without loss of generality assume $t_\infty < T$. Else, the estimate holds true trivially. Now, for all $\epsilon>0$ there exists $\delta >0$ such that $t_\infty+\epsilon \leq T$ and 
  \[\sup_{0\leq s \leq t_\infty+\epsilon} \norm{f_\infty(s)}{V} >r+\delta.\]
  Let $N\in \N$ such that for all $n\geq N$ ,
  \[\sup_{0\leq t\leq T} \norm{f_\infty(t) - f_n(t)}{V} <\delta.\]
  By triangle inequality, it holds that
  \[ \sup_{0\leq s\leq t_\infty+\epsilon}\norm{f_n(t)}{V} > r,\]
  and thus $t_n< t_\infty+\epsilon$. 
  
  The proof for $s_n$ in the opposite direction works quite similar. Without loss of generality we assume $s_\infty>0$. Then, for all $\epsilon >0$ there exists $\delta >0$ such that $s_\infty - \epsilon >0$ and 
  \[ \sup_{0< s<s_\infty-\epsilon} \norm{f_\infty(s)}{V} <r-\delta.\]
  Choosing $N\in \N$ as above, we get that for all $n\geq N$,
  \[\sup_{0<s<s_\infty-\epsilon} \norm{f_\infty(s)}{V} <r,\]
  and thus, $s_n>s_\infty -\epsilon$. Since $\epsilon>0$ was chosen arbitrarily, the first result follows. 
  
  To prove the last statement in the lemma we assume additionally $s_\infty = t_\infty$, and observe
  \begin{equation*}
    \liminf_{n\to \infty} t_n \geq \liminf_{n\to \infty} s_n \geq s_\infty =t_\infty\geq \limsup_{n\to\infty} t_n \geq \limsup_{n\to\infty} s_n.\qedhere
  \end{equation*}
\end{proof}
We also obtain the following convergence result.
\begin{lem}
  \label{lem:convdet}
  In the setting and notation of \autoref{lem:exitdet}, it holds for all $\epsilon >0$, $r>0$,
  \[ \lim_{n\to \infty} f_n(s_n^{(r+\epsilon)} \wedge .) =  f_\infty,\quad \text{uniformly on }[0,t_\infty^{(r)}].\]
\end{lem}
\begin{proof}
  Without loss of generality assume $t_\infty^{(r)}>0$, else, the claim holds true trivially. Now, let $N\in\N$ such that for all $n\geq N$,
  \[\sup_{0\leq t\leq T} \norm{f_n(t) - f_\infty(t)}V <\epsilon.\]
  Thus,
  \[ \sup_{0\leq t\leq t_{\infty}^{(r)}} \norm{f_n}V \leq r + \epsilon,\]
  which implies $t_\infty^{(r)} \leq s_n^{(r+\epsilon)}$, for all $n\geq N$. Finally,
  \[ \lim_{n\to\infty} \sup_{0 \leq t\leq t_\infty^{(r)}} \norm{f_n(t\wedge s_n^{(r+\epsilon)}) - f_\infty(t)}V \leq \lim_{n\to\infty} \sup_{0\leq t\leq T} \norm{f_n(t) - f_\infty(t)}V = 0. \qedhere\]
\end{proof}
We now consider the stochastic situation. Let $Y_n$, $n\in \bar \N:= \N\cup \{\infty\}$ be continuous stochastic processes on the Banach space $V$, with explosion times denoted by $\sigma_n$, $n\in \bar \N$. We do not require the processes to be adapted. For $r>0$ denote by $\varsigma^{(r)}_n$ and $\tau^{(r)}_n$, the first exit times of $Y_n$ of the respectively open and closed balls of radius $r$, for $n\in \bar\N$. More precisely
\begin{align*}
  \tau_n^{(r)} := \inf\{ t\geq 0 \,\vert \, \norm{Y_n(t)}{V} > r\},\quad   \varsigma_n^{(r)} := \inf\{ t\geq 0 \,\vert \, \norm{Y_n(t)}{V} \geq r\}.
\end{align*}
\begin{ass}\label{a:YnLoc}
  Assume that for all $n\in \bar\N$ there exist  stochastic processes $Y_n^{(r)}$, which are $\F_T$-measurable have almost surely paths in $C([0,T];V)$ and satisfy $Y_n^{(r)} = Y_n$ on $\llbrak 0, \tau^{(r)}_n\rrbrak$.
\end{ass}
\begin{ass}\label{a:YnConv}
  Assume that for all $r>0$ it holds that 
  \[ \lim_{n\to \infty} \sup_{0\leq t\leq T} \norm{Y_\infty^{(r)}(t) -Y_n^{(r)}(t)}{V} = 0,\quad\text{in probability.}\]
\end{ass}
\begin{rmk}\label{rmk:YnLoc}
  Note that $\varsigma_n^{(r)}$ are also the exit times of the processes $Y_n^{(r)}$. In particular, \autoref{a:YnLoc} yields that for all $n\in\bar\N$ and all $r'\geq r >0$,
  \[ Y_n^{(r)} = Y_n^{(r')},\quad \text{on } \llbrak 0,\tau_n^{(r)}\rrbrak.\]
\end{rmk}

The following result is proven also in~\cite[Thm 2.1]{kunzeApprox}, but was stated under the assumption that the stochastic processes are adapted. Note that this is not a restriction since we can pass over to the filtration generated by $Y_{n}$, $n\in \bar \N$, without any problems. However, we provide a path-wise and more direct proof which reduces the problem to the deterministic situation of the the previous lemmas. This will heavily rely on the subsequence criterion, which says that a sequence of random variables in a metric space converges in probability if and only if every subsequence has an almost surely convergent subsequence with the same limit, see e.\,g. \cite[Lemma 4.2]{kallenberg}.
\begin{prop}\label{prop:Stopping}
  Let Assumptions~\ref{a:YnLoc} and~\ref{a:YnConv} hold true. 
  \begin{enumerate}[label=(\roman*)]
  \item For all $r>0$, $\epsilon>0$, it holds almost surely that,
    \[ \liminf_{n\to\infty} \tau_n^{(r)} \leq \tau_\infty^{(r)} \leq \varsigma_\infty^{(r+\epsilon)} \leq \limsup_{n\to\infty} \varsigma_n^{(r+\epsilon)}.\]
    Moreover, for all $r>0$, $\epsilon_0>0$ and $(n_k)\subset \N$, with $n_k\to\infty$ as $k\to\infty$, there exists a subsequence $(n_{k_j})\subset (n_k)$ such that for all $\epsilon \in (0,\epsilon_0)$, almost surely,
    \begin{equation}
      \limsup_{j\to\infty} \tau_{n_{k_j}}^{(r)} \leq \tau_\infty^{(r)} \leq \varsigma_\infty^{(r+\epsilon)} \leq \liminf_{j\to\infty} \varsigma_{n_{k_j}}^{(r+\epsilon)}. \label{eq:taurapprox}
    \end{equation}
  \item Almost surely, 
    \[ \lim_{\Q\ni r\to\infty} \liminf_{n\to\infty} \tau_n^{(r)} \leq \sigma_\infty \leq \lim_{\Q\ni r\to\infty}\limsup_{n\to\infty} \varsigma_n^{(r)}.\]
    Moreover, for all $(n_k)\subset \N$, with $n_k\to\infty$ as $k\to\infty$, and $r>0$ denote by $\overline{\tau}^{(r)}$ and $\underline{\varsigma}^{(r)}$ respectively left and right hand side of~\eqref{eq:taurapprox}. Then, it holds with probability one,
    \begin{equation}\label{eq:tauinftyapprox} 
      \lim_{\Q\ni r\to\infty} \overline{\tau}^{(r)} = \lim_{\Q\ni r\to\infty}\underline{\varsigma}^{(r)} = \sigma_\infty.      
    \end{equation}
  \item For all $r>0$ and $\epsilon>0$ it holds that
    \[ \lim_{n\to \infty} Y_n(.\wedge \tau_\infty^{(r)}\wedge \varsigma_{n}^{(r+\epsilon)}) = Y_\infty(.\wedge \tau_\infty^{(r)}),\qquad \text{u.\,c.\,p.}\]
    Moreover, for all $r>0$
    \[ \lim_{n\to \infty} Y_n(.\wedge \tau_\infty^{(r)}) = Y_\infty(.\wedge \tau_\infty^{(r)}),\qquad \text{u.\,c.\,p.}\]   
  \end{enumerate}
\end{prop}
\begin{rmk}
  Part (iii) includes that $\tau_\infty^{(r)}< \sigma_n$, for large $n$ at least along subsequences, where the meaning of ``large'' will typically depend on $\omega$. 
\end{rmk}
\begin{proof}
  By u.\,c.\,p. convergence of $Y_n$ and subsequence criterion, each $(n_k)$ admits a subsequence $({n_{k_j}})$ such that almost surely,
  \[\lim_{j\to\infty}\sup_{0\leq t\leq T} \norm{ Y_\infty^{(r)} - Y_{n_{k_j}}^{(r)}}{V} = 0.\]
  Let $(n_{k_j})$ be a subsequence of $(n_k)$ and $\Omega' \subset \Omega$ be a set of full measure such that the convergence holds for $r'>r>0$ fixed, and such that the paths of all the processes considered in the following are continuous. We emphasize that these are at most countably many. Recall Remark~\ref{rmk:YnLoc}, which states that for all $\epsilon \in (0,r'-r]$, $\varsigma^{(r+\epsilon)}_n$ and $\tau_n^{(r)}$ are the exit times of $Y_n^{(r')}$. The second claim of (i) follows from \autoref{lem:exitdet} and~\eqref{eq:exitestdet}, applied to $Y_n^{(r')}$, $n\in \N$, for all $\omega \in \Omega'$. To verify the first claim,  recall that we have just shown that for $\delta>0$ and all $\omega \in \Omega'$, there exists a $J>0$ such that
  \[ \sup_{j\geq J} \tau_{n_{k_j}}^{(r)} < \tau_\infty^{(r)} + \delta.\]
  Hence, for all $N\in \N$
  \[\inf_{n\geq N} \tau_{n}^{(r)} \leq \inf_{n\geq (n_{k_J}\vee N)} \tau_{n}^{(r)} \leq \sup_{n\geq (n_{k_J}\vee N)} \tau_{n}^{(r)} \leq \sup_{j\geq J} \tau_{n_{k_j}}^{(r)} < \tau_\infty^{(r)} +\delta.\]
  We let $N\to \infty$ and since $\delta >0$ was chosen arbitrarily, the first claim in (i) holds on all of $\Omega'$. Note that the estimates for $\varsigma_{n}^{(r+\epsilon)}$ holds true due to the same arguments. \\
  
  We now directly apply \autoref{lem:convdet} to $Y^{(r')}_n(\omega,.)$, for all $\omega\in\Omega'$, and obtain
  \[ \lim_{j\to\infty}\sup_{0\leq t\leq T}\norm{Y_{n_{k_j}}^{(r')} (t\wedge \varsigma_{n_{k_j}}^{(r+\epsilon)}\wedge \tau_{\infty}^{(r)}) - Y_\infty^{(r')}(t\wedge \tau_{\infty}^{(r)})}V = 0.\]
  Recall that $Y_n^{(r')} = Y_n$ on $\llbrak 0,\varsigma^{(r+\epsilon)}_n\rrbrak$, for all $n\in \bar\N$, $\epsilon \in [0,r'-r]$ by assumption and monotonicity of the exit times. In fact, this finishes the proof of (iii). 

  Part (ii) is a direct application of the first. Let $\Omega'\subset \Omega$ be a set of full measure such that the first statement holds true for all $r\in (0,\infty)\cap \Q$. Hence, part (i) yields%App
  \begin{equation}
    \limsup_{n\to\infty} \varsigma_n^{(r)}(\omega) \geq \varsigma_\infty^{(r)}(\omega),\quad \forall r >0, r\in \Q, \label{eq:varsigmaninf}
  \end{equation}
  for all $\omega\in \Omega'$. Note that $\varsigma^{(r)}_n$ is increasing in $r>0$ and bounded by $T$ so that its limit exists. This yields the right hand side of the first claim and for the left hand side, we argue in the same way.
  
  For fixed $r>0$ denote by $\Omega_r\subset \Omega$ the set such that all the statements up to now hold true for fixed $r>0$ and set $\Omega' := \bigcap_{r>0,r\in\Q} \Omega_r$. Now, taking limits only along $\Q$,
  \[ \sigma_\infty = \lim_{r\to\infty} \tau_\infty^{(r)} \geq \limsup_{r\to\infty} \overline{\tau^{(r)}}  \geq   \liminf_{r\to\infty} \overline{\tau^{(r)}} \geq \sigma_\infty.\]
  For the latter estimate we just used that $\overline{\tau^{(r)}} \geq \tau^{(r-\epsilon)}_\infty$ for $\epsilon \in (0,r)$, \eqref{eq:taurapprox}. We can use the same argumentation for $\underline{\varsigma}^{(r)}$. 
\end{proof}
By iterative construction of subsequences, Kunze and van Neerven prove the following additional results, see~\cite[Thm 2.1.(3) and Cor 2.5]{kunzeApprox}.  As above, we can drop the adaptness assumption without any problems.
\begin{prop}\label{prop:Stopping2}
  In the situation of the previous proposition, for all $t>0$ it holds that
  \[ \lim_{n\to\infty} Y_{n}(t) \1_{\llbrak 0,\sigma_\infty\wedge \sigma_n\llbrak }(t;.) = Y_\infty(t) \1_{\llbrak 0,\sigma_\infty\llbrak}(t,.) \qquad\text{in probability.}\]
  If, in addition, $\sigma_n = T$ almost surely for all $n\in \bar\N$, then
  \[ \lim_{n\to \infty} \sup_{0\leq t\leq T} \norm{Y_\infty (t) -  Y_{n}(t)}{V}  = 0,\quad \text{in probability.}\]
\end{prop}
In the special situation where $\tau_\infty^{(r)} = \varsigma_\infty^{(r)} = 1$, we can get rid of the ``$\epsilon$'' in \autoref{prop:Stopping}, as a consequence of \autoref{lem:exitdet}.
\begin{prop}
  If in addition to \autoref{a:YnLoc} and~\ref{a:YnConv}, it holds that 
  \begin{equation}
    \label{eq:tausigma}
    \P{ \tau_\infty^{(r)} = \varsigma_\infty^{(r)}} = 1,    
  \end{equation}
  for some $r>0$, then, as limits in probability,
  \[\lim_{n\to\infty} \tau_n^{(r)} = \lim_{n\to\infty} \varsigma_n^{(r)} = \tau_\infty^{(r)}.\]
  In particular, uniformly on compacts in probability,
  \[\lim_{n\to\infty} Y_n(.\wedge \tau_n^{(r)}) = \lim_{n\to\infty} Y_n(.\wedge \varsigma_n^{(r)}) = Y_\infty(.\wedge \tau^{(r)}_\infty).\]
\end{prop}
\begin{rmk}
  Condition~\eqref{eq:tausigma} is essential since one can construct counterexamples in the deterministic case even when $V= \R$. 

  On the other hand, it is well-known that~\eqref{eq:tausigma} holds true for all $r>0$ when $Y_\infty$ is a real-valued Brownian motion. However, the situation becomes more delicate e.\,g. when $V$ is an infinite dimensional Hilbert space and $Y_\infty$ is the mild solution of a stochastic evolution equation which is not a strong solution on $V$.  
\end{rmk}
\begin{proof}
  As in the proof of \autoref{prop:Stopping}, fix $r' >r>0$ and let $(n_k)\subset \N$ such that
  \[ \lim_{k\to\infty}\sup_{0\leq t\leq T}\norm{Y^{(r)}_{n_k}(t)- Y_\infty^{(r)}(t)}{}+\sup_{0\leq t\leq T}\norm{Y^{(
        r')}_{n_k}(t)- Y_\infty^{(r')}(t)}{} = 0,\]
  almost surely. Let $\Omega'\subset \Omega$ be a set of full measure such that $\tau_\infty^{(r)} = \varsigma_\infty^{(r)}$ and the convergence and continuity of the involved processes hold true on all of $\Omega'$. Again, since for each $\epsilon <r'-r$, $\tau^{(r+\epsilon)}_{n}$ and $\varsigma^{(r+\epsilon)}_{n}$ are the exit times of $Y_{n}$, $n\in \bar\N$, we can now apply \autoref{lem:exitdet} to $Y^{(r')}$ to get the convergence for the exit times.
  
  Moreover, recall that $Y_n(t\wedge \tau_n^{(r)}) = Y_n^{(r)}(t\wedge \tau_n^{(r)})$ so that
  \begin{align*}
    \label{eq:49}
    \sup_{0\leq t\leq T}\norm{Y_\infty(t\wedge \tau_\infty^{(r)}) - Y_n(t\wedge \tau_n^{(r)})}{} &\leq \sup_{0\leq t\leq T}\norm{Y_\infty^{(r)}(t) - Y_n^{(r)}(t)}{} \\
                                                                                                 &\quad + \sup_{0\leq t\leq T}\norm{Y_\infty^{(r)} (t\wedge\tau_\infty^{(r)}) - Y_\infty^{(r)}(t\wedge \tau_n^{(r)})}{}.
  \end{align*}
  The first term vanishes as $n\to \infty$, almost surely, by \autoref{a:YnConv}. Switching to subsequences the latter one also vanishes almost surely, since $Y^{(r)}_\infty$ is uniformly continuous on the compact set $[0,T]$. We can apply the whole procedure to an arbitrary subsequence and thus, by subsequence criterion we obtain the convergence in probability.
\end{proof}
%%%%%%%%%%%%%%%%%%

\subsection{Localized Wong-Zakai Approximation}
\label{ssec:wzalocal}
We are now prepared to extend the Wong-Zakai approximation to the case where the solution of~\eqref{eq:eeq} might explode in finite time. For the localization, we will consider smooth truncation functions $h_r$, $r>0$. More precisely, we assume that $h_r\in C^\infty(\R_{\geq 0})$ is non-increasing, constant to $1$ on $[0,r^2]$ and constant to $0$ on $[(r+1)^2,\infty)$. Moreover, we assume for the first two derivatives $h'_r$ and $h''_r$ that
\begin{equation}
  \label{eq:32}
  \sup_{r>0} \left(\norm{ h_r'}{\infty} +  \norm{h_r''}{\infty}\right) \leq c <\infty.
\end{equation}

\begin{lem}\label{lem:truncfct}
  The map $\Xi_r: u\mapsto h_r(\norm{u}{\alpha}^2)$ is of class $C^2_b$ from $E_\alpha$ into $\R$. In addition, for $i\in \{1,2\}$,
  \[\supp (D^{i}\Xi_r) \subset \{u\in E_\alpha\vert  r\leq \norm{u}{\alpha} \leq r+1\}.\]
  In particular, $\Xi_r$ has global Lipschitz constant smaller than $2c(r+1)$.
\end{lem}
Here, we denote by $D^{i}$ the $i$th Fr\'echet derivative and will
write $D = D^1$ in the following. 
\begin{proof}
  Recall that $E_\alpha$ is a real Hilbert space, hence
  \[ D(\norm{.}{\alpha}^2)(u)v = 2\scal{u}{v}{\alpha}\]
  Indeed, for $u$, $v\in E_\alpha$ and $\epsilon>0$,
  \begin{equation*}
    \frac1{\epsilon} \left[\scal{u+\epsilon v}{u+\epsilon v}{\alpha}  - \scal{u}{u}{\alpha}\right]
    = 
    \scal{v}{u}{\alpha} + \scal{u}{v}{\alpha} + \epsilon\scal{v}{v}{\alpha} \rightarrow 2\scal{u}{v}{\alpha}.
  \end{equation*}
  Moreover, by bi-linearity of the scalar product,
  \[ D^2 (\norm{.}{\alpha}^2) (u)[ v,w] = 2\scal{v}{w}{\alpha},\quad u,\,v,\,w\in E_\alpha\]
  Using chain rule, this yields
  \begin{gather*}
    D\Xi_r(u)v = 2h_r'(\norm{u}{\alpha}^2)\scal{v}{u}{\alpha},\\
    D^2\Xi_r(u)[v,w] = 2 h_r'(\norm{u}{\alpha}^2) \scal{w}{v}{\alpha} + 4h_r''(\norm{u}{\alpha}^2)\scal{v}{u}{\alpha}\scal{w}{u}{\alpha}.
  \end{gather*}
 The statement on the support follows directly from definition of $h_r$. Mean-value theorem for Fr\'echet derivatives yields 
  \[ \norm{\Xi_r}{Lip} \leq \sup_{u\in E_\alpha} \norm{D\Xi_r(u)}{\Lbd(E_\alpha;\R)}\leq 2c(r+1).\qedhere \]
\end{proof}

\begin{lem}\label{lem:truncLip}
  Let $E$ and $V$ be Banach spaces, $\Phi : V\to E$ Lipschitz continuous on bounded sets, and $h:V\to\R$ Lipschitz continuous with Lipschitz constant $L_h$ and 
  \[\supp(h) \subset \{ u\in V\,\vert\, \norm{u}{V}\leq r\} =: \bar K_r,\]
  for some $r\in\N$. Denote by $L_r$ the Lipschitz constant of $\Phi$ on $\bar K_r$. The map $u\mapsto h(u)\Phi(u)$ is globally bounded and Lipschitz continuous from $V$ into $E$ with Lipschitz constant
  \begin{equation}
    \label{eq:truncLipNorm}
     \norm{h\Phi}{Lip(V;E)} \leq  L_r \sup_{w\in V} \abs{h(v)}  + L_h \sup_{w\in \bar K_r} \norm{\Phi(w)}{V}.    
  \end{equation}
\end{lem}
\begin{proof}
  Let $u$, $v\in V$. W.\,l.\,o.\,g. assume that $\norm{v}{V}\leq \norm{u}{V}$ and write
  \[ h(u)\Phi(u) - h(u)\Phi(v) = h(u) (\Phi(u) - \Phi(v)) + \Phi(v) (h(u)- h(v)).\]
  The first term vanishes when $\norm{u}{V}\geq r$. Else, we know that $\norm{v}{V}\leq \norm{u}{V}< r$ by assumption, so that
  \begin{equation}
    \norm{h(u) (\Phi(u) - \Phi(v))}{E} \leq \sup_{w\in \bar K_r} \abs{h(w)} L_r \norm{u -v}{V}.
  \end{equation}
  The second summand vanishes when $\norm{v}{V}\geq r$. Else, Lipschitz continuity of $h$ yields
  \[ \norm{\Phi(v)}{E}\abs{h(u)-h(v)} \leq \norm{h}{Lip(V;\R)} \sup_{w\in \bar K_r} \norm{\Phi(w)}{E} \norm{u-v}{V}.\qedhere\]
\end{proof}

Let us now make the assumptions we will impose on the coefficients more precise. Recall that we have $\alpha \in [0,1)$, fixed and so, $E_\alpha$ will take the part of $V$ in the previous lemma.
\begin{ass}%
  \label{a:Ballloc}
  $B:E_\alpha\to E_0$ is Lipschitz continuous on bounded sets.
\end{ass}

\begin{ass}%
  $\;$ \label{a:Cloc}
  \begin{enumerate}[label=($C$.\roman*)]
  \item $C: E_\alpha\to \HS(U;E_\alpha)$ is Lipschitz continuous on bounded sets
  \item $\sigma_k:E_\alpha\to E_\alpha$, $k\in \N$ are twice Fr\'echet
    differentiable. $\sigma_k$ and its derivatives map bounded sets
    into bounded sets, for each $k\in \N$.
  \item For all $N\in \N$ there exists an $L_\Sigma^{(N)} >0$ such that for all $u,v\in E_\alpha$ with $\norm{u}{\alpha}$, $\norm{v}{\alpha} \leq N$,
    \[ \norm{\Sigma_n(u) - \Sigma_n(v)}{E_\alpha} \leq L_\Sigma\norm{u-v}{E_\alpha} ,\;\forall\, u,v\in E_\alpha,\,\forall n\in\N\].
  \item There exists $\Sigma_\infty:E_\alpha\rightarrow E_\alpha$ such that for all $u\in E_\alpha$,
\[ \lim_{n\to\infty} \Sigma_n (u) =\Sigma_\infty(u).\]
  \end{enumerate}
\end{ass}

%%%%%%%%%%%%%%%%%%%%%%%%%%%%%%%%%%%%%%%%%%%%%%% 

The truncated coefficients are then defined as
\[ B^{(r)}(u) := h_r(\norm{u}{\alpha}^2) B(u),\qquad C^{(r)}(u):= h_r(\norm{u}{\alpha}^2) C(u),\quad u\in E_\alpha.\]
We now use the notation introduced in the beginning of Section~\ref{sec:wza}, in particular, 
\[ \sigma^{(r)}_k (u) := C^{(r)}(u)e_k,\qquad \Sigma^{(r)}_n (u) := \frac12 \sum_{k=1}^n D\sigma_k^{(r)}(u) \sigma_k^{(r)}(u).\]
for $u\in E_\alpha$. Indeed, $B^{(r)}$ and $C^{(r)}$ do fulfill the assumptions we have imposed formerly.
\begin{lem}\label{lem:trunc}
  For all $r>0$, $B^{(r)}$ and $C^{(r)}$ fulfill \autoref{a:Ball} and~\ref{a:Call}, respectively. 
\end{lem}
\begin{proof}
  First, let us emphasize that \autoref{lem:truncLip} yields global  boundedness and Lipschitz continuity of $B^{(r)}$ and $C^{(r)}$. The Fr\'echet differentiability of $\sigma_k^{(r)}$ and its derivative is inherited from differentiability of $\sigma_k$, $D\sigma_k$ and by application of \autoref{lem:truncfct}. This also yields boundedness of $\sigma_k$ and its first two derivatives. 
  
  It remains to verify the properties of $\Sigma_n^{(r)}$. For all $u\in E_\alpha$,
  \begin{gather}
    \label{eq:SigmanN}
    \begin{split}
    \Sigma_n^{(r)}(u) &= \sum_{k=1}^n(D\Xi_r(u)\sigma_k^{(r)}(u)) \sigma_k(u) + \sum_{k=1}^n \Xi_r(u) D\sigma_k(u)\sigma_k^{(r)}(u) \\
    &= 2 h_r(\norm{u}{\alpha}^2) h_r' (\norm{u}{\alpha}^2)  \sum_{k=1}^n \scal{\sigma_k(u)}{u}{\alpha} \sigma_k(u) + h_r(\norm{u}{\alpha}^2)^2  \Sigma_n(u).      
    \end{split}
  \end{gather}
  To see that the first summation in the last equation converges as $n\to \infty$, we apply triangle and Cauchy-Schwarz inequality,
  \begin{equation}
    \label{eq:33}
    \norm{ \sum_{k=1}^n \scal{\sigma_k(u)}{u}{\alpha} \sigma_k(u) }{\alpha} \leq  \norm{u}{\alpha} \sum_{k=1}^n\norm{\sigma_k(u)}{\alpha}^2.
  \end{equation}
  In fact, since $C(u)\in \HS(U;E_\alpha)$, the sequence $(\sigma_k(u))_k$ is square summable in $E_\alpha$, and we can define 
  \begin{equation}
    \label{eq:19}
    \Sigma_\infty^{(r)}(u) := 2 h_r(\norm{u}{\alpha}^2) h_r' (\norm{u}{\alpha}^2)  \sum_{k=1}^\infty \scal{\sigma_k(u)}{u}{\alpha} \sigma_k(u) + h_r(\norm{u}{\alpha}^2)^2  \Sigma_\infty(u).    
  \end{equation}
  We obtain $\Sigma_n^{(r)}\to \Sigma_\infty^{(r)}$ strongly in $E_\alpha$ by assumptions on $\Sigma_n$ and the same estimates as in~\eqref{eq:33}, more detailed
  \begin{multline}
    \label{eq:30}
    \norm{\Sigma_\infty^{(r)}(u) - \Sigma_n^{(r)}(u)}{\alpha} \leq \norm{\Sigma_\infty - \Sigma_n(u)}{\alpha} +\sum_{k=n+1}^\infty \abs{\scal{\sigma_k(u)}{u}{\alpha}} \norm{\sigma_k(u)}{\alpha} \\
    \leq  \norm{\Sigma_\infty(u) - \Sigma_n(u)}{\alpha} + \norm{u}{\alpha} \sum_{k=n+1}^\infty  \norm{\sigma_k(u)}{\alpha}^2 \longrightarrow 0. 
  \end{multline}
  To prove Lipschitz continuity of $\Sigma_n^{(r)}$ we go back into~\eqref{eq:SigmanN} and first apply \autoref{lem:truncLip} to the last summand. For the remaining part, we decompose for $u$, $v\in E_\alpha$, and $n\in\N$, 
  \begin{align*}
    \scal{\sigma_k(u)}{u}{\alpha} \sigma_k(u) - \scal{\sigma_k(v)}{v}{\alpha}\sigma_k(v) 
    &= \scal{\sigma_k(u)}{u-v}{\alpha} \sigma_k(u)\\
    &\quad + \scal{\sigma_k(u)-\sigma_k(v)}{v}{\alpha} \sigma_k(u) \\
    &\quad + \scal{\sigma_k(v)}{v}{\alpha} \left(\sigma_k(u)-\sigma_k(v)\right)\\
    &=: R_1^k + R_2^k + R_3^k.
  \end{align*}
  With the same estimates as for~\eqref{eq:33} we get
  \[\sum_{k=1}^n \norm{R_1^k}{\alpha} \leq  \norm{u-v}{\alpha} \norm{C(u)}{\HS(U;E_\alpha)}^2,\]
  which is independent of $n\in\N$. For the remaining part, we apply Cauchy-Schwarz inequality on $E_\alpha$ and on $\R^n$, to get
  \begin{gather*}
    \begin{split}
    \sum_{k=1}^n \norm{R_2^k}{\alpha} + \norm{R_3^k}{\alpha}
    &\leq\norm{v}{\alpha}\sum_{k=1}^n \norm{\sigma_k(u)-\sigma_k(v)}{\alpha} \left(\norm{\sigma_k(u)}{\alpha}+\norm{\sigma_k(v)}{\alpha}\right)\\
    &\leq  \norm{v}{\alpha} \sqrt{\sum_{k=1}^n \norm{\sigma_k(u)-\sigma_k(v)}{\alpha}^2} \left(\sqrt{\sum_{k=1}^n  \norm{\sigma_k(u)}{\alpha}^2}  + \sqrt{\sum_{k=1}^n  \norm{\sigma_k(v)}{\alpha}^2}\right)\\
    &\leq \norm{v}{\alpha} \left(\norm{C(u)}{\HS(U; E_\alpha)} + \norm{C(v)}{\HS(U;E_\alpha)}\right) \norm{C(u)-C(v)}{\HS(U;E_\alpha)}.    
    \end{split}
  \end{gather*}
  In other words, we have Lipschitz continuity on bounded sets, with local Lipschitz constants independent of $n\in \N$. By \autoref{lem:truncLip} and~\eqref{eq:SigmanN}, this yields that $\Sigma_n^{(r)}$, $n\in\N$, are globally Lipschitz with uniform Lipschitz constant. 
\end{proof}

We denote by $X^{(r)}$ the unique global mild solution on $E_\alpha$, of the truncated equation,
\begin{equation}
  \label{eq:eeqloc}
  \d X^{(r)}(t) = \left[ A X^{(r)}(t) + B^{(r)}(X^{(r)}(t))\right] \d t + C^{(r)}(X^{(r)}(t)) \d W_t,
\end{equation}
 with intial condition $X^{(r)}(0) = X_0$. Respectively, $X_n^{(r)}$ and $Z_{m,n}^{(r)}$, for $n$, $m\in \N$, $r>0$, denote the solutions of the localized approximating equations
\begin{equation} \label{eq:eeqr}
  \left\{ \begin{aligned}
      \d X^{(r)}_n(t) &= \left[A X^{(r)}_n(t) + B_n^{(r)}(X^{(r)}(t)) \right] \d t+ \sum_{k=1}^n \sigma^{(r)}_k(X^{(r)}_n(t)) \d \beta^k(t),\\
      X^{(r)}_n(0) &= X_0,
    \end{aligned}\right.
\end{equation}
and $\omega$-wise,
\begin{equation} \label{eq:eeqWZA}
  \left\{ \begin{aligned}
      \ddt {Z}^{(r)}_{m,n}(t) &= A Z^{(r)}_{m,n}(t) + B^{(r)}(Z^{(r)}_{m,n}(t)) - \Sigma^{(r)}_\infty(Z^{(r)}_{m,n}(t)) \\
      &\qquad\qquad\qquad\qquad\qquad\qquad+ \sum_{k=1}^n \sigma^{(r)}_k(Z^{(r)}_{m,n}(t)) \dot{\beta}^k_m(t),\;t> 0,\\
      Z^{(r)}_{m,n}(0) &= X_0.
    \end{aligned}\right.
\end{equation}
Here $B_n^{(r)}:= B^{(r)}- \Sigma_n^{(r)} + \Sigma_\infty^{(r)}$. The latter one has been defined in~\eqref{eq:19}.
%%%%%%%%%%%%%%%%%%%%%%%%%%%%%%%%%%%%%%%%%%%%%%% 

With $X$, $X_n$ and $Z_{m,n}$, for $m$, $n\in\N$ we denote the unique maximal solutions of the non-truncated equations \eqref{eq:eeq}, \eqref{eq:SEEnoiseapprox} and \eqref{eq:eeqWZAbd}, respectively. Moreover, by $\tau$, $\tau_n$ and $\tau_{m,n}$ we denote their $E_\alpha$-explosion times, and for $r>0$,
\begin{align*}
  \allowdisplaybreaks
  \tau^{(r)} &:=  \inf\{ t\geq 0\,\vert\,  t<\tau,\,\norm{X (t)}{\alpha} > r \} ,\\
  \tau_{n}^{(r)}&:= \inf\{ t\geq 0\,\vert\,  t<\tau_{n},\,\norm{X_n(t)}{\alpha} > r \},\\
  \tau_{m,n}^{(r)}&:= \inf\{ t\geq 0\,\vert\,  t<\tau_{m,n},\,\norm{Z_{m,n(t)}}{\alpha} > r \}.
\end{align*}

In the same way we define the exit times of the open balls
\begin{align*}
  \allowdisplaybreaks
  \varsigma^{(r)} &:= \inf\{ t\geq 0\,\vert\, \norm{X^{(r)}(t)}{\alpha} \geq r \} = \inf\{ t\geq 0\,\vert\, t<\tau,\, \norm{X (t)}{\alpha} \geq r \}, \\
  \varsigma_{n}^{(r)}&:= \inf\{ t\geq 0\,\vert\, \norm{X_{n}^{(r)}(t)}{\alpha} \geq r \} = \inf\{ t\geq 0\,\vert\,  t<\tau_{n},\,\norm{X_{n}(t)}{\alpha} \geq r \},\\
  \varsigma_{m,n}^{(r)}&:= \inf\{ t\geq 0\,\vert\, \norm{Z_{m,n}^{(r)}(t)}{\alpha} \geq r \} = \inf\{ t\geq 0\,\vert\,  t<\tau_{m,n},\,\norm{Z_{m,n}(t)}{\alpha} \geq r \}.
\end{align*}
\begin{rmk}
  $Z_{m,n}^{(r)}$ and $\tau_{m,n}^{(r)}$ can be seen as $\F_T$-measurable functions of $\omega\in \Omega$, but $Z_{m,n}^{(r)}$ is \emph{not} adapted and hence, $\tau_{m,n}^{(r)}$ is not a stopping time. 
\end{rmk}

The following result extends \cite[Thm 2.1 and Prop 7.3]{nakayamasupport} and seems to be new even when $\alpha = 0$.
\begin{thm}
  \label{thm:WZAlocal}
  Let Assumptions~\ref{a:A}, \ref{a:Ballloc} and \ref{a:Cloc} hold true. For all $r>0$, $\epsilon>0$, it holds that
  \[ \lim_{n\to\infty}\lim_{m\to\infty} \E{\sup_{0\leq t\leq T}\norm{X(.\wedge \tau^{(r)})-Z_{m,n}(.\wedge \varsigma_{m,n}^{(r+\epsilon)}\wedge \tau^{(r)}) }{\alpha}^{2p}} = 0\]
\end{thm}
\begin{proof}
  Thanks to \autoref{lem:trunc} we can apply \autoref{thm:WZA} and observe that the unique solution $X^{(r)}$ of the localized equation can be approximated first by the solutions of equations with finite dimensional noise $X^{(r)}_n$ and then by Wong-Zakai approximations $Z_{m,n}^{(r)}$. By uniqueness claims of the existence results, we get
  \[ X^{(r)} = X^{(r')},\; \text{on }\llbrak 0,\tau^{(r)}\rrbrak,\quad X_n^{(r)} = X_n^{(r')},\; \text{on }\llbrak 0,\tau_n^{(r)}\rrbrak, \quad Z_{m,n}^{(r)} = Z_{m,n}^{(r')},\; \text{on }\llbrak 0,\tau_{m,n}^{(r)}\rrbrak,\]
  for $r'>r$, Now, the assumptions of \autoref{prop:Stopping} are fulfilled respectively for $Z_{m,n}$ and $X_n$, and for $X_n$ and $X$. For $\epsilon'\in (0,\epsilon)$, triangle inequality yields
  \begin{multline}
    \label{eq:27}
    \sup_{0\leq t\leq T}\norm{X(.\wedge \tau^{(r)})-Z_{m,n}(.\wedge \varsigma_{m,n}^{(r+\epsilon)}\wedge \tau^{(r)}) }{\alpha} \leq \\
    \leq \sup_{0\leq t\leq T}\norm{X(.\wedge \tau^{(r)})-X_{n}(.\wedge \tau_{n}^{(r+\epsilon')}\wedge \tau^{(r)}) }{\alpha}  + \\
    +\sup_{0\leq t\leq T}\norm{X_n(.\wedge \tau^{(r)}\wedge \tau^{(r+\epsilon')}_n)-Z_{m,n}(.\wedge \varsigma_{m,n}^{(r+\epsilon)}\wedge \tau^{(r)}) }{\alpha}.
  \end{multline}
  For every subsequence in $n$ there exists a further subsubsequence such that the first summand on the right hand side vanishes as $n\to\infty$ by \autoref{prop:Stopping}. Applying Lebesgue's Dominated Convergence Theorem, the convergence holds true also in $L^{2p}$, along the subsubsequences. 
  
  For the second term, we take converging subsequences $(X_{n_k})$ and $(u_{m^k_l, n_k})$ such that on a set of full measure $\Omega'\subset \Omega$, uniformly in $[0,T]$,
  \begin{alignat*}{2}
    X_{n_k}^{(r)}&\to X^{(r)},\qquad &X_{n_k}^{(r')}&\to X^{(r)'},\quad \text{ as } k\to\infty\\
    u_{m^k_l,n_k}^{(r)}&\to X_{n_k}^{(r)}, & u_{m_l^k, n_k}^{(r')} &\to X_{n_k}^{(r')}, \quad \text{ as } l\to\infty.
  \end{alignat*}
  Applying the proof of \autoref{lem:convdet} to $X_{n_k}^{(r')}$, there exists $k_0\in \N$ such that for all $k>k_0$ it holds that
  \begin{equation}
    \label{eq:44}
    \tau_{n_k}^{(r+\epsilon')} \geq \tau^{(r)}.    
  \end{equation}
  Therefore, the second term on the right hand side of~\eqref{eq:27} can be estimated by  
  \[\sup_{0\leq t\leq T}\norm{X_{n_k}(.\wedge \tau^{(r+\epsilon')}_{n_k})-u_{m_l^k,n_k}(.\wedge \varsigma_{m_l^k,n_k}^{(r+\epsilon)}\wedge \tau_{n_k}^{(r+\epsilon')}) }{\alpha},\qquad \forall\,k>k_0.\]
  This term goes to $0$ as $l\to\infty$, and dominated convergence yields $L^{2p}$-convergence. Let us be more precise, 
  \begin{multline*}
    \EE \sup_{0\leq t\leq T}\norm{X_{n_k}(.\wedge \tau^{(r+\epsilon')}_{n_k} \wedge \tau^{(r)} )-u_{m_l^k,n_k}(.\wedge \varsigma_{m_l^k,n_k}^{(r+\epsilon)}\wedge \tau_{n_k}^{(r)} \wedge \tau^{(r)}) }{\alpha}^{2p} \leq  \\
    \leq \E{\sup_{0\leq t\leq T}\norm{X_{n_k}(.\wedge \tau^{(r+\epsilon')}_{n_k})-u_{m_l^k,n_k}(.\wedge \varsigma_{m_l^k,n_k}^{(r+\epsilon)}\wedge \tau_{n_k}^{(r+ \epsilon')}) }{\alpha}^{2p}\1_{\tau_{n_k}^{(r+\epsilon)} \geq \tau^{(r)}}}\\
    + K_{\alpha,p} (2r + \epsilon' + \epsilon)^{2p}\, \P{\tau_{n_k}^{(r+\epsilon')} < \tau^{(r)}}.
  \end{multline*}
  The first term goes to $0$, for $l\to \infty$, as we have discussed above. The second one does, as $k\to\infty$ owing to~\eqref{eq:44}. For each subsequence we can find again such subsubsequences so that $L^{2p}$-convergence holds true. But since $L^{2p}$-convergence is metrizable, convergence is equivalent to the statement that every subsequence admits a subsubsequence converging to the same limit, which we have just shown.
\end{proof}

We close this section with a result for the case where linear growth
of $B$ and $C$ holds true, but not necessarily global Lipschitz
continuity and boundedness. In this case, the solutions still exist
globally, cf. \cite[Theorem 3.20]{SFBPDir}.
\begin{thm}\label{thm:WZAling}
  Assume that Assumptions~\ref{a:A}, \ref{a:Ballloc} and \ref{a:Cloc} hold true and, in addition, that there exists an $M>0$ such that
  \begin{equation}
    \label{eq:50}
    \norm{B(u)}{0} + \norm{C(u)}{\HS(U; E_\alpha)} \leq M(1+ \norm{u}{\alpha}),\quad \forall\,u\in E_\alpha,
  \end{equation}
  and $D\sigma_k$ is globally bounded. Then, uniformly on compacts in probability,
  \[ \lim_{n\to\infty} \lim_{m\to\infty} Z_{m,n} = X.\]
\end{thm}
\begin{proof}
  First note that under the given constraints the solutions of~\eqref{eq:eeqloc},~\eqref{eq:eeqr} and~\eqref{eq:eeqWZA} exist globally and take values almost surely in $C([0,T]; E_\alpha)$, cf. \cite[Theorem 3.20]{SFBPDir}. We have shown already that the assumptions of \autoref{prop:Stopping2} are fulfilled which finally yields the convergence claim.
\end{proof}

%%% Local Variables: 
%%% mode: latex
%%% TeX-master: "0paper"
%%% End: 

\section[Forward Invariance]{Forward Invariance for Deterministic Equations}
\label{sec:inv}

Consider a separable real Banach space $E$ and the semilinear (deterministic) evolution equation
\begin{equation}
  \label{eq:eeqdet}
  \ddt u(t) = A u(t) + B(u(t)),\quad t\in [0,T],\qquad u(0) = u_0.
\end{equation}
where $A:\dom(A)\subset E\to E$ is a linear operator on $E$ and $B: E\to E$ is Borel measurable. In this section, we discuss conditions under which a closed set $\MM\subset E$ is forward invariant for~\eqref{eq:eeqdet}, i.\,e. $u_0\in \MM$ yields that the (local) solution $u$ takes values in $\MM$. 

Now, assume that $u:[0,T] \to E$ is differentiable at $0$ with $\dot u (0) = g$ and $u(0) = u_0$, i.\,e.
\[ u(t) = u_0 + t g + o(t),\quad t>0.\]
Hence, a necessary condition that there exists an $\epsilon>0$ such that $u(t) \in \MM$, for all $t\in [0,\epsilon)$ is given by,
\begin{equation}
  \dist_E(u_0 + t g; \MM) = o(t). \label{eq:nagumoVec}
\end{equation}
Conversely, when $\dot u = F\circ u$ for a function $F:E\to E$ the condition
\begin{equation}
  \forall u_0\in \MM:\quad \lim_{t\searrow 0} \tfrac1{t} \dist_E(u_0 + t F(u_0); \MM) = 0, \label{eq:nagumo}
\end{equation}
is called Nagumo- or tangency-condition and is well-known to be also
sufficient in many cases. 
Moreover, if $\MM$ is closed convex and replacing $\lim$ by $\liminf$
then \eqref{eq:nagumo} becomes
equivalent to the condition that for all $\phi \in E^*$ such that
$\phi (h) = \inf_{f\in \MM} \phi(f)$, it holds that $\phi(g) \geq 0$,
see \cite[Lemma 4.1]{deimling}. For a detailed discussion also in
connection to the geometry behind the Nagumo
condition~\eqref{eq:nagumo} we refer to~\cite{pavel1984differential}.

We go back to the theory of evolution equations, where
\[ F(u) := Au + B(u).\]
We additionally know that $A$ generates a strongly continuous semigroup $(S_t)$ in the applications we are interested in. From \cite[Section 4.1]{pavel1984differential} we extract the following.
\begin{prop}\label{prop:Stangential}
  Let $A$ be the generator of a strongly continuous semigroup $(S_t)$ on a Banach space $E$. For a closed subset $\MM\subset E$, $u_0\in \MM\cap \dom(A)$ and an element $v\in E$, the following tangential conditions are equivalent,
  \begin{enumerate}[label=(\roman*)]
  \item  $\lim_{t\searrow 0} \frac1t \dist_E(u_0 + t(Au_0 + v); \MM) = 0$, 
  \item $\lim_{t\searrow 0} \frac1t \dist_E(S_t u_0 + tv; \MM) = 0$.
  \end{enumerate}
\end{prop}
Moreover, the so called tangential points, which are the points for which (i) or (ii) are satisfied, can be identified in the following way.
\begin{prop}[{\cite[Prop 4.1.4]{pavel1984differential}}]
  Let $A$ be the generator of a strongly continuous semigroup $(S_t)$ on a Banach space $E$ and let $\MM\subset E$ be closed. Moreover, assume that $S_t(\MM)\subset \MM$. If $v\in E$ fulfills
  \[\lim_{t\searrow 0}\frac1t \dist_E(u_0 +tv;\MM) = 0,\qquad \forall u_0\in \MM,\]
  for some $u_0 \in \MM$, then $v$ also satisfies the second statement in Proposition~\ref{prop:Stangential}.
\end{prop}
This allows to separate the tangential conditions for $A$ and $B$. In fact, if $B$ satisfies~\eqref{eq:nagumo} for $u_0\in \MM$ and $S_t\MM \subset \MM$, then it also holds that
\[\lim_{t\searrow 0} \frac1t \dist_E(S_t u_0 + t B(u_0); \MM) = 0.\]
The converse direction does not hold, in general. However, in some special situation as when $B:E\to E$ is Lipschitz continuous, the latter conditions is known to be necessary and sufficient for forward invariance for evolution equations. We refer to~\cite[Chapter 4]{pavel1984differential} for a detailed discussion and proofs. 

The equations discussed in the previous sections are beyond the scope of these results. As in the previous sections, $B$ will be only continuous on a certain subspace of $E$. 

Recall from \autoref{sec:Pre} that when $A$ is the generator of an analytic $C_0$-semigroup of negative type, then we have defined the inter- and extrapolation spaces $E_\alpha$, $\alpha \in \R$ and for $\alpha \in (0,1)$ we have
\[ \dom(A) = E_1 \hookrightarrow E_\alpha \hookrightarrow E_0 = E,\]
Thus, Kuratowski's Theorem yields that $E_\alpha$ is a Borel subset of $E$ for $\alpha>0$. In the sequel, when $\MM\subset E$ we will use the notation $\MM_\alpha := \MM \cap E_\alpha$, $\alpha \geq 0$.

\begin{ass}\label{a:pruess}
  \begin{itemize}
  \item[(A)] $A$ is the generator of an analytic $C_0$-semigroup of negative type, denoted by $(S_t)$,
  \item[(B)] $B: \MM_\alpha \to E$ is Lipschitz continuous, for some $\alpha <1$,
  \item[(M)] $\MM\subset E$ is closed,
  \item[(N)] Assume that $S_t (\MM) \subset \MM$ and the so called Nagumo condition is satisfied, that is  
    \begin{equation*}
      \lim_{\epsilon \searrow 0} \tfrac1{\epsilon} \dist_E( u +\epsilon B(t,u); \MM) = 0,
    \end{equation*}
    for all $u\in \MM_\alpha$ and $t\in [0,T]$.
  \end{itemize}
\end{ass}
Note that also $\MM_\alpha$ is closed as a subset of $E_\alpha$ under Assumption~\ref{a:pruess}.

An established way to prove existence and forward invariance results in such a setting, but under weaker constraints on $B$, is the concept of $\epsilon$-approximate solutions, see \cite{pavel1984differential} or \cite{pruessInvariant} for instance. In order to construct these approximations and the solution one often uses compactness of the semigroup $(S_t)$, which we will not have in the situation of \autoref{sec:positivityresults}. Instead, we use a result of Pruess~\cite{pruessInvariant} in this direction relying estimates of the non-compactness of $B$. 
\begin{lem}\label{lem:Hausdorffm}
  Assume that $B: \MM_\alpha\to E$ is Lipschitz continuous with Lipschitz constant $L>0$, then for all Borel sets $G\subset \MM_\alpha $ it holds that 
  \[ \nu_{E}(B(G)) \leq  2L \nu_{E_\alpha}(G),\]
  where $\nu$ is the Hausdorff measure of non-compactness defined for $G\subset E$ as
  \[ \nu_E(G) := \inf\setc{r> 0}{G\text{ admits a finite covering of balls (in $E$) with radius $r$}}.\]
\end{lem}
\begin{proof}
  It is easy to see that for $K_{\alpha}(u,r)$, the $E_\alpha$-ball of radius $r>0$, centered at $u\in \MM_\alpha$ it holds that
  \begin{equation}
    \label{eq:45}
    B(K_\alpha(u,r)\cap \MM_\alpha) \subset K_0(B(u), L r),    
  \end{equation}
  where $L$ is the Lipschitz constant of $B$. %This directly extends to general Borel sets $G\subset \MM_\alpha$.
  Let $G\subset \MM_\alpha$ be a Borel set with $\nu_{E_\alpha}(G) = r< \infty$. For $r'>r$ let $n\in \N$ and $u_1$,...$u_n\in E_\alpha$ such that 
  \[G\subset \bigcup_{k=1}^n K_\alpha(u_k,r').\]
  Now, choose arbitrary $\tilde u_k \in \MM_\alpha \cap K_\alpha(u_k,r')$, $k=1,...,n$, and thus  
  \( G\subset \bigcup_{k=1}^n K_\alpha(\tilde u_k, 2r').\)
  With~\eqref{eq:45} we get 
  \begin{multline*}
    \label{eq:48}
    B(G) \subset B\left( \bigcup_{k=1}^n K_{\alpha}(\tilde u_k, 2r') \cap \MM_\alpha\right) =\\
    =  \bigcup_{k=1}^n B(K_\alpha(\tilde u_k; 2r')\cap \MM_\alpha) \subset \bigcup_{k=1}^n K_0(B(\tilde u_k), 2Lr').\qedhere
  \end{multline*}
\end{proof}

\begin{thm}[Pruess {\cite{pruessInvariant}}]\label{thm:viabPruess}
  Assume that Assumption~\ref{a:pruess} holds true for some $\alpha\in [0,1)$ and let $\gamma\in (\alpha,1)$. For all $u_0\in \MM_{\gamma}$ there exists a unique mild solution $u:[0,T]\to \MM$ of the evolution equation
  \begin{equation}
    \label{eq:13}
    \dot u(t) = Au(t) + B(u(t)),\qquad u(0) = u_0.    
  \end{equation}
  Moreover, $u$ is $E_{\gamma}$-continuous on $[0,T]$ and $E$-continuously differentiable on $(0,T]$.
\end{thm}
Let us collect the respective results from~\cite{pruessInvariant}. First note that Assumption~\ref{a:pruess} is sufficient for Assumptions (A), ($\Omega$), (Y), (F), (S) and (L) in~\cite{pruessInvariant}. Thanks to the estimate from Lemma~\ref{lem:Hausdorffm} we can apply~\cite[Theorem 2]{pruessInvariant} which yields local existence. Moreover, for any local solution $u$ it holds that
\begin{gather}
  \begin{split}
    \norm{u(t)}{\alpha} &\leq K \norm{u_0}{\alpha} + K_\alpha \int_0^t \norm{B(u(s))}{E} \tfrac{\d s}{(t-s)^\alpha} \\
    &\leq K \norm{u_0}{\alpha} + K_T \left(1+\int_0^t\norm{u(s)}{\alpha} \tfrac{\d s}{(t-s)^\alpha}\right).      
  \end{split}\label{eq:13a}
\end{gather}
Indeed, by Lipschitz continuity of $B$ there exists $M>0$ such that 
\[ \norm{S_t B(x)}{\alpha}  \leq K_{E,\alpha} t^{-\alpha} M \left(1 + \norm{x}{\alpha} \right).\]
Hence, by Lemma~\ref{lem:gronwall} 
\[\norm{u(t)}{\alpha} \leq K \left(1+ \norm{u_0}{\alpha}\right),\] 
for all $t\in[0,T]$ such that $u(t)$ solves~\eqref{eq:13}. Assuming that $u$ would be a non-continuable mild solution would now contradict \cite[Theorem 4]{pruessInvariant}, and thus, $u$ is a global mild solution. The last statement is then a consequence of the regularity theorem in~\cite{pruessInvariant}. Finally, by global Lipschitz assumptions on $B$ we get uniqueness of the solution.

\begin{rmk}
  In fact, $u$ is even a mild solution in $E_{\gamma}$, since $u_0\in E_{\gamma}$ and by \autoref{lem:StHa},
  \[ \int_0^t \norm{(-A)^{\gamma} S_{t-s} B(u(s))}{} \d s \leq M \int_0^t (t-s)^{-(\gamma- \alpha)} (1+ \norm{u(s)}{\alpha}) \d s <\infty.\]
\end{rmk}
\begin{rmk}
  By concatenation, the existence result extends to $[0,\infty)$ without further effort. 
\end{rmk}
We now replace the global Lipschitz assumption by the local one
\begin{equation}
  \label{a:invbloc}
  B:\MM_\alpha \to E \text{ is Lipschitz continuous on bounded sets}.   \tag{$B_{loc}$}
\end{equation}

\begin{thm}\label{thm:viabPruesslocal}
  Assume that \autoref{a:pruess} holds true for some $\alpha\in [0,1)$ with $(B)$ replaced by~\eqref{a:invbloc}, and fix $\gamma\in (\alpha,1)$. For all $u_0\in \MM_\gamma$, there exists a maximal $t_\infty>0$ and a unique local mild solution on $[0,t_\infty)$ of the evolution equation~\eqref{eq:13}. In particular, $u(t)\in \MM_\gamma$ for all $t\in [0,t_\infty)$. Moreover, $u$ is $E_{\gamma}$-continuous on $[0,t_\infty)$, $E$-continuously differentiable on $(0,t_\infty)$ and it holds that $t_\infty = \infty$ or 
  \[ \lim_{t\nearrow t_\infty} \norm{u(t)}{\alpha} = \infty.\]
\end{thm}
\begin{proof}
  For $N\in \N$ let $h_N\in C^\infty([0,\infty); \R)$ be a non-increasing function such that $h=1$ on $[0,N]$ and $h=0$ on $[N+1,\infty)$. Moreover, assume that 
  \[\sup_{N\in \N} \norm{h_N'}{\infty} <\infty.\]
  We now get that $B_N(u):= h_N(\norm{u}{\alpha}) B(u)$ is globally Lipschitz continuous. To verify the Nagumo condition, let $\tilde u\in \MM_\alpha$. Without loss of generality assume that $\tilde c:= h_N(\norm{\tilde u}{\alpha})>0$. Writing $\epsilon' := \tilde c\epsilon$ we get 
  \begin{equation}
    \label{eq:34}
    \frac1{\epsilon} \dist_E( \tilde u  + \epsilon h_N(\norm{\tilde u}{\alpha}) B(\tilde u) ;\MM) %\leq \\
    \leq \frac{\tilde c}{\epsilon'}  \dist_E( \tilde u  + \epsilon' B(\tilde u) ;\MM) \longrightarrow 0,
  \end{equation}
  as $\epsilon \searrow 0$. Hence, the assumptions of Theorem~\ref{thm:viabPruess} are satisfied for the localized equation
  \[ \ddt u = Au + B_N(u(t)),\qquad u(0) = u_0,\]
  and we get for each $N\in \N$ a unique global mild solution, say $u_N$. Set 
  \[ t_N := \inf\{ t\geq 0 \,\vert \norm{u_N(t)}{\alpha} > N\}\wedge T,\]
  then by uniqueness claim of Theorem~\ref{thm:viabPruess} it holds that $u_N = u_{N+1} = u_{N+k}$ on $[0,t_N]$ and $(t_N)_{N\in \N}$ is an increasing sequence. Set $t_\infty := \lim_{N\to \infty} t_N$ and 
  \[u(t) := u_N(t),\qquad \text{on } [0,t_N]\]
  which is well-defined. For $t<t_\infty$, then exists an $N\in \N$ such that $t<t_N$ and thus,
  \[u(t) =u_N(t) =  S_t u_0 + \int_0^t S_{t-s}B(u_N(s)) \d s = S_tu_0 + \int_0^t S_{t-s} B(u(s)) \d s.\]
  In fact, either $t_\infty = \infty$ and $u$ exists even on all of $[0,\infty)$ or $t_\infty <\infty$ and 
  \[\lim_{t\nearrow t_\infty} \norm{u(t)}{\alpha}  = \infty.\qedhere\]
\end{proof}
By approximation, the existence results extend to all initial values in $\MM_\alpha$, instead of $\MM_\gamma$ only:
\begin{cor}\label{cor:startinalpha}
  Assume that Assumption~\ref{a:pruess} holds true but instead of (B)
  assume that there exists an open set $\cO$ in $E_\alpha$ such that
  $\MM_\alpha \subset \cO$ and $B: \cO \to E$ is Lipschitz continuous
  on bounded sets. Then, the statement of Theorem~\ref{thm:viabPruesslocal} also holds
  true with $\gamma = \alpha$. 
\end{cor}
\begin{proof}
  First, by standard existence results \cite[Thms 3.3.3]{henryGeo}
  there exists a unique mild solution $u$ in $E_\alpha$, for all
  initial data $u_0 \in \cO$. By approximation, we show that $u$ stays
  in $\MM$ when started there. Let $\gamma'\in (\alpha, 1)$.

  Given $u_0\in \MM_\alpha$, set $u_0^{(n)} := S_{\frac1{n}} u_0$,
  which converges to $u_0$ in $E_\alpha$. Since $(S_t)$ is analytic
  and $\MM$ is $(S_t)$-invariant it holds that $u_0^n\in \MM_{\gamma'}$ and thus, Theorem~\ref{thm:viabPruesslocal} yields unique maximal mild solutions $u^{n}$ of~\eqref{eq:13} for initial data $u_0^{n}$. We denote the explosion times by $t_\infty^n$ for $u^{n}$ and by $t_\infty$ for $u$. 
  
  By continuity in initial data, it holds that $u^n \to u$ in $E_\alpha$, uniformly on compact subintervals of $[0,t_\infty)$, see \cite[Thm 3.4.1]{henryGeo}. Since $\MM_\alpha$ is closed in $E_\alpha$, this finishes the proof.
\end{proof}

We close this section by proving an easy-to-check condition sufficient for the Nagumo condition~\ref{eq:nagumo}.

\begin{lem}\label{lem:pip}
  Let $I\subset \R$ be a not necessarily bounded interval, $E = L^2(I)$ and consider the closed set 
  \[\MM := \setc{h\in H}{h(\xi) \geq 0 \;\d \xi-\text{a.e.}}.\] 
  Assume that $V\hookrightarrow E$ is an arbitrary Banach space and 
  \[F:V\to E\]
  Lipschitz continuous, satisfying the point-wise inward-pointing property $\d \xi$-a.e.
  \begin{equation}
    \label{eq:pip}
    h\in  \MM_V:=\MM\cap V, \text{ and } \xi\in I: h(\xi) = 0 \Longrightarrow  F(h)(\xi) \geq 0.   %\tag{PIP}
  \end{equation}
  Then, $F$ satisfies the Nagumo condition~\eqref{eq:nagumo}, i.\,e.
  \begin{equation*}
    \lim_{\epsilon\searrow 0} \epsilon^{-1}\dist_E\left(g + \epsilon F(g), \MM\right) = 0,\qquad \forall\, g \in \MM_V.
  \end{equation*}
\end{lem}
\begin{proof}
  Fix $g\in \MM_{V}$, $\epsilon>0$ and set
  \[ \underline{F} := \operatorname{ess}\inf\setc{F(g)(\xi)}{\xi \in I}\geq-\infty\]
  If $\underline{F} \geq 0$ then obviously~\eqref{eq:nagumo} holds true, so it suffices to consider the case $\underline{F} <0$. Define 
  \[ h_\epsilon(\xi):= \max\{ g(\xi) + \epsilon F(g)(\xi), 0\},\;\xi \in I,\]
  which is an element of $\MM$ by definition. Note that~\eqref{eq:pip} implies that for $\d\xi$ almost all $\xi \in I$
  \[ - \epsilon F(g)(\xi) > g(\xi) \Longrightarrow g(\xi) >0.\]
  Here, recall that $g\geq 0$ a.\,e. Moreover,
  \begin{gather*}
    \begin{split}
      \int_I \left(g(\xi) + \epsilon F(g)(\xi) - h_\epsilon(\xi)\right)^2 \d\xi &= \int_I \left(g(\xi) + \epsilon F(g)(\xi) \right)^2 \1_{g(.) +\epsilon F(g)(.) < 0}(\xi)\d \xi\\
      &= \int_I \left(- \epsilon F(g)(\xi)  - g(\xi)\right)^2 \1_{-\epsilon F(g)(.) > g(.) >  0}(\xi)\d \xi\\
      &\leq  \epsilon^2 \int_I F(g)(\xi)^2 \1_{\epsilon F(g)(.) > g(.) >  0}(\xi)\d \xi \\
      &= o(\epsilon^2).
    \end{split}
  \end{gather*}
  The latter estimate holds because~\eqref{eq:pip} yields $\d\xi$-a.e. 
  \[F(g)(\xi) \1_{-\epsilon F(g)(.) > g(.) >  0}(\xi) \longrightarrow 0, \]
  and by dominated convergence theorem the convergence is also true in $L^2$. Indeed, $h_\epsilon$ is the minimal projection of $g+\epsilon F(g)$, $\epsilon>0$, onto $\MM$ in the sense that 
  \[ \dist_E(g+\epsilon F(g); \MM) = \norm{g + \epsilon F(g) - h_\epsilon}{E},\]
  which then finishes the proof.
\end{proof}

%%% Local Variables:
%%% mode: latex
%%% TeX-master: "0paper"
%%% End:

\section[Proofs]{Phase Separation and Approximation: Proofs}
\label{sec:positivityproofs}

We now apply the previous two sections to \eqref{eq:mbp}. Using the notation from Section~\ref{sec:positivityresults}, we will work on the spaces,
\[\L^2:= L^2(\R_+)\oplus \L^2(\R_+) \oplus \R,\quad \H^k(\R_+) := H^k(\R_+) \oplus H^k(\R_+)\oplus \R.\]
In order to reformulate the coupled systems of S(P)DEs~\eqref{eq:SPDE}, we define the coefficients
\begin{align*}
  \cA =& \begin{pmatrix}
    \eta_+ \Delta_+ & 0 & 0 \\ 
    0 & \eta_- \Delta_-  & 0 \\
    0 & 0 &0
  \end{pmatrix} - c\Id, \\
  \cB(u)(x) =& \begin{pmatrix}
    \mu_1(x, u_1(x), u_1'(x))  + \ddx u_1(x) \* \varrho \left(\cI(u)\right) \\
    \mu_2(x, u_2(x), u_2'(x))  - \ddx u_2(x) \* \varrho \left(\cI(u)\right) \\
    \varrho \left(\cI(u)\right) 
  \end{pmatrix}+ c\Id,\\
    \Sigma_n(u)(x) =& \frac12\sum_{k=1}^n
    \begin{pmatrix}
      \ddy \sigma_1(x,u_1(x))\sigma_1(x,u_1(x)) ( T_\zeta e_k(u_3+ x))^2\\
      \ddy \sigma_2(x,u_2(x))\sigma_2(x,u_2(x)) ( T_\zeta e_k(u_3 - x))^2\\
      0
    \end{pmatrix},\\
  \cC(u)(w)(x) =& \begin{pmatrix}
    \sigma_1(x, u_1(x)) T_\zeta w (x_*+x)  \\
    \sigma_2(x, u_2(x)) T_\zeta w(x_* -x)\\
    0
  \end{pmatrix},
\end{align*}
where 
\begin{equation*}
  \cI(u) :=
  \begin{cases}
    (\ddx u_1, \ddx u_2),&\kappa_+ = \kappa_- = \infty,\\
    (u_1, u_2),&\kappa_+, \kappa_-< \infty,
  \end{cases}  
\end{equation*}
for $ u= (u_1,u_2,x_*)\in \H^2,\,w\in U,\,x\geq 0$. Here, we write $\mu_1 = \mu_+$, $\mu_2(x,y,z):= -\mu(-x, y,-z)$, and $\sigma_1 := \sigma_+$, $\sigma_-:= \sigma_-(-x,y)$. $x,y,z\in \R$.

As we will see below, $\Sigma_n$ converges under sufficient assumptions on $\sigma_{1\slash 2}$ and $\zeta$ strongly to 
\begin{equation}
  \label{eq:42}
  \Sigma_\infty(u)(x) = \frac12 \begin{pmatrix}
    \ddy \sigma_1(x,u_1(x))\sigma_1(x,u_1(x)) \norm{\zeta(u_3+x,.)}{L^2(\R)}^2\\
    \ddy \sigma_2(x,u_2(x))\sigma_2(x,u_2(x))\norm{\zeta(u_3 -x,.)}{L^2(\R)}^2\\
    0
  \end{pmatrix},
\end{equation}
The domain of the diagonal operator $\cA$ is then given by
\[ \dom(\cA) = \dom(\Delta_+) \times \dom(\Delta_-) \times \R,\]
where $\Delta_+$ and $\Delta_-$ denote the Laplacian on $\R_+$ with respective boundary conditions defined in~\eqref{eq:bc}. The constant $c>0$ is arbitrary and used to shift the spectrum of $-\cA$ to the positive half-line, so that $-\cA$ is positive self-adjoint. Hence, its fractional powers $(-\cA)^{\alpha}$, $\alpha \in \R$, are well defined and we set $E:= \L^2$. Let us shortly note that $E_\alpha \subset \H^{2\alpha}$ for all $\alpha>0$ and $E_1 = \DA$ with equivalence of norms. 

Writing $X=(u_1,u_2,x_*)$,~\eqref{eq:SPDE} becomes the stochastic evolution equation
\begin{equation}
  \d X(t) = \left[ \cA X(t) + \cB(X(t))\right] \d t + \cC(X(t)) \d W_t, \label{eq:SEEq}
\end{equation}
with initial conditions $X(0) = X_0\in \dom(\cA)$. The approximating equations~\eqref{eq:PDEwza} then become random evolution equations, with initial data $Z_{m,n}(0)  = X_0$, $m$, $n\in \N$, which read (piecewise where $\do \beta_m^k$ is well-defined) as
\begin{equation}
  \ddt Z_{m,n}(t) =  \cA Z_{m,n}(t) + \cB(Z_{m,n}(t))- \Sigma_\infty(Z_{m,n}(t)) + \sum_{k=1}^n (\cC(Z_{m,n}(t)) e_k) \dot \beta_m^k(r). \label{eq:SEEqWZA} 
\end{equation}
\begin{defn}
  A set $\MM\subset \L^2$ is called forward invariant for the stochastic evolution equation~\eqref{eq:SEEq} with initial conditions $X(0) = X_0$, if $X_0\in \MM$ yields $X \in \MM$ on $\llbrak 0,\tau \llbrak$, where $(X,\tau)$ is the unique maximal mild solution of~\eqref{eq:SEEq}.
\end{defn}
\begin{rmk}
  As we have intensively discussed in the previous section, the forward invariance property is defined in the same way for deterministic or random evolution equations. 
\end{rmk}

For the following discussion, set 
\begin{gather*}
  \MM := \{ (u_1,u_2,x) \in \L^2 \,\vert \, u_1 \geq 0,\, u_2 \leq 0\, \text{a.\,e.}\}= L^2_+\times L^2_- \times \R,\\
  L^2_{+} := \{ u\in L^2(\R_+)\,\vert \, u\geq 0\,\text{a.\,e.}\},\quad 
  L^2_{-} := \{ u\in L^2(\R_+)\,\vert \, u\leq 0\,\text{a.\,e.}\}.
\end{gather*}

The following example illustrates why we make the detour using the geometric criterions from Section~\ref{sec:inv}, instead of the direct, infinite-dimensional formulation of ``inward-pointing'' and ``parallel to the boundary'' constraints.
\begin{ex}\label{ex:parallelinfinite}
  Set $E:= L^2(\R_+)$, $A:= \Delta$ with Dirichlet boundary conditions, say, and $C(u)w := \sigma(x,u(x)) T_\zeta w(x)$, where $\sigma$ is as in \autoref{a:sigma}. A classical approach to prove forward invariance for the cone of non-negative functions in $L^2$ is to show 
  \[ \d \norm{ X(t)^{-}}{E}^2 \leq 0,\quad \forall t\geq 0,\]
  for $u^-:= \min\{0, u\}$, see~\cite[Lemma 3.6]{stannatAnalysis} for a mathematically rigorous procedure. Here, $X$ is the solution of the stochastic evolution equation
  \begin{equation}
    \label{eq:19}
    \d X(t) = AX(t) \d t + C(X(t))\d W_t.
  \end{equation}
  At least formally, on that way one ends up with the ``parallel-to-the-boundary'' condition
  \[\langle C(u)w, u^-\rangle = 0,\quad \forall u\in H^1(\R_+).\]  
  With Fubini theorem, this can be rewritten as
  \begin{multline}
    \label{eq:22}
    0 = \int_0^\infty \sigma(x,u(x)) u^-(x) T_\zeta w(x)\d x = \int_\R \int_0^\infty \sigma(x,u(x)) u^-(x) \zeta(x,y) \d x\,  w(y)\d y.
  \end{multline}
  Since the equality has to be true for all $w\in L^2(\R)$, this requires the inner integral to be $0$ for all $y\in \R$, and thus
  \[\int_0^\infty \sigma(x,u(x)) u(x)\1_{u(x) \leq 0} \zeta(x,y) \d y  = 0.\]
  Except for degenerate choices of $\zeta$ and due to differentiability constraints on $\sigma$, this excludes the case $\sigma(x,u(x)) = \sigma\* u(x)$ for a constant $\sigma\in \R$ so that the condition would be too restrictive.
\end{ex}
We now discuss Dirichlet and first order boundary conditions separately. Since we will be able to reuse many calculations for Dirichlet boundary conditions, we start with the first order case.

\subsection{First Order Boundary Conditions}
\label{ssec:proofs1}
Let $\kappa_1$, $\kappa_2<\infty$, then
\begin{equation}
  \label{eq:24}
  \dom(\Delta_+) = \{u\in H^2(\R_+)\,\vert\, \ddx u(0) = \kappa_1 u(0)\},\quad  \dom(\Delta_-) = \{u\in H^2(\R_+)\,\vert\, \ddx u(0) = \kappa_2 u(0)\},
\end{equation}
and $\DA = \dom(\Delta_+) \times \dom(\Delta_-) \times \R$. Recall that up to equivalences of norms $E = \L^2$, $E_1 = \DA$ and
$E_{\frac12} = \H^1$, cf. \cite{grisvardCara}.
\begin{prop}\label{prop:locLipE12}
  Let Assumptions~\ref{a:rho}, \ref{a:zeta},~\ref{a:mu1} and~\ref{a:sigma1} hold true. Then, $\cB:E_{\frac12} \to E$ and $\cC: E_{\frac12} \to \HS(U;E_{\frac12})$ are Lipschitz continuous on bounded sets. 
\end{prop}
\begin{proof}
  This is shown in~\cite[Lemma 3.6 and 3.11.(i)]{SFBP1stOrder}.
\end{proof}
Moreover, we observe from Appendix~\ref{A:noise},~\autoref{thm:Cdiff}, that under Assumptions~\ref{a:zeta} and~\ref{a:sigma1}, $\cC:E_{\sfrac12} \to \HS(U;E_{\sfrac12})$ of class $C^2$ and its derivatives map bounded sets into bounded sets. It remains to show that the $\Sigma_n$ admit Lipschitz constants uniformly in $n\in\N$.
\begin{lem}\label{lem:SigmaLipH1}
  Assume that~\ref{a:zeta} and \ref{a:sigma1} are satisfied. For all $N\in\N$ there exists $L_\Sigma^{(N)}>0$ such that for all $u$, $v\in \H^1$ with norm smaller than $N$ it holds that
  \[\norm{\Sigma_n(u) - \Sigma_n(v)}{\H^1} \leq L_\Sigma^{(N)} \norm{u-v}{\H^1},\quad \forall n\in\N.\]
\end{lem}
\begin{proof}
  First recall that for all $x\in \R$, by Parseval's identity
  \begin{equation}
    \label{eq:parseval}
    \sum_{k=1}^\infty \abs{ T_{\zeta}e_k(x)}^2 = \sum_{k=1}^\infty \abs{\scal{\zeta(x,.)}{e_k}{L^2(\R)}}^2 = \norm{\zeta(x,.)}{L^2}^2.
  \end{equation}
  \autoref{lem:Tc} tells us that $T_\zeta e_k$ takes values in the Banach algebra $BUC^1(\R_+)$ and so does $(T_\zeta e_k)^2$. Hence, with $N_{D\sigma\sigma}(u):= \ddy \sigma(.,u) \sigma(.,u)$,  and \autoref{lem:HSest},
  \begin{multline}
    \label{eq:37}
    \norm{2\Sigma_n^1(u) -2 \Sigma_n^1(v)}{H^1}\leq \\
    \leq K \norm{ N_{D\sigma\sigma}(u) - N_{D\sigma\sigma}(v)}{H^1} \sup_{z\in \R}\sum_{i=0}^1 \sum_{k=1}^n \abs{ T_{\zeta^{(i)}} e_k(z)}^2  \\ 
    +  K\norm{ N_{D\sigma\sigma}(u)}{H^1(\R_+)}\sup_{z\in \R}\sum_{i=0}^1 \sum_{k=1}^n \abs{ T_{\zeta^{(i)}_{u,v}} e_k(z)}^2\leq  \\
    \leq K \norm{ N_{D\sigma\sigma}(u) - N_{D\sigma\sigma}(v)}{H^1}  \sup_{z\in \R}\sum_{i=0}^1 \norm{\zeta^{(i)}(z,.)}{L^2} \\
    +  K\norm{ N_{D\sigma\sigma}(u)}{H^1(\R_+)} \sup_{z\in \R}\sum_{i=0}^1  \norm{\zeta^{(i)}(z-u_3,.) - \zeta^{(i)}(z-v_3,.)}{L^2}^2 
  \end{multline}  
  where we set $\zeta_{u,v}(x,y) := \zeta(x-u_3,y) -\zeta(x-v_3,y)$. Similar to~\eqref{eq:zetaft} we get by application of fundamental theorem of calculus and Fubini theorem
  \begin{multline}
    \label{eq:38}
    \int_\R \abs{\zeta^{(i)}(z-x,y) -\zeta^{(i)}(z-\tilde x,y)}^2 \d y  \\
    \leq \int_\R\int_0^1 \abs{\zeta^{(i+1)}(z-x+\epsilon(x-\tilde x),y)}^2 \abs{x-\tilde x}^2\d y \d\epsilon \\
    \leq  \abs{x-\tilde x}^2 \sup_{z\in \R} \norm{\zeta^{(i+1)}(z,.)}{L^2}^2,
  \end{multline}
  which is finite by \autoref{a:zeta}. Note that, since $\sigma_1$ and $\sigma_2$ fulfill \autoref{a:mugrowthsD} for $m=2$, it holds that $\sigma \ddy\sigma$ fulfills this assumption for $m=1$ and thus, by \autoref{thm:Nlip} the Nemytskii operator $N_{D\sigma\sigma}$ is Lipschitz on bounded sets on $H^1(R_+)$, for $\sigma= \sigma_1$ and $\sigma=\sigma_2$. Hence, the local Lipschitz constants of $\Sigma_n$ depend on $\sigma_1$, $\sigma_2$ and $\zeta$ only, but are particularly independent of $n\in \N$. 
\end{proof}
\begin{lem}\label{lem:SigmaConvH1}
  Let Assumption~\ref{a:zeta} and \ref{a:sigma1} be satisfied. Then, for $\Sigma_\infty$ defined in~\eqref{eq:42}, it holds that $\lim_{n\to \infty} \Sigma_n(u) = \Sigma_\infty(u) \in \H^1$, for all $u\in \H^1$. 
\end{lem}
\begin{proof}
  Again, let $\sigma= \sigma_1$ or $\sigma= \sigma_2$ and set $N_{D\sigma\sigma}(u):=\ddy \sigma(x,u(x))\sigma(x,u(x))$ and fix $u\in H^1(\R_+)$ and $z\in \R$. By Parseval's identity, for all $x\in \R$, 
  \[\sum_{k=1}^n \ddy \sigma(x,u(x))\sigma(x,u(x))  T_\zeta e_k(x-z)^2\longrightarrow \ddy \sigma(x,u(x))\sigma(x,u(x)) \norm{\zeta(x-z,.)}{L^2}^2,\] 
  as $n\to\infty$. Moreover, the sequence is bounded by the square-integrable function
  \[ x\mapsto N_{D\sigma\sigma}(u)(x) \sup_{z\in \R}\norm{\zeta(z,.)}{L^2}^2,\]
  and hence, Lebesgue's dominated convergence theorem yields that $\Sigma_n \to \Sigma_{\infty}$ in $L^2(\R_+)$. The first weak derivative is given by 
  \begin{gather}
    \label{eq:41}
    \begin{split}
          \tfrac{\d}{\d x}  N_{D\sigma\sigma}(u)(x) \sum_{k=1}^n \abs{T_\zeta e_k(x-z)}^2 
          &= \left(\tfrac{\d}{\d x} N_{D\sigma\sigma}(u)(x) \right) \sum_{k=1}^n \abs{T_\zeta e_k(x-z)}^2 \\
          &\quad + 2N_{D\sigma\sigma}(u)(x) \sum_{k=1}^n T_{\zeta'}e_k(x-z) T_{\zeta} e_k(x-z).
    \end{split}
\end{gather}
More detailed computations concerning the derivatives of $T_\zeta$ will be given in \autoref{A:noise}, below. Note that the second summand converges, again by Parseval's identity, to
\[ N_{D\sigma\sigma}(u)(x)\scal{\zeta'(x-z,.)}{\zeta(x-z,.)}{L^2(\R)}.\]
By dominated convergence theorem, the series convergences in $L^2(\R_+)$ and in the same way we can treat the first summand in~\eqref{eq:41}. Summarizing, $\Sigma_n(u)$ is $\H^1$-convergent and thus the limit $\Sigma_\infty(u)$ is an element of $\H^1$, too.
\end{proof}

Collecting the latter two results, the assumptions of \autoref{thm:WZAlocal} are satisfied and we observe for all $r>0$, $\epsilon >0$,
\begin{equation}
  \label{eq:wzconv1storder}
  \lim_{n\to \infty}\lim_{m\to \infty} \EE\sup_{0\leq t\leq T}\norm{X(t\wedge \tau^{(r)}) - Z_{m,n}(t\wedge \tau^{(r+ \epsilon)}_{m,n} \wedge \tau^{(r)})}{\H^1}^{2p} = 0.
\end{equation}
This already finishes the proof of \autoref{thm:wza1}. Part (b) of \autoref{thm:fwdinv} is covered by the following theorem.

\begin{thm}\label{thm:fwdinvH1}
  Assume that Assumptions~\ref{a:rho},~\ref{a:zeta},~\ref{a:mu1},~\ref{a:sigma1} and in addition the pointwise inward pointing assumption~\ref{a:inpoint} hold true. Then, the set $\MM\cap \H^1$ is forward invariant for the stochastic evolution equation~\eqref{eq:SEEq}.
\end{thm}
\begin{proof}
  Let $X= (u_1,u_2,x_*)$ and $Z := (w_1,w_2,y_*)$ be the unique mild solutions respectively of~\eqref{eq:SEEq} and~\eqref{eq:SEEqWZA} on $\H^1$ and write 
  \[\Phi_{m,n}(t,X) := \cB(X) - \Sigma_\infty +  \sum_{k=1}^n \cC(X)e_k \dot \beta_k^m(t).\]
  \begin{enumerate}[label={Step \Roman*.}, wide]
  \item First, we show that the nonlinearity $\Phi_{m,n}$ fulfills the Nagumo condition on $[0,\frac1mT)$.  Note that $\Phi_{m,n}$ is constant in time now and so we write 
    \[ \Phi_{m,n}(t,X) =: (\Phi_{m,n}^1(X), \Phi_{m,n}^2(X), \Phi_{m,n}^3(X))\]
    on $[0,\frac1m T)$. According to \autoref{lem:pip} it now suffices to prove that $\Phi^1_{m,n}$ satisfy the condition~\eqref{eq:pip} - note that the problem for $\Phi^2_{m,n}$ becomes the same after reflection. Let $u\in H^1(\R_+)$ with $u\geq 0$. Recall that there exists a set $G$ such that $\R_+\setminus\{G\}$ is a null set and $u$ is differentiable on $G$, see e.\,g. \cite[Theorem 5.8.5]{evans} . For all $x\in G$ with $u(x) = 0$, it thus holds by non-negativity of $u$ that also $\ddx u(x) = 0$. In particular, it holds for all $x\in G$ with $u(x) = 0$, that
    \[\mu_1(x,u(x),\ddx u(x)) + \varrho(a, b) \ddx u(x) = \mu_1(x,u(x),\ddx u(x)) \geq 0.\]
    Here, $a, b\in \R$ are arbitrary. Even more straight forward, we get that the noise and correction term vanish for all $x\in \R_+$ such that $u(x) = 0$ and thus, $\Phi^{1}_{m,n}(u)\geq 0$. As we have seen in \autoref{lem:pip}, this yields
    \[\lim_{\epsilon \to 0} \frac1{\epsilon} \dist_{L^2}(u_0 + \epsilon \Phi^1_{m,n}(u_0); L^2_+) = 0,\;\forall u_0\in L^2_+,\]
    and the respective result holds true for $\Phi^2_{m,n}$ and $L^2_-$.
  \item We apply \autoref{cor:startinalpha} iteratively on $[\frac{k}{m} T, \frac{k+1}m T)$, as long as the solution is continuable. The uniqueness and maximality claim in \autoref{thm:viabPruesslocal} shows that the unique mild solution of~\eqref{eq:SEEqWZA} stays in $\MM\cap \H^1$ up to the explosion time $\tau_{m,n}$, for all $\omega\in\Omega$, and for all $m$, $n\in\N$.
  \item By~\eqref{eq:wzconv1storder}, we find a subsequence such that the convergence holds true almost surely and thus, $X(t\wedge \tau^{(r)})\in \MM$ almost surely, for all $r>0$, $r\in \Q$. With $r\to\infty$, along the countable set $\Q$, we get that $X(t) \in \MM$ on $\llbrak 0,\tau\llbrak$. 
  \end{enumerate}
\end{proof}

\subsection{Dirichlet Boundary Conditions}
\label{ssec:proofsdirichlet}
We consider the case $\kappa_+=\kappa_+ = \infty$ in which
\[\dom(\Delta_+) = \dom(\Delta_-) = H^2(\R_+)\cap H^1_0(\R_+).\] 
In particular, we have $E=\L^2$ and $E_\alpha = \H^{2\alpha}$, for all
$\alpha<\sfrac14$, see~\cite[Lemma 4.1]{SFBPDir}. From \cite[Lemma
4.3]{SFBPDir} we get that the assumption imposed in the beginning yield Lipschitz continuity on bounded sets. 
\begin{prop}\label{prop:locLipE1Ea}
  Assume that Assumptions~\ref{a:rho},~\ref{a:zeta},~\ref{a:mu}, and \ref{a:sigma} hold true. Then, for $\alpha \in [0,\frac14)$,
  \[\cB:\dom(\cA) \to \H^{2\alpha},\qquad \text{and}\qquad \cC: \DA \to \HS(U; \DA)\]
  are Lipschitz continuous on bounded sets.
  
  In particular, for all initial data $X_0\in \DA$, there exist unique maximal mild solution $(X,\tau)$ and $(Z_{m,n}, \tau_{m,n})$ on $\DA$, resp. of~\eqref{eq:SEEq} and of~\eqref{eq:SEEqWZA}.  Moreover, $X$ and $Z_{m,n}$ have almost surely continuous paths in $\DA$ and $\tau$, $\tau_{m,n}>0$.
\end{prop}
\begin{proof}
  See \cite[Lemma 4.3 and Theorem 3.17]{SFBPDir}.
\end{proof}
The following lemma can be shown exactly in the same way as \autoref{lem:SigmaLipH1} above, but replacing $\H^1$ by $\H^2$. Note that, due to Assumption~\ref{a:sigma}.\ref{ai:sigmabc}, $\Sigma_n(u)$ fulfills Dirichlet boundary conditions at $0$ when $u$ does. 
\begin{lem}\label{lem:SigmaLipH2}
  Let Assumptions~\ref{a:zeta} and \ref{a:sigma} be satisfied. For all $N\in\N$ there exists $L_\Sigma^{(N)}>0$ such that for all $u$, $v\in \H^2$ with norm smaller than $N$ it holds that
  \[\norm{\Sigma_n(u) - \Sigma_n(v)}{\H^2} \leq L_\Sigma^{(N)} \norm{u-v}{\H^2},\quad \forall n\in\N.\]
  Moreover, $\Sigma_n$ maps $\DA$ into $\DA$.
\end{lem}

By \autoref{thm:Cdiff}, $\cC$ is of class $C^2$ and its derivatives map bounded sets into bounded sets. Hence, the definition of $\Sigma_n$ is consistent with the definition in Section~\ref{sec:wza}. 
Applying \autoref{lem:SigmaConvH1} we get $\Sigma_n(u)\rightarrow \Sigma_\infty(u)$ in $\H^1$, for all $u\in \H^1$. With the same arguments, also the second weak derivatives converge in $\L^2$. Indeed, this works iteratively by applying chain rule and the same arguments as in the proof of \autoref{lem:SigmaConvH1} above. We will not go into more details but note that, as we will show in the proof of \autoref{lem:Tc},
\[\tfrac{\d^2}{\d x} (T_\zeta e_k)^2(x) = T_{\zeta''}e_k (x)T_\zeta e_k(x) + (T_\zeta e_k)^2(x).\] 
\begin{lem}\label{lem:SigmaConvH2}
  Let Assumption~\ref{a:zeta} and \ref{a:sigma} be satisfied. Then, for all $u\in \H^2$ it holds that $\lim_{n\to\infty} \norm{\Sigma_\infty(u) - \Sigma_n(u)}{\H^2} = 0$.
\end{lem}
Recall that $\DA$ is a closed subset of $\H^2$ so that $\Sigma_\infty$ maps $\DA$ into $\DA$, since $\Sigma_n$ does for all $n\in\N$. 
Hence, we can apply the \autoref{thm:WZAlocal} to finish the proof of \autoref{thm:wzadirichlet} and obtain,
\begin{equation}
  \label{eq:wzconvdirichlet}
  \lim_{n\to \infty}\lim_{m\to \infty} \EE\sup_{0\leq t\leq T}\norm{X(t\wedge \tau^{(r)}) - Z_{m,n}(t\wedge \tau^{(r+ \epsilon)}_{m,n} \wedge \tau^{(r)})}{\H^2}^{2p} = 0.
\end{equation}
\begin{rmk}
  Technically, we have to choose $\eta\in (0,1)$, we set $\tilde E:= \dom((-\cA)^\eta)$. Then, the restriction of $\cA$ to $\tilde E$ fulfills again \autoref{a:A} and, moreover, $\tilde E_\theta = E_1$, for $\theta:= 1-\eta<1$, so that we fit into the notation of \autoref{sec:wza}; see also \autoref{prop:reiteration} and \autoref{rmk:interpol:reiteration}. 
\end{rmk}

In order to apply the forward invariance results from \autoref{sec:inv} to the approximating solutions $Z_{m,n}$, we need to assure that $\cB$, $\Sigma_k^n$ and $\sigma_k$ are also Lipschitz on bounded sets as mapping from $E_\alpha$ into $E$, for some $\alpha<1$, whereas the Nagumo conditions needs to be satisfied on $E$ itself. Here, as above, we set $E:= \L^2$ so that $E_1= \DA$, $E_\alpha \subset \H^{2\alpha}$, $\alpha\in (0,1)$ and $E_\alpha = H^{2\alpha}$ for all $\alpha \in [0,\sfrac14)$. 
\begin{lem}
  \label{lem:BLipAlpha}
  Assume that Assumptions~\ref{a:rho},~\ref{a:zeta},~\ref{a:mu}, and~\ref{a:sigma} hold true. Then, $\cB$, $\cC(.)e_k$ and $\Sigma_n$, $k\in \N$, $n\in \N$  are Lipschitz continuous on bounded sets from $E_\alpha$ into $E$, for all $\alpha >\sfrac34$. 
\end{lem}
\begin{proof}
  First, note that $\cI$ is continuous from $E_\alpha$ into $\R^2$ for
  all $\alpha >\sfrac34$ by continuity of the trace operator on
  Sobolev spaces, cf. \cite[Thm 9.4]{lionsmagenes1}. We now keep $\alpha\in (\sfrac34,1)$ fixed. Since $\varrho$ is locally Lipschitz and the gradient is Lipschitz from $H^1_0$ into $L^2$, it holds that
  \[u\mapsto \varrho(\cI(u))
  \begin{pmatrix}
    \ddx u_1\\ -\ddx u_2 \\ 1
  \end{pmatrix}\]
  is Lipschitz continuous on bounded sets from $E_\alpha$ into $E$. For $\mu:=\mu_1$ or $\mu:= \mu_2$ we now need to prove that the Nemytskii operator $N_\mu (u):= \mu(.,u(.),\ddx u(.))$ is Lipschitz from $H^{2\alpha}(\R_+)$ into $L^2(\R_+)$. Let $u$, $v\in H^{2\alpha}(\R_+)$ with $\norm{u}{H^{2\alpha}}, \norm{v}{H^{2\alpha}} \leq N_1$ for some $N_1\in \N$. By Sobolev imbeddings, we can assume that $u$, $v\in BUC^1(\R_+)$ and $\norm{u}{C^1_b}$, $\norm{v}{C^1_b}\leq N_2$ for some $N_2\geq N_1$. Denote by $L$ the Lipschitz constant of $(y,z)\mapsto \mu(x,y,z)$ on the $\R^2$-ball of radius $N_2$. Indeed, $L$ can be chosen independently of $x\in \R$ by Assumption~\ref{a:mu}. Then,
  \begin{multline}
    \label{eq:43}
    \int_0^\infty \abs{\mu(x,u(x),\ddx u(x)) - \mu(x,v(x),\ddx v(x))}^2 \d x\\
    \leq L^2 \int_0^\infty \abs{u(x) - v(x)}^2 + \abs{\ddx u(x) - \ddx v(x)}^2 \d x.
  \end{multline}
  Since $E_\alpha \subset \H^{2\alpha}(\R_+)$, this finishes the proof for $\cB$.
  
  For $\cC(.)e_k$, recall that~\autoref{a:sigma} is stronger than~\autoref{a:sigma1} and as a consequence of \autoref{prop:locLipE12}, $\cC(.)e_k$ is Lipschitz continuous on bounded sets on $E_{\sfrac12}$. The same follows from \autoref{lem:SigmaLipH1} for $\Sigma_n$. Finally, recall the imbedding relation
  \[E_{\alpha} \hookrightarrow E_{\sfrac12} \hookrightarrow E,\]
  which yields that $\cC(.)e_n$ and $\Sigma_n$ are also Lipschitz on bounded sets from $E_\alpha$ into $E$, for all $n\in\N$. 
\end{proof}

We close this section by the following theorem, which is the remaining part, (a), of \autoref{thm:fwdinv}.
\begin{thm}\label{thm:fwdinvH2}
  Let Assumption~\ref{a:rho}, \ref{a:zeta}, \ref{a:mu}, \ref{a:sigma} and in addition the pointwise inward pointing assumption~\ref{a:inpoint} hold true. Then, the set $\MM\cap \DA$ is forward invariant for the stochastic evolution equation~\eqref{eq:SEEq}.
\end{thm}

\begin{proof}
  Similar to the proof \autoref{thm:fwdinvH1}, let $X= (u_1,u_2,x_*)$ and $Z_{m,n}$ be the unique mild solutions respectively of~\eqref{eq:SEEq} and~\eqref{eq:SEEqWZA} on $\DA$ and write 
  \[\Phi_{m,n}(t,X) := \cB(X) - \Sigma_\infty +  \sum_{k=1}^n \cC(X)e_k \dot \beta_k^m(t).\]  
  In the same way as in Step I of the proof of \autoref{thm:fwdinvH1}, we get that $\Phi_{m,n}$ fulfills the Nagumo condition on $\L^2$.
  We consider $\Phi_{m,n}$, restricted to each interval $[\frac{k}{m}T, \frac{k+1}{m}T)$, $k<m$, as a map from $E_\alpha$ to $E$, for $\alpha \in (\sfrac34,1)$. Its Lipschitz continuity on bounded is covered by \autoref{lem:BLipAlpha}. Since $(-\cA, \DA)$ is positive self-adjoint, the assumptions of \autoref{thm:viabPruesslocal} yield existence of a unique maximal mild solution $\tilde Z_{m,n}$ on $E_\alpha$ up the the explosion time $\tilde \tau_{m,n}$, taking values in $\MM\cap E_{\alpha}$. When $\tilde\tau_{m,n} > \frac1m T$, we construct the solution on $[\frac1m T,\frac2m T)$ and concatenate the solutions. We iterate this argument as long as we find $k< m$ with $\tilde \tau_{m,n} < \frac{k}{m}T$ or $\tilde \tau_{m,n} = T$. Recall that this works $\omega$-wise. 
  
  On the other hand, by \autoref{prop:locLipE1Ea}, $\Phi_{m,n}$ is also Lipschitz continuous on bounded sets from $E_1$ into $E_{\alpha'}$, for any $\alpha'<\sfrac14$ and thus, there exists a unique maximal mild solution $Z_{m,n}$ of~\eqref{eq:SEEqWZA} on $E_1$, with explosion time $\tau_{m,n}$. By the continuous imbedding $E_1\hookrightarrow E_\alpha$, $Z_{m,n}$ is also a mild solution on $E_\alpha$ and thus, the uniqueness and maximality claim yields $\tilde \tau_{m,n} \geq \tau_{m,n}$ and $Z_{m,n} = \tilde Z_{m,n} \in \MM\cap E_1$ on $[0,\tau_{m,n})$ for all $\omega\in \Omega$. Hence, the set $\MM\cap E_1$ is forward invariant for~\eqref{eq:SEEqWZA}. By switching to subsequences, the convergence~\eqref{eq:wzconvdirichlet} implies $X(t)\in \MM$ on $\llbrak 0,\tau \llbrak$. 
\end{proof}

%%% Local Variables:
%%% mode: latex
%%% TeX-master: "0paper"
%%% End:

%%% Appendix

\begin{appendix}
  \section[Nemytskii Operators]{Nemytskii Operators on Sobolev Spaces}
  \label{A:nem}
  We continue with the analysis on the Sobolev spaces $H^k(\R_+)$, $k\in \N$. In this section we prove some regularity results on the nonlinear Nemytskii operator
\[ u \mapsto N(u)(x) = \mu(x,u(x)),\]
where, $\mu :\R_+ \times \R^d \rightarrow \R$ and $x\in \R_+$. Note that these operators are well-understood but most of the literature focuses on bounded domains, see e.g. \cite{appell, valentPaper, valentBook}. However, in the case of unbounded domains several additional conditions on $\mu$ are necessary to make them work. First, we state a result on the spaces $H^k$ which guarantees, that under certain assumptions on $\mu$, $N$ will map $H^k$ into $H^k$. For a proof we refer to \cite[Theorem 1]{valentPaper}, of which it is a special case.
\begin{lem}\label{lem:mult}
  For each integer $k\geq 1$ the space $H^k(\R_+)$ is a Banach algebra. In particular, there exists a constant $c$ such that for all $u$, $v\in H^k(\R_+)$ it holds that $uv\in H^k(\R_+)$ and
  \[ \norm{uv}{H^k} \leq c \norm{u}{H^k}\norm{v}{H^k}.\]
\end{lem}

\subsection{Continuity}
We now adapt the proof of the continuity result \cite[Theorem 2]{valentPaper} to our setting, but with some corrections in the proof. For notational reasons we also introduce the Nemytskii operators
\begin{equation}
  \label{eq:defNxNy}
  N_x(u)(x) := \left(\ddx \mu\right)(x,u(x)),\quad N_{y_j}(u)(x):=\left(\tfrac{\partial}{\partial y_j} \mu\right)(x,u(x)),\;j=1,...,d,
\end{equation}
for $\quad u\in H^k(\R_+; \R^d),\;x\in \R_+$.
In order for $N$ to map $H^k$ into $H^k$ again, we need certain growth restrictions, which is not the case on bounded domains. For a multiindex $\alpha$ we denote by $D^\alpha$ the respective partial derivative operator. 
\begin{ass}\label{a:mugrowthsD} Assume $\mu\in C^m(\R_{\geq 0}\times \R^d, \R)$ and
  \begin{enumerate}[label=(\alph*)]
  \item \label{hypi:mugrowthsDx} For each integer $l$, $0\leq l\leq m$ there exists an $ a_l\in L^2(\R_+)$ and some $ b_l:\R^d\rightarrow \R_+$ locally bounded, such that 
    \[\abs{D^{(l,0,...,0)} \mu(x,y)} \leq  b_l(y)\left( a_l(x) +  \abs{y}\right),\,\forall\,x\in \R_+,\,y\in \R^d\] 
  \item \label{hypi:mugrowthsDy} For each multiindex $\alpha$ with $\alpha_1 < \abs{\alpha} \leq m$, the family of functions $(D^\alpha \mu(x,.))_{x\in \R_{\geq 0}}$ is equicontinuous and $\sup_{x\in \R_+} \abs{D^\alpha \mu(x,.)}$ is locally bounded.
  \end{enumerate}
\end{ass}
\begin{ass}\label{a:mulip}
  Assume that $\mu\in C^{m}(\R_{\geq 0} \times \R^d, \R)$ and $D^{\alpha}\mu(x,.)$ is locally Lipschitz for all multi-indices $\alpha$, $\abs{\alpha}\leq m$ with Lipschitz constants uniform in $x\in \R_{\geq 0}$, i.\,e. we assume that for all $r\geq 0$ there exists $L_r\geq 0$ such that   \begin{equation*}
    \abs{D^{\alpha}\mu(x,y) - D^{\alpha}\mu(x,z)} \leq L_r \abs{y-z}.
  \end{equation*}
  holds for all $x$, $y$, $z\in \R^d$ with $\abs{y}$, $\abs{z}\leq r$ and $\alpha$, $\abs{\alpha}\leq m$
\end{ass}
\begin{rmk}
  If $\mu$ satisfies \autoref{a:mugrowthsD} for some integer $m\geq 1$, then $\mu$ satisfies \autoref{a:mulip} for $m-1$.
\end{rmk}
\begin{rmk}
  Recall the Sobolev imbeddings
  \[H^{m+1}(\R_+) \hookrightarrow BUC^m(\R_+).\]
  As usual $BUC^m(\R_+)$ is equipped with the $C^m_b$-norm. In the following, we will work with the $BUC^m$ representative of the elements in $H^{m+1}$ without further comment.
\end{rmk}
Note that \autoref{a:mugrowthsD} is stronger than \cite[Assumption 6.2]{SFBPDir} so that we get the following two results from~\cite[Appendix 1]{SFBPDir}.
\begin{thm}\label{thm:Ncont}
  If \autoref{a:mugrowthsD} holds for some integer $m\geq 1$, then the operator $N$ is continuous from $(H^m(\R_+))^d$ into $H^m(\R_+)$.
\end{thm}

\begin{thm}\label{thm:Nlip}
  Let $\mu$ satisfy Assumptions~\ref{a:mugrowthsD} and~\ref{a:mulip} for some positive integer $m$. Then, $N$ is Lipschitz continuous from bounded subsets of $(H^m(\R_+))^d$ into $H^m(\R_+)$.
\end{thm}

%%% Differentiability
\subsection{Differentiability}
\label{ssec:nemdiff}
We now discuss differentiability of $N$ in Fr\'echet sense. Here we run into the following problem compared with the literature. To get continuity of the Fr\'echet derivatives, Valent~\cite{valentBook} uses that $H^m$ is a Banach algebra and the Nemytskii operator corresponding to $(\tfrac{\partial}{\partial y_j} \mu)$ maps into $H^m$. On unbounded domains, this would exclude the linear case $\mu(x,y):= y$ which is of particular interest for applications in this work. We resolve this problem in \autoref{lem:Nuniform}.
Note that multiplication is not only bilinear continuous on $H^k$, but also from $C^k_b\times H^k$ into $H^k$. More precisely, see \autoref{lem:HtimesC}, for all $k\geq 0$ there exists $c>0$ such that for all $g\in C^k_b(\R_+)$, $u\in H^k(\R_+)$ it holds that
\begin{equation}
  \label{eq:multCmHm}
  \norm{gu}{H^k} \leq c \norm{g}{C^k_b} \norm{u}{H^k}.
\end{equation}
We start with a result on continuity of the nonlinear operators, adapting \cite[Theorem 2]{valentPaper}. We now write shortly $H^m(\R_+)^d$ for $H^m(\R_+;\R^d)$. 
\begin{thm}\label{thm:NcontC}
  If \autoref{a:mugrowthsD}.\ref{hypi:mugrowthsDy} holds for some integer $m\geq 1$, then the operator $N_{y_j}$ is continuous from $H^m(\R_+)^d$ into $C^{m-1}_b(\R_+)$ and maps bounded sets into bounded sets. 
\end{thm}
\begin{proof}
  We prove the continuity in a similar way as done for~\autoref{thm:Ncont}. First, let $m=1$, and $(u_n) \subset H^1(\R_+)^d$  converging to some $u\in H^1(\R_+)^d$. Sobolev imbeddings imply that $u_n$, $u \in BUC(\R_+)^d$, $n\in \N$ and $u_n \to u$ uniformly, as $n\to \infty$. Thus, $x\mapsto \frac{\partial}{\partial y_i} \mu(x,u(x))$ is continuous and bounded. 

 Define $R:= \sup_{n\in \N} \norm{u_n}{\infty}<\infty$ and observe that the family of functions
 \[y\mapsto  \tfrac{\partial}{\partial y_j}\mu(x,y),\quad x\in \R_+\]
 is uniformly equicontinuous on the $\R^d$-ball of radius $R$.

  Let $\epsilon >0$ be arbitrary and $\delta= \delta_{\epsilon,R}>0$ such that for all $y$, $\tilde y\in \R^d$ with $\abs{y}$, $\abs{\tilde y} \leq R$ and $\abs{\tilde y-y}<\delta$ it holds that
  \[\sup_{x\in \R_+} \abs{  \frac{\partial}{\partial y_i} \mu(x,y) - \frac{\partial}{\partial y_i}  \mu (x,\tilde y)} <\epsilon. \]
  Now, let $N_\delta \in \N$ such that for all $n\geq N$ it holds that
  \[ \norm{u_n - u}{\infty} <\delta.\]
  Hence, $\sup_{x\in \R_+} \norm{ \frac{\partial}{\partial y_i} \mu(x,u(x)) -  \frac{\partial}{\partial y_i} \mu(x,u_n(x))}{\infty} <\epsilon$ for all $n\geq N_\delta$, and thus
  \[\norm{N_{y_j}(u_n) - N_{y_j} (u)}{\infty} \to 0, \quad n\to\infty.\] 
  Let $M\subset H^1(\R_+)$ be bounded and $R>0$ so, that $M$ is contained in the radius $R$ ball of $C_b(\R_+)$. Then, for all $u\in M$,
  \[  \norm{N_{y_j}(u)}{\infty} \leq \sup_{\abs{y}< R} \sup_{x\in \R_+} \abs{(\tfrac{\partial}{\partial y_j} \mu) \mu(x,y)} <\infty.\]
  We finish the proof by induction, so assume the claim holds true for $m\in \N$. By induction hypothesis, $N_{y_j}$ is continuous from $H^{m+1}$ into $C^{m-1}_b$, so it remains to show that the same holds true for $\frac{\d}{\d x}   N_{y_j}$. Chain rule yields
  \begin{equation}\label{eq:chainrule}
    \tfrac{\d}{\d x} N_{y_j}(u)(x) = \frac{\partial^2}{\partial x \partial y_j} \mu(x,u(x)) + \sum_{i=1}^d \frac{\partial^2}{\partial y_i \partial y_j} \mu(x,u(x)) \nabla u_i(x),  
  \end{equation}
  for $u\in H^{m+1}(\R_+)^d\hookrightarrow BUC^m(\R_+)^d$. The function $\bar \mu$ defined as
  \[\bar \mu(x,y,z) :=  \frac{\partial}{\partial x} \mu(x,y) + \sum_{i=1}^d \frac{\partial}{\partial y_i} \mu(x,y) z_i,\]
  for $x\in \R_{\geq 0}$, $(y,z)\in \R^{d+d}$, satisfies \autoref{a:mugrowthsD}.\ref{hypi:mugrowthsDy} for $m$. Hence, by induction hypothesis, the Nemytskii operators corresponding to the $y_j$ (and $z_i$)-derivatives of $\bar \mu$ are continuous and map bounded sets into bounded sets, from $H^{m}(\R_+)^{d+d}$ into $C^{m-1}_b(\R_{\geq 0})^{d+d}$. Since the map $u\mapsto \nabla u_i$ is linear continuous from $H^{m+1}(\R_+)^d$ into $H^{m}(\R_+)$, we get the properties for $\frac{\d}{\d x}   N_{y_j}$.
\end{proof}

In the following, we write for $j=1,..,d$, $u\in H^m(\R_+)^d$, $v\in H^m(\R_+)$,
\[ \widetilde N_{y_j}(u,v) := N_{y_j}(u) v = (\tfrac{\partial}{\partial y_j} \mu)(.,u(.)) v(.).\]
\begin{lem}
  \label{lem:Nuniform}
  Let $m\geq 1$ and $\mu$ fulfilling \autoref{a:mugrowthsD}.\ref{hypi:mugrowthsDy} for $m+1$, then, the mapping
  \[ \Phi_j: \,u\mapsto \widetilde N_{y_j}(u,.)\]
  is continuous from $H^m(\R_+)^d$ into $\Lbd(H^m(\R_+))$, for all $j= 1,...,d$. Moreover, $\Phi_i$ maps bounded sets into bounded sets.
\end{lem}
\begin{proof}
  Note that $\tilde \mu(x,y,z) := \mu(x,y)z$ fulfills \autoref{a:mugrowthsD}.\ref{hypi:mugrowthsDx} for $m$ so that \autoref{thm:Ncont} yields continuity of
  \[ H^m(\R_+)^{d+1} \ni (u,v)\mapsto \widetilde N_{y_j}(u,v) \in H^m(\R_+).\] 
  Of course, $\widetilde N_{y_j}$ is linear in its second argument so that $\widetilde N_{y_j}(u,.)\in \Lbd(H^m(\R_+))$, for each $u\in H^m(\R_+)^d$. It remains to prove continuity in the uniform operator topology for which we proceed by induction, again. 
  \begin{enumerate}[label= Step \Roman*:,fullwidth]
  \item With $m=1$ let $(u^{(n)})\subset H^1(\R_+)^d$, $u\in H^1(\R_+)^d$ and $v\in H^1(\R_+)$. First note that by \autoref{thm:NcontC}, $N_{y_j}$ is continuous from $H^1$ into $C_b$, so that
    \[ \norm{\widetilde N_{y_j}(u^{(n)},v) - \widetilde N_{y_j}(u,v)}{L^2} \leq c \norm{N_{y_j} (u^{(n)})- N_{y_j}(u^{(n)})}{C_b} \norm{v}{L^2},\]
    converges to $0$, as $n\to\infty$, uniformly in $v$. Similar we get uniform $L^2$-convergence for $N_{y_j}(u^{(n)})\ddx v$ and for the operator
    \[ (u,v) \mapsto (\tfrac{\partial^2}{\partial x\partial y_j} \mu)(x,u(x))v(x).\]
    Moreover, for all $i$, $j\in \{1,...,d\}$, by Sobolev imbeddings
    \begin{multline}
      \label{eq:51}
      \int_{\R_+} \abs{\tfrac{\partial^2}{\partial y_j\partial y_i} \mu (x,u^{(n)}(x)) \nabla u_i^{(n)}(x) - \tfrac{\partial^2}{\partial y_j\partial y_i}\mu(x,u(x))\nabla u_i(x)}^2 \abs{v(x)}^2 \d x \leq \\
      \leq K \norm{v}{H^1}^2 \int_{\R_+} \abs{\tfrac{\partial}{\partial y_j\partial y_i} \mu (x,u^{(n)}(x)) \nabla u_i^{(n)}(x) - \tfrac{\partial}{\partial y_j\partial y_i}\mu(x,u(x))\nabla u_i(x)}^2  \d x .        
    \end{multline}
    Note that $\tfrac{\partial}{\partial y_i}\mu$ fulfills \autoref{a:mugrowthsD}.\ref{hypi:mugrowthsDy} and recall that multiplication is continuous from $C_b\times L^2$ into $L^2$. Hence, the integral converges to $0$, as $n\to\infty$ by application of \autoref{thm:NcontC} to the corresponding Nemytskii operator, and so we conclude continuity of $\Phi$.
    
    Using almost the same estimates and applying the corresponding of part \autoref{thm:NcontC}, we get that $\Phi_i$ maps bounded sets into bounded sets again. 
  \item By induction hypothesis, the Lemma holds true for $m\in \N$ fixed, so assume that $\mu$ fulfills \autoref{a:mugrowthsD}.\ref{hypi:mugrowthsDy} for $m+2$. Then, $\Phi_j$ is continuous from $H^{m+1}(\R_+)^d$ into $\Lbd(H^m(\R_+))$. Let $(u^{(n)})\subset H^{m+1}(\R_+)^d$ converging to $u\in H^{m+1}(\R_+)^d$. Note that
    \begin{multline}
      \label{eq:52}
      \norm{\Phi(u^{(n)}) - \Phi(u)}{\Lbd(H^{m+1})}^2 \leq \\
      \leq \sup_{w\in H^{m}} \frac{\norm{\Phi(u^{(n)})w - \Phi(u)w}{H^m}^2}{\norm{w}{H^m}^2}+ \sup_{v\in H^{m+1}}\frac{\norm{\frac{\d^{m+1}}{\d x^{m+1}} (\Phi(u^{(n)})v - \Phi(u)v)}{L^2}^2}{\norm{v}{H^{m+1}}^2}.
    \end{multline}
    The first term vanishes as $n\to\infty$ by induction hypothesis. For all $v\in H^{m+1}$, we can write the latter one can be estimated by 
    \begin{align*}
      \frac{\d^{m+1}}{\d x^{m+1}} (\Phi(u^{(n)})v - \Phi(u)v) &= \frac{\d^{m}}{\d x^{m}} [(\ddx N_{y_j}(u^{(n)}) - \ddx N_{y_j}(u)) v] \\
                                                          &\qquad + \frac{\d^{m}}{\d x^{m}} [(N_{y_j}(u^{(n)}) - N_{y_j}(u))w],
    \end{align*}
    for $w:= \ddx v \in H^m$. By induction hypothesis, the latter term converges to $0$ in $L^2$, uniformly over all $w\in H^m(\R_+)$. For the first summand, we observe that $D_x \mu(x,y)$ fulfills \autoref{a:mugrowthsD}.\ref{hypi:mugrowthsDy} for $m+1$ so that the induction hypothesis applied on 
    \[ \Psi(u)v := (\tfrac{\partial^2}{\partial x \partial y_j} \mu)(.,u(.)) v\]
    yields $L^2$ convergence. Plugging in into~\eqref{eq:52} finally yields the convergence uniform in $\Lbd(H^{m+1})$. With the same decomposition, we deduce from induction hypothesis that $\Phi_i$ maps bounded sets from $H^{m+1}$ into bounded sets of $\Lbd(H^{m+1})$.\qedhere
  \end{enumerate}
\end{proof}

%%%%%%%%%%%%%%%%%%%%%%%%%%%%%%%%%%%% 

Based on the continuity in the uniform topology, we are now able to prove Fr\'echet differentiability.
\begin{thm}\label{thm:Ndiff}
  If \autoref{a:mugrowthsD} holds for some integer $m+1$, $m\geq 1$, then the operator $N$ defined above is in $C^1(H^m(\R_+)^d, H^m(\R_+))$ with derivative
  \[ DN(u)v = \sum_{j=1}^d  N_{y_j}(u)  v_j,\quad u,v\in H^m(\R_+)^d.\]
\end{thm}
\begin{proof}
  From \autoref{thm:Ncont} we already know that $N$ maps $H^m(\R_+;\R^d)$ continuously into $H^m(\R_+)$. Moreover, \autoref{lem:Nuniform} tells us that $DN$, defined as above, is continuous from $H^m(\R_+)^d$ into $\Lbd(H^m(\R_+))$. Thus, it remains to verify that $DN$ is at least the G\^ateaux derivative of $N$, i.\,e. that for any $u$, $v\in H^{m +1}(\R_+;\R^d)$,
  \begin{equation}
    \frac{1}{\epsilon} \norm{N(u+\epsilon v) - N(u) - \epsilon \sum_{j=1}^d N_{y_j}(u)v_j}{H^m} \longrightarrow 0,\;\quad \text{ as }\epsilon\rightarrow 0.\label{eq:Ndiffqu}
  \end{equation}
  By fundamental theorem of calculus, we get for fixed $u$, $v\in H^{m}$, and all $x\in \R_+$,
  \begin{gather}
    \label{eq:NFundThm}
    \begin{split}
      % \mu(x, u(x) + \epsilon w(x)) - \mu(.,u(x)) 
      N(u+\epsilon v)(x) - N(u)(x)&= \int_0^1 \sum_{j=1}^d \frac{\partial}{\partial y_j}\mu(x,u(x) + t\epsilon v(x)) \epsilon v_j(x)\d t  \\
      &= \epsilon \sum_{j=1}^d \int_0^1 \widetilde N_{y_j}(u+t\epsilon v,v_j)(x) \d t
    \end{split}    
  \end{gather}
  The map $ t\mapsto \widetilde N_{y_i}(u+t\epsilon v, v_i)$ is continuous from $[0,1]$ into $H^{m}(\R_+)$ by \autoref{thm:Ncont}. Therefore, the integral in equation~\eqref{eq:NFundThm} can be considered as an Bochner integral and~\eqref{eq:Ndiffqu} follows from \autoref{lem:Nuniform} and the estimate
  \begin{multline*}
    \norm{N(u+\epsilon v) - N(u) - \epsilon \sum_{j=1}^d N_{y_j}v_j}{H^m}\leq \\ \leq \abs{\epsilon}\sum_{j=1}^d \int_0^1 \norm{\left(N_{y_j}(u+t\epsilon v) - N_{y_j}(u)\right)v_j}{H^m} \d t. \qedhere
  \end{multline*}
\end{proof}
If $\mu \in C^2$, then we define for $u\in H^m(\R_+)^d$, $v$, $w\in H^m(\R_+)$, $x\in \R_+$,
\begin{align*}
  N_{y_i,y_j} (u)(x) &:= \frac{\partial^2}{\partial y_i \partial y_j} \mu(x, u(x)).% , \\
  % \widetilde N_{y_i,y_j} (u,v,w)(x) &:= \frac{\partial^2}{\partial y_i \partial y_j} \mu(x, u(x))v(x)w(x),
\end{align*}

\begin{thm}\label{thm:Ndiff2}
  Assume that $\mu$ satisfies \autoref{a:mugrowthsD} for $m+2$, $m\geq 1$, then the Nemytskii operator $N: H^m(\R_+)^d \to H^m(\R_+)$ is of class $C^2$ with second derivative
  \[ D^2 N(u)[v,w] = \sum_{i=1}^d \sum_{j=1}^d N_{y_i,y_j}(u) v_j w_i ,\] %=  \sum_{i=1}^d \sum_{j=1}^d \widetilde N_{y_i,y_j} (u,v_j,w_i)
  for $u$, $v$, $w\in H^m(\R_+)^d$.
\end{thm}
\begin{proof}
  By the previous theorem, $N$ is of class $C^1$, so we have to show the same for the map 
  \[ DN: H^m(\R_+) \to \Lbd(H^m(\R_+)^d; H^m(\R_+)), \quad DN(u) := \left(v\mapsto \sum_{j=1}^d \widetilde N_{y_j}(u,v_j)\right)\]
  Since $H^m$ is a Banach algebra, we get for $u$, $\bar u \in  H^m(\R_+)^d$, 
  \begin{multline*}
    \norm{D^2N(u) - D^2N(\bar u)}{\Lbd(H^m(\R_+)^d,H^m(\R_+)^d;H^m)}\\
    \leq \sum_{i,j=1}^d\sup_{\norm{v}{}= 1}\sup_{\norm{w}{}= 1}\norm{ \left(N_{y_i,y_j}(u) - N_{y_i,y_j} (\bar u)\right) v_j w_i}{H^m}\\
    \leq c \sum_{i,j=1}^d\sup_{\norm{v}{}= 1}\norm{N_{y_i,y_j}(u)v - N_{y_i,y_j} (\bar u)v}{H^m}.
  \end{multline*}
  Now, we apply \autoref{lem:Nuniform} to $\frac{\partial}{\partial y_i} \mu(x,u(x))$, $i=1,...,d$, which indeed fulfill \autoref{a:mugrowthsD}.\ref{hypi:mugrowthsDy} for $m+1$. This yields continuity of $D^2N$. 

  To finish the proof it again suffices to show differentiability in G\^ateaux sense. Fix $u$, $w\in H^m(\R_+)^d$ and let $\epsilon >0$. As in the proof of \autoref{thm:Ndiff}, cf.~\eqref{eq:NFundThm}, we get by fundamental theorem of calculus, for all $v\in H^m(\R_+)^d$
  \begin{multline*}
    DN(u+\epsilon w)v - DN(u)v - D^2N(u)(v,w) \\
    = \epsilon \sum_{i,j=1}^d \int_0^1 \left(N_{y_i,y_j}(u+t\epsilon w) - N_{y_i,y_j}(u)\right)w_iv_j \d t \\
    = \epsilon \int_0^1 D^2 N(u+t\epsilon w)(v,w) - D^2N(u)(v,w)\d t.
  \end{multline*}
  From the first part of this proof we know that $\epsilon \mapsto D^2N(u+ \epsilon w)$ is uniformly continuous from $[0,1]$ into the space of continuous bilinear operators. Hence, the right hand side is in $o(\epsilon)$, uniformly in $v\in H^m$.
\end{proof}

The following conclusion is a combination of \autoref{thm:NcontC} and the representations of $DN$ and $D^2N$.
\begin{cor}\label{cor:DND2Nbdsets}
  Under the assumptions of respectively \autoref{thm:Ndiff} and~\ref{thm:Ndiff2}, the maps 
  \[DN: H^m(\R_+)^d \to \Lbd(H^m(\R_+)^d;H^m(\R_+))\]
  and
  \[D^2N : H^m(\R_+)^d \to \Lbd(H^m(\R_+)^d, H^m(\R_+)^d;H^m(\R_+)\]
  map bounded sets into bounded sets.
\end{cor}

%%% Local Variables:
%%% mode: latex
%%% TeX-master: "0paper"
%%% End:

  \section{The Noise Operator}
  \label{A:noise}
  We will now study the operator-valued map $\cC$, defined previously by
\[ (\cC(u)w)(x) =
\begin{pmatrix}
  \sigma_+(x,u_1(x)) (T_\zeta w)(u_3+x)\\ \sigma_-(-x,u_2(x)) (T_\zeta w)(u_3-x) \\ 0
\end{pmatrix}
\]
for $u\in \dom(\cC)\subset L^2(\R_+)\oplus L^2(\R_+) \oplus \R$, $w\in L^2(\R)=: U$ and $x\in \R$. We can reduce the problem to the operator
\begin{equation}
  \label{eq:Ccomponent}
  \Psi: (u,x_*)\mapsto \sigma(.,u(.)) T_\zeta (.+x_*)
\end{equation}
for $\sigma: \R^2\to \R$, and $\zeta$ and an integral kernel $\zeta:\R^2\to \R$, which we aim to take values in spaces of Hilbert-Schmidt operators like $\HS(U; L^2(\R_+))$. As above, we write
\[ T_\zeta \colon w\mapsto \int_\R \zeta (x,y)w(y) \d y\]
and define the Nemytskii operator 
\[ N_\sigma\colon u\mapsto \sigma(.,u(.)).\]
Naturally, it will make sense to separate the study of $\Psi$ into the operators $N_\sigma$ and $T_\zeta$. Recall that we have discussed the Nemytskii operators $N_\sigma$ in~\autoref{A:nem}.
\subsection{The Hilbert-Schmidt Property}
\label{Asec:NoiseHS}
Note that on $L^2(D)$, for a domain $D\subset \R^d$, $d\in \N$, every Hilbert-Schmidt operator is of the form $T_\kappa$, for an integral kernel $\kappa$ satisfying
\begin{equation}
  \int_D\int_D \abs{\kappa(x,y)}^2 \d x\d y<\infty,
\end{equation}
see e.\,g. \cite[Section XI.6]{dunfordschwartz2}. When $D$ has
infinite Lebesgue measure, this condition is obviously violated for
convolution kernels $\kappa(x,y) = \kappa(x-y)$, in which have been
interested in Example~\ref{ref:conv} for instance. Hence, $T_\zeta$
itself will typically not be Hilbert-Schmidt on the spaces of
interest. We skip the proofs in the following three lemmas since they
will be the same as the proofs of respectively Lemma 7.1, 7.2 and 7.4
in \cite{SFBPDir}.

\begin{lem}\label{lem:HtimesC}
  For any integer $n\geq 0$, multiplication is bilinear continuous from $H^n(\R_+) \times C_b^n(\R_{\geq 0})$ into $H^n(\R_+)$. 
\end{lem}

The lemma is the first step in the direction to separate our discussion of $\Psi$ into the operators $N_\sigma$ and $T_\zeta$. Provided that $\zeta$ is sufficiently nice, $T_\zeta$ will indeed map into the space of bounded and uniformly continuous functions. 

\begin{ass}
  \label{a:A:zeta}
  Let $n\geq 1$, $\zeta(.,y) \in C^{n+1}(\R)$ for all $y\in \R$ and $\tfrac{\partial^{i}}{\partial x^i}\zeta(x,.)\in L^2(\R)$ for all $x\in \R$, $i\in\{0,\dots, n+1\}$. Moreover, 
  \begin{equation}
    \sup_{x\in \R} \norm{\tfrac{\partial^{i}}{\partial x^i} \zeta(x,.)}{L^2(\R)} <\infty,\quad i=0,1,\ldots ,n+1.\label{eq:A:Azeta}
  \end{equation}
  In the following, we use the notation $\zeta^{(i)}:=\tfrac{\partial^{i}}{\partial x^i} \zeta$.
\end{ass}
\begin{rmk}\label{rmk:conv}
  For convolution kernels $\zeta(x,y):= \zeta(x-y)$ this assumptions is satisfied when $\zeta \in H^{n+1}(\R)\cap C^{n+1}(\R)$.
\end{rmk}

\begin{lem}\label{lem:Tc}
  Let \autoref{a:A:zeta} be fulfilled for $n\in \N$. Then, $T_\zeta$
  maps $U$ into $BUC^n(\R)$. Moreover, $T_\zeta w$ and its first $n$
  derivatives are Lipschitz continuous for all $w\in U$ and it holds
  that $T_\zeta \in \Lbd (U; BUC^2(\R))$.
\end{lem}

\begin{lem}\label{lem:HSest}
  Let $n\in\N$ and \autoref{a:A:zeta} be satisfied. For $u\in H^n(\R_+)$ and $x_*\in \R$ it holds that
  \[\norm{u \* T_\zeta(.+x_*)))}{\HS(U; H^n(\R_+))} \leq K \norm{u}{H^n(\R_+)} \sup_{x\in \R}\sum_{i=0}^n\norm{\zeta^{(i)}(x,.)}{L^2(\R)}\]
\end{lem}
For application in \autoref{sec:positivityproofs} we need to deal $\cC$ on the domain of the Dirichlet Laplacian. In fact, Assumption~\ref{a:sigma} and Lemma~\ref{lem:HtimesC} ensure $N_\sigma(u)\in H^2(\R_+)\cap H^1_0(\R_+)$ for all $u\in H^2(\R_+) \cap H^1_0(\R_+)$.

\subsection{Lipschitz Continuity and Differentiability}

In order to apply the results let us introduce the translation group
$(\theta_x)_{x\in \R}$ which is strongly continuous on $BUC(\R)$. 
\begin{rmk}\label{rmk:Psi2C}
  By the structure of the direct sum of Hilbert spaces, the following two results directly extend to $\cC$ as a mapping from $H^n(\R_+)\oplus H^n(\R_+) \oplus \R$ into $\HS(U; H^n(\R_+) \oplus H^n(\R_+) \oplus\R)$.  
\end{rmk}
\begin{rmk}
  Note that for $x\in \R$
  \[ \theta_x\circ T_{\zeta} =   T_{\zeta_x},\]
  where $\zeta_x := \zeta(x+.,.)$ satisfies Assumption~\ref{a:A:zeta}, whenever $\zeta$ does.
\end{rmk}
We impose the following conditions on $\sigma$.
\begin{ass}\label{a:A:sigma}
  Let $n\geq 1$ and assume that $\sigma \in C^n(\R^2;\R)$ satisfies 
  \begin{enumerate}[label=(\roman*)]
  \item\label{ai:A:sigmagrowths} For every multi-index $I = (i,j) \in \N^2$ with $\abs{I} \leq n$ there exist $a_I\in L^2(\R_+)$ and $b_I \in L^\infty_{loc}(\R, \R_+)$ such that
    \[  \abs{\tfrac{\partial^{\abs{I}}}{\partial x^{i}\partial y^{j}} \sigma(x,y)} \leq  \begin{cases}   b_I(y)\left( a_I(|x|) +\abs{y}\right), & j=0, \\ b_I (y), & j\neq 0. \end{cases}\]
  \item\label{ai:A:sigmalip} $\sigma$ and its partial derivatives (in $x$ and $y$) are locally Lipschitz with Lipschitz constants independent of $x\in \R$.  
  \end{enumerate}
\end{ass}

\begin{thm} \label{thm:Cdiff}
  Let $n\in \N$ and assume that Assumption~\ref{a:A:zeta} is fulfilled for $n+1$ and, respectively,~\ref{a:A:sigma} for $n+2$. Then, $\Psi$ is of class $C^2$ from $H^n(\R_+)\oplus \R$ into $\HS:= \HS(U; H^n(\R_+))$, with derivatives
  \begin{align*}
    D\Psi(u,x)(v,y) &=  DN_\sigma (u)v \mHC \theta_{x}T_\zeta  + y N_\sigma(u)\mHC \theta_{x} T_{\zeta '} \\
                    &=\left( w\mapsto 
                      \ddy \sigma(.,u) v T_\zeta w(.+x) + y \sigma(.,u) T_{\zeta'} w(.+x)\right),\\
    D^2\Psi(u)[(v,y),(\bar v,\bar y)]  &= 
                                         D^2N_\sigma (u)[v,\bar v] \* \theta_{x}T_\zeta  + y DN_\sigma(u)\bar v \* \theta_{x} T_{\zeta '}\\
                    &\qquad 
                      + \bar y DN_\sigma (u)v \mHC \theta_{x}T_{\zeta'} + y\bar y N_\sigma(u)\mHC \theta_{x} T_{\zeta ''}.
  \end{align*}
  Moreover, $\Psi$, $D\Psi$, and $D^2\Psi$ map bounded sets into bounded sets.% and, if $u\in H^2(\R_+)\cap H^1_0(\R_+)$, then also $\Psi(u,x)w \in H^2(\R_+)\cap H^1_0(\R_+)$ for all $x\in \R$, $w\in U = L^2(\R)$.
\end{thm}
\begin{rmk}
  For $n=2$ and under the additional assumption that $\sigma(0,0) = 0$ it holds that $\Psi(u,x)\in H^2\cap H^1_0(\R_+)$, when $u\in H^2\cap H^1_0(\R_+)$. This even translates to $D\Psi$ and $D^2\Psi$.
\end{rmk}
\begin{proof}
  Let $u$, $v\in H^n(\R_+)$, $x$, $y\in \R$ and $\epsilon >0$, then with Lemma~\ref{lem:HSest},
  \begin{multline}
    \label{eq:10}
    \norm{\Psi((u,x)+\epsilon (v,y))-\Psi(u,x) - \epsilon D\Psi(u,x)(v,y)}{\HS(U;H^n)} \\
    \leq K_{\zeta} \norm{N_\sigma(u+v) - N_\sigma(u) - DN_\sigma(u)\epsilon v}{H^n} \\
    + K \norm{N_\sigma(u) + \epsilon DN_\sigma(u)v}{H^n} \sup_{z\in \R} \sum_{i= 0}^n \norm{\zeta^{(i)}_{\epsilon y}(z,.)}{L^2(\R)}
  \end{multline}
  In fact, the first term is in $o(\epsilon)$ because of differentiability of $N_\sigma$ we get from \autoref{thm:Ndiff}. For the second summand, we have defined
  \[ \zeta_{z}(x,y) := \zeta (z + x, y) - \zeta(x,y) - z\ddx\zeta(x,y),\; x,y,z\in \R.\]
  Thus,
  \begin{multline}
    \label{eq:zetaft}
    \sup_{z\in \R} \norm{ \zeta_{\epsilon y}(z,.)}{L^2}^2 = \sup_{z\in\R}\int_\R \abs{\zeta(z+\epsilon y, \xi ) - \zeta(z,\xi) - \epsilon y \zeta'(z,\xi)}^2 \d \xi\\
    \leq \abs{\epsilon y}^2\sup_{z\in\R}\int_\R \int_0^1 \abs{\zeta'(z+\alpha \epsilon y,\xi) - \zeta'(z,\xi)}^2\d\alpha \d \xi.
  \end{multline}
  Using fundamental theorem of calculus again, we obtain 
  \[\sup_{z\in\R}\int_\R \int_0^1 \abs{\zeta'(z+\alpha \epsilon y, \xi) - \zeta'(z,\xi )}^2\d\alpha \d \xi\leq \abs{\epsilon y}^2 \sup_{z\in \R} \norm{\zeta''(z,.)}{L^2},\]
  which goes to $0$, as $\epsilon \to 0$. Using that~\eqref{eq:A:Azeta} holds for $i=0,..,n+2$, the same calculation can be done for $\zeta^{(i)}$, $i=1,\dots, n$ which then shows that $D\Psi$ is at least the G\^ateaux derivative of $\Psi$. To finish the proof, it is now enough to show that 
  \[ D\Psi: H^n\oplus \R \to  \Lbd(H^n\oplus \R;\HS(U;H^n))\]
  is G\^ateaux differentiable, and 
  \[D^2\Psi: H^2\oplus \R \to  \Lbd(H^n\oplus\R, H^n\oplus\R; \HS(U;H^n))\]
  is continuous. Let us start with the latter claim and show continuity of each summand separately. To this end we first decompose as above
  \[ D^2\Psi(u,x)[(v,y),(\bar v,\bar y)] = \sum_{k=1}^4 R_k(u,v,\bar v,x,y,\bar y).\]
  Consider $u$, $\tilde u$, $v$, $\bar v\in H^n$, $x,\tilde x,y,\bar y\in \R$. Because $N_\sigma\in C^2$ by \autoref{thm:Ndiff2}, we get
  \begin{multline}
    \label{eq:14}
    \norm{R_1(u,v,\bar v,x,y,\bar y) - R_1(\tilde u,\tilde x,v,\bar v)}{L_2(U;H^n)} \\
    \leq \norm{\left( D^2N_\sigma(u)[v,\bar v] - D^2N_\sigma( \tilde u)[v,\bar v]\right)\mHC \theta_x(T_\zeta (.))}{\HS} \\
    + \norm{D^2N_\sigma(\tilde u)[v,\bar v] \mHC \left(\theta_x(T_\zeta (.)) - \theta_{\tilde x}(T_\zeta(.))\right)}{\HS}.
  \end{multline}
  Applying \autoref{lem:HSest} we see that both terms go to $0$, as $\norm{u-\tilde u}{H^n} + \abs{x-\tilde x}$ does, and that the convergence is uniformly in $v$, $\bar v\in H^n$ with norm smaller than $1$. Indeed, for the first term this is continuity of $D^2N$, the second term can be estimated by
  \begin{multline}
    \norm{D^2N_\sigma(\tilde u)[v,\bar v] \mHC \left(\theta_x(T_\zeta (.)) - \theta_{\tilde x}(T_\zeta(.))\right)}{\HS(U;H^n)} \\
\leq K \norm{D^2N_\sigma(\tilde u)}{\Lbd(H^n)}\norm{v}{H^n} \sup_{z\in \R} \sum_{i=0}^n \norm{\zeta^{(i)}_{x}(z,.) - \zeta^{(i)}_{\tilde x}(z,.)}{L^2(\R)}.
  \end{multline}
  Convergence of the right hand side follows with the same procedure as in~\eqref{eq:zetaft}. For $R_2$ and $R_3$ we use continuity of $DN_\sigma$, for $R_4$ continuity of $N_\sigma$ itself. With almost the same estimates, we observe that $D^2\Psi$ maps bounded sets into bounded sets. In fact, this property is inherited by $N$, $DN$ and $D^2N$, see~\autoref{cor:DND2Nbdsets}.

  It remains to show that $D^2\Psi$ is indeed the derivative of $D\Psi$. The derivative of the second summand can be computed in the same way as $D\Psi$ itself has been computed. For the first summand, we have to be slightly more careful, but note that by~\autoref{lem:HSest}
  \begin{multline*}
    \sup_{\norm{v}{}\leq 1}\norm{DN_\sigma(u)v \*(\theta_{x+\epsilon y}T_\zeta - \theta_x T_\zeta - \epsilon y T_{\zeta'})}{\HS(U;H^n)} \\ \leq  K \norm{DN_\sigma(u)}{\Lbd(H^n(\R))} \sup_{z\in\R}\sum_{i=0}^n \norm{\zeta^{(i)}_{\epsilon y}(z,.)}{L^2(\R)}
  \end{multline*}
  which is in $o(\epsilon)$ thanks to~\eqref{eq:zetaft}. The remaining estimates follow in the same way: First apply~\autoref{lem:HSest}, but then use that $N_\sigma$ is of class $C^2$. Hence, $D^2\Psi$ is the G\^ateaux derivative of $D\Psi$. By continuity of $D^2\Psi$, the differentiability also holds true in Fr\'echet sense.
\end{proof}

%%% Local Variables:
%%% mode: latex
%%% TeX-master: "0paper"
%%% End:

\end{appendix}

%-------------------------------
%\nocite*
\bibliographystyle{plain}
\bibliography{litPaper0}

\begin{thebibliography}{10}

\bibitem{amannInvariant}
H.~Amann.
\newblock Invariant sets and existence theorems for semilinear parabolic and
  elliptic systems.
\newblock {\em J. Math. Anal. Appl.}, 65(2):432--467, 1978.

\bibitem{appell}
J.~Appell and P.~P. Zabrejko.
\newblock {\em Nonlinear superposition operators}, volume~95 of {\em Cambridge
  Tracts in Mathematics}.
\newblock Cambridge University Press, Cambridge, 1990.

\bibitem{dPKZ}
G.~Da~Prato, S.~Kwapie{\'n}, and J.~Zabczyk.
\newblock Regularity of solutions of linear stochastic equations in {H}ilbert
  spaces.
\newblock {\em Stochastics}, 23(1):1--23, 1987.

\bibitem{dPZstochConv}
G.~{Da Prato} and J.~Zabczyk.
\newblock {A note on stochastic convolution}.
\newblock {\em {Stochastic Analysis and Applications}}, {10}({2}):{143--153},
  {1992}.

\bibitem{dPZinf}
G.~{Da Prato} and J.~Zabczyk.
\newblock {\em {Stochastic Equations in Infinite Dimensions}}.
\newblock {Encyclopedia of Mathematics and its Applications}. Cambridge
  University Press, 1992.

\bibitem{deimling}
K.~Deimling.
\newblock {\em Ordinary differential equations in {B}anach spaces}.
\newblock Lecture Notes in Mathematics, Vol. 596. Springer-Verlag, Berlin-New
  York, 1977.

\bibitem{dunfordschwartz2}
N.~Dunford and J.~T. Schwartz.
\newblock {\em Linear operators. {P}art {II}}.
\newblock Wiley Classics Library. John Wiley \& Sons, Inc., New York, 1988.
\newblock Spectral theory. Selfadjoint operators in Hilbert space, With the
  assistance of William G. Bade and Robert G. Bartle, Reprint of the 1963
  original, A Wiley-Interscience Publication.

\bibitem{engelnagel}
K.-J. Engel and R.~Nagel.
\newblock {\em One-parameter semigroups for linear evolution equations}, volume
  194 of {\em Graduate Texts in Mathematics}.
\newblock Springer-Verlag, New York, 2000.
\newblock With contributions by S. Brendle, M. Campiti, T. Hahn, G. Metafune,
  G. Nickel, D. Pallara, C. Perazzoli, A. Rhandi, S. Romanelli and R.
  Schnaubelt.

\bibitem{evans}
L.~C. Evans.
\newblock {\em Partial differential equations}, volume~19 of {\em Graduate
  Studies in Mathematics}.
\newblock American Mathematical Society, Providence, RI, second edition, 2010.

\bibitem{filipovicHJM}
D.~Filipovi{\'c}, S.~Tappe, and J.~Teichmann.
\newblock Term structure models driven by {W}iener processes and {P}oisson
  measures: existence and positivity.
\newblock {\em SIAM J. Financial Math.}, 1(1):523--554, 2010.

\bibitem{grisvardCara}
P.~Grisvard.
\newblock Caract\'erisation de quelques espaces d'interpolation.
\newblock {\em Arch. Rational Mech. Anal.}, 25:40--63, 1967.

\bibitem{henryGeo}
D.~Henry.
\newblock {\em {Geometric Theory of Semilinear Parabolic Equations}}, volume
  840 of {\em {Lecture Notes in Mathematics}}.
\newblock Springer Berlin Heidelberg, 1981.

\bibitem{ikedaWatanabe}
N.~Ikeda and S.~Watanabe.
\newblock {\em Stochastic differential equations and diffusion processes},
  volume~24 of {\em North-Holland Mathematical Library}.
\newblock North-Holland Publishing Co., Amsterdam; Kodansha, Ltd., Tokyo,
  second edition, 1989.

\bibitem{jachimiakInvariance}
W.~Jachimiak.
\newblock A note on invariance for semilinear differential equations.
\newblock {\em Bull. Polish Acad. Sci. Math.}, 45(2):181--185, 1997.

\bibitem{kallenberg}
O.~Kallenberg.
\newblock {\em Foundations of Modern Probability}.
\newblock Applied probability. Springer, 2002.

\bibitem{SFBPDir}
M.~Keller-Ressel and M.~S. M{\"u}ller.
\newblock {A Stefan-type stochastic moving boundary problem}.
\newblock {\em Stochastics and Partial Differential Equations: Analysis and
  Computations}, 4(4):746--790, 2016.

\bibitem{kunzeApprox}
M.~Kunze and J.~van Neerven.
\newblock Continuous dependence on the coefficients and global existence for
  stochastic reaction diffusion equations.
\newblock {\em J. Differential Equations}, 253(3):1036--1068, 2012.

\bibitem{lionsmagenes1}
J.-L. Lions and E.~Magenes.
\newblock {\em Non-homogeneous boundary value problems and applications - 1}.
\newblock Springer, Springer, 1972.

\bibitem{lunardiAnalytic}
A.~Lunardi.
\newblock {\em {Analytic Semigroups and Optimal Regularity in Parabolic
  Problems}}.
\newblock {Progress in Nonlinear Differential Equations and Their
  Applications}. Birkh{\"a}user Basel, 1995.

\bibitem{lunardiInterpol}
A.~Lunardi.
\newblock {\em Interpolation theory}.
\newblock Appunti. Scuola Normale Superiore di Pisa (Nuova Serie). [Lecture
  Notes. Scuola Normale Superiore di Pisa (New Series)]. Edizioni della
  Normale, Pisa, second edition, 2009.

\bibitem{milianComparison}
A.~Milian.
\newblock Comparison theorems for stochastic evolution equations.
\newblock {\em Stoch. Stoch. Rep.}, 72(1-2):79--108, 2002.

\bibitem{SFBP1stOrder}
M.~S. M{\"u}ller.
\newblock {A stochastic Stefan-type problem under first-order boundary
  conditions}.
\newblock {\em forthcoming in The Annals of Applied Probability}.

\bibitem{diss}
M.~S. M{\"u}ller.
\newblock Semilinear stochastic moving boundary problems.
\newblock Doctoral thesis, TU Dresden, 2016.

\bibitem{nakayamasupport}
T.~Nakayama.
\newblock Support theorem for mild solutions of {SDE}'s in {H}ilbert spaces.
\newblock {\em J. Math. Sci. Univ. Tokyo}, 11(3):245--311, 2004.

\bibitem{nakayamaViab}
T.~Nakayama.
\newblock Viability theorem for {SPDE}'s including {HJM} framework.
\newblock {\em J. Math. Sci. Univ. Tokyo}, 11(3):313--324, 2004.

\bibitem{pavelInvariant}
N.~Pavel.
\newblock Invariant sets for a class of semi-linear equations of evolution.
\newblock {\em Nonlinear Anal.}, 1(2):187--196, 1976/77.

\bibitem{pavel1984differential}
N.~H. Pavel.
\newblock {\em Differential equations, flow invariance and applications},
  volume 113.
\newblock Pitman Pub., 1984.

\bibitem{pazySemigroups}
A.~Pazy.
\newblock {\em {Semigroups of Linear Operators and Applications to Partial
  Differential Equations}}.
\newblock Number~44 in {Applied Mathematical Sciences}. Springer, 1992.

\bibitem{pruessInvariant}
J.~Pr{\"u}ss.
\newblock On semilinear parabolic evolution equations on closed sets.
\newblock {\em J. Math. Anal. Appl.}, 77(2):513--538, 1980.

\bibitem{stannatAnalysis}
M.~Sauer and W.~Stannat.
\newblock Analysis and approximation of stochastic nerve axon equations.
\newblock {\em Mathematics of Computation}, 2016.

\bibitem{stefanEis}
J.~Stefan.
\newblock {{\"U}ber die Theorie der Eisbildung, insbesondere {\"u}ber die
  Eisbildung im Polarmeere.}
\newblock {\em {Wien. Ber. XCVIII, Abt. 2a (965--983)}}, 1888.

\bibitem{stroockvaradhanSupport}
D.~W. Stroock and S.~R.~S. Varadhan.
\newblock On the support of diffusion processes with applications to the strong
  maximum principle.
\newblock In {\em Proceedings of the {S}ixth {B}erkeley {S}ymposium on
  {M}athematical {S}tatistics and {P}robability ({U}niv. {C}alifornia,
  {B}erkeley, {C}alif., 1970/1971), {V}ol. {III}: {P}robability theory}, pages
  333--359. Univ. California Press, Berkeley, Calif., 1972.

\bibitem{zabzcykWZA}
G.~Tessitore and J.~Zabczyk.
\newblock Wong-{Z}akai approximations of stochastic evolution equations.
\newblock {\em J. Evol. Equ.}, 6(4):621--655, 2006.

\bibitem{twardowskaSurvey}
K.~Twardowska.
\newblock Wong-{Z}akai approximations for stochastic differential equations.
\newblock {\em Acta Appl. Math.}, 43(3):317--359, 1996.

\bibitem{valentBook}
T.~Valent.
\newblock {\em Boundary Value Problems of Finite Elasticity: Local Theorems on
  Existence, Uniqueness, and Analytic Dependence on Data}, volume~31 of {\em
  Springer Tracts in Natural Philosophy}.
\newblock Springer New York, 1988.

\bibitem{valentPaper}
Tullio Valent.
\newblock A property of multiplication in {S}obolev spaces. {S}ome
  applications.
\newblock {\em Rend. Sem. Mat. Univ. Padova}, 74:63--73, 1985.

\bibitem{weisEvEqBS}
J.~M. A.~M. van Neerven, M.~C. Veraar, and L.~Weis.
\newblock Stochastic evolution equations in {UMD} {B}anach spaces.
\newblock {\em J. Funct. Anal.}, 255(4):940--993, 2008.

\bibitem{wz}
E.~Wong and M.~Zakai.
\newblock On the relation between ordinary and stochastic differential
  equations.
\newblock {\em Internat. J. Engrg. Sci.}, 3:213--229, 1965.

\bibitem{zabczykFwdInvariance}
J.~Zabczyk.
\newblock Stochastic invariance and consistency of financial models.
\newblock {\em Atti Accad. Naz. Lincei Cl. Sci. Fis. Mat. Natur. Rend. Lincei
  (9) Mat. Appl.}, 11(2):67--80, 2000.

\end{thebibliography}
%-------------------------------
\end{document}